\documentclass[12pt,reqno]{amsart}

\usepackage[utf8]{inputenc}
\usepackage[T1]{fontenc}
\usepackage{lmodern}

\usepackage{amsmath, amsfonts, amssymb, amsthm, amscd, amsbsy}
\usepackage{mathtools}
\usepackage{mathrsfs}

\usepackage[dvipsnames]{xcolor}
\usepackage{enumitem}
\usepackage{graphicx}
\usepackage{tikz}
\usepackage{pgfplots}
\pgfplotsset{compat=1.18}
\usepackage[normalem]{ulem}
\usepackage{cancel}
\usepackage[a4paper]{geometry}
\numberwithin{equation}{section}
\allowdisplaybreaks

\theoremstyle{plain}
\newtheorem{theorem}{Theorem}[section]
\newtheorem{corollary}[theorem]{Corollary}
\newtheorem{lemma}[theorem]{Lemma}
\newtheorem{proposition}[theorem]{Proposition}

\theoremstyle{definition}

\newtheorem{remark}{Remark}[section]

\theoremstyle{plain}
\newtheorem{theorema}{Theorem}

\usepackage[colorlinks,breaklinks,
linkcolor=blue,
citecolor=OliveGreen,
urlcolor=NavyBlue]{hyperref}

\title[Sharp $L^2$-CKN inequalities on orthants]{Sharp $L^2$-Caffarelli--Kohn--Nirenberg and weighted Poincar\'e inequalities on half-spaces and orthants and their stability}

\author[N.~Lam]{Nguyen Lam}
\address{School of Science and the Environment, Grenfell Campus, Memorial University of Newfoundland, Corner Brook, NL A2H 5G4, Canada}
\email{nlam@mun.ca}

\author[Y.~Lodha]{Yukta Lodha}
\address{Department of Mathematics, University of Connecticut, Storrs, CT 06269, USA}
\email{yukta.lodha@uconn.edu}

\author[G.~Lu]{Guozhen Lu}
\address{Department of Mathematics, University of Connecticut, Storrs, CT 06269, USA}
\email{guozhen.lu@uconn.edu}

\author[A.~N.~Sengupta]{Ambar N.~Sengupta}
\address{Department of Mathematics, University of Connecticut, Storrs, CT 06269, USA}
\email{ambar.sengupta@uconn.edu}

\date{\today}
\thanks{}


\newcommand{\abs}[1]{\lvert#1\rvert}

\newcommand{\R}{{\mathbb R}}
\newcommand{\Rnkp}{{\mathbb R}^n_{k,+}}

 \newcommand{\sgn}{\operatorname{sgn}}

\newcommand{\pxi}{{\Big( \prod_{i=n-k+1}^{n} x_i}\Big)}
\newcommand{\pxis}{{\Big( \prod_{i=n-k+1}^{n} x_i^2}\Big)}

\newcommand{\Sn}{\mathbb{S}^{n-1}}

\newcommand{\dive}{\operatorname{div}}

\newcommand{\norm}[1]{\left\lVert #1\right\rVert}
\renewcommand{\abs}[1]{\left| #1\right|}

\begin{document}
\begin{abstract}
Though the sharp $L^{2}$-Caffarelli--Kohn--Nirenberg (CKN) inequalities have been
extensively studied in the entire Euclidean spaces, 
the corresponding problem on domains whose boundary contains the origin remains largely unexplored. 
We investigate the sharp $L^{2}$-CKN inequalities on half-spaces and orthants $\mathbb R^{n}_{k,+}$ by computing explicitly the optimal constants, determining all possible extremal functions, and establishing exact identities for the deficits. Since the singular weights $|x|^{-2b}$ rule out the lifting argument that is available for
the simpler Heisenberg Uncertainty Principle, we develop an approach based on the transformations
$u(x)=|x|^{m}v(x)$ for an appropriately chosen $m$ 
combined with spherical harmonic decompositions and weighted identities. Moreover, we establish
weighted Poincar\'e inequalities 
associated with measures of the form
\[
e^{-\delta|x|^{\tau}}|x|^{\beta}\,\pxis\,dx,
\]
 together with their sharp constants, extremizers
and stability estimates, which substantially extend those of the classical Gaussian Poincar\'e
inequality. On the full orthant, the linear modes cease to be admissible
competitors, since all odd spherical harmonics are annihilated by the lifting; the first
non-radial mode is then of degree two, and both the sharp constant and the manifold of
optimizers change accordingly. Finally, we establish several stability estimates, and
second-order stability estimates, of the CKN inequalities on the half-spaces and orthants throughout the full parameter range.
\end{abstract}
\maketitle
\section{Introduction}
The Caffarelli--Kohn--Nirenberg (CKN) inequalities, introduced in the celebrated paper \cite{CKN}, form a broad family of weighted interpolation inequalities. In their general form, they assert that
\[
\bigl\||x|^{\gamma}u\bigr\|_{L^{r}(\mathbb R^{n})}
\leq C
\bigl\||x|^{\alpha}\nabla u\bigr\|_{L^{p}(\mathbb R^{n})}^{\sigma}
\bigl\||x|^{\beta}u\bigr\|_{L^{q}(\mathbb R^{n})}^{1-\sigma},
\]
where $1\leq p,q<\infty$, $r>0$, $0\leq\sigma\leq1$, and the parameters obey the scaling relation
\[
\frac1r+\frac{\gamma}{n}
=
\sigma\left(\frac1p+\frac{\alpha-1}{n}\right)
+
(1-\sigma)\left(\frac1q+\frac{\beta}{n}\right).
\]
This family contains a number of important inequalities in analysis, among them the Sobolev, Hardy and Gagliardo--Nirenberg inequalities.

In this paper, we concentrate on the $L^{2}$-subclass of the CKN inequalities:
\begin{equation}\label{L2CKN}
\left\|\frac{u}{|x|^{a}}\right\|_{L^{2}(\mathbb R^{n})}
\left\|\frac{\nabla u}{|x|^{b}}\right\|_{L^{2}(\mathbb R^{n})}
\geq C(n,a,b)
\left\|\frac{u}{|x|^{(a+b+1)/2}}\right\|_{L^{2}(\mathbb R^{n})}^{2},
\end{equation}
for all $u\in C_{c}^{\infty}(\mathbb R^{n}\setminus\{0\})$.

This subclass is of particular interest, since it contains several classical inequalities as special cases. Indeed, the choice $a=-1$, $b=0$ gives the Heisenberg uncertainty principle (HUP) \cite{Heis1927, Kennard1927}, whose mathematical theory is surveyed in \cite{Fol97}; the choice $a=1$, $b=0$ gives the Hardy inequality, for which we refer to the monographs \cite{BEL, GM1}; and the choice $a=b=0$ gives the hydrogen uncertainty principle. In this way, the family \eqref{L2CKN} provides a unified framework linking several fundamental inequalities of analysis. Each of these endpoints has generated a literature of its own: improvements and remainder terms for the Hardy inequality \cite{BV, FS08, GR23}, factorization and ground state representation methods \cite{GL2017, GLMP, GLMW}, and extensions to manifolds, to stratified and homogeneous groups, and to operators with monomial or Dunkl-type weights \cite{DLL22, DPH25, Flynn20, FLL21, FLL22, FLL25, FLLM23, FLY24, NVH20, Wang22, YSK15}, as well as magnetic versions \cite{CKLL24, LL23} and Hardy--Rellich counterparts \cite{DDLL25, DoLL24, HT25}.

The study of the sharp constants and of the extremizers in \eqref{L2CKN} has attracted considerable attention. Costa \cite{Cos08} obtained short and elementary proofs for certain subclasses of \eqref{L2CKN}, using integration by parts together with expanding-the-square arguments. Subsequently, Catrina and Costa \cite{CC09} (see also \cite{CFL21}) determined the sharp constants and characterized the extremal functions for the full range of the parameters, by means of variational methods, spherical harmonic decompositions and Kelvin-type transforms. More clearly, the following result was proved in \cite{CC09}:

\begin{theorema}\label{Th:CKNI}
According to the location of the point $(a,b)$ in the plane, we have:
\begin{enumerate}
    \item In the region $\mathcal A$, the best constant $C(n,a,b)$ in the $L^{2}$-CKN inequality \eqref{L2CKN} is
    \[
    C(n,a,b)=\frac{\left|n-(a+b+1)\right|}{2},
    \]
    and it is attained by the functions
    \[
    u(x)=D\exp\!\left(\frac{t|x|^{\,b+1-a}}{b+1-a}\right),
    \]
    with $t<0$ in $\mathcal A_{1}$ and $t>0$ in $\mathcal A_{2}$, where $D$ is a nonzero constant.

    \item In the region $\mathcal B$, the best constant is
    \[
    C(n,a,b)=\frac{\left|n-(3b-a+3)\right|}{2},
    \]
    and it is attained by the functions
    \[
    u(x)=D|x|^{2(b+1)-n}\exp\!\left(\frac{t|x|^{\,b+1-a}}{b+1-a}\right),
    \]
    with $t>0$ in $\mathcal B_{1}$ and $t<0$ in $\mathcal B_{2}$.

    \item The only values of the parameters for which the best constant is not attained are those on the line $a=b+1$, where the best constant is
    \[
    C(n,b+1,b)=\frac{|n-2(b+1)|}{2}.
    \]
\end{enumerate}
\end{theorema}

Here
\[
\begin{aligned}
\mathcal A_{1} &:= \{(a,b)\mid b+1-a>0,\; b\le \tfrac{n-2}{2}\},\\
\mathcal A_{2} &:= \{(a,b)\mid b+1-a<0,\; b\ge \tfrac{n-2}{2}\},\\
\mathcal A &:= \mathcal A_{1}\cup\mathcal A_{2},\\
\mathcal B_{1} &:= \{(a,b)\mid b+1-a<0,\; b\le \tfrac{n-2}{2}\},\\
\mathcal B_{2} &:= \{(a,b)\mid b+1-a>0,\; b\ge \tfrac{n-2}{2}\},\\
\mathcal B &:= \mathcal B_{1}\cup\mathcal B_{2}.
\end{aligned}
\]
In particular, the optimal constants $C(n,a,b)$ in the Hardy inequality, in the HUP, and in the hydrogen uncertainty principle are $\frac{|n-2|}{2}$, $\frac{n}{2}$ and $\frac{n-1}{2}$, respectively; their squares $\frac{(n-2)^{2}}{4}$, $\frac{n^{2}}{4}$ and $\frac{(n-1)^{2}}{4}$ are the constants appearing in the classical formulations, in which the two factors on the left-hand side of \eqref{L2CKN} are multiplied rather than paired with their square roots.

More recently, Cazacu, Flynn, Lam and Lu \cite{CFLL24} investigated identities, inequalities and stability phenomena associated with the Hardy and the $L^{2}$-CKN inequalities. In particular, they established several stability versions of \eqref{L2CKN}, and of the related uncertainty principles, by combining exact remainder identities with weighted Poincar\'e inequalities. Their results show that the corresponding deficits control the distance to the family of extremizers, although explicit sharp stability constants were not available in general.

The situation changes considerably when the origin lies on the boundary of the domain. In this direction, it was shown in \cite{Caz10, Fall12, FM12}, for instance, that the sharp constant in the Hardy inequality
\[
\int_{\Omega}|\nabla u|^{2}\,dx\ \ge\ C\int_{\Omega}\frac{|u|^{2}}{|x|^{2}}\,dx ,
\]
which is the square of the constant in the case $a=1$, $b=0$ of \eqref{L2CKN}, can take any value between $\frac{(n-2)^{2}}{4}$, which is the sharp constant on domains containing the origin, and $\frac{n^{2}}{4}$. Moreover, the optimal constant is attained by nontrivial functions as long as it is strictly smaller than $\frac{n^{2}}{4}$. In particular, when $\Omega$ is the half-space $\mathbb{R}^{n-1}\times\mathbb{R}_{>0}$, the following Hardy inequality has been studied in \cite{FTT09}:
\begin{equation} \label{Hardy_halfspace}
\int_{\mathbb{R}^{n-1}\times\mathbb{R}_{>0}}|\nabla u(x)|^{2}\,dx
\geq
\frac{n^{2}}{4}\int_{\mathbb{R}^{n-1}\times\mathbb{R}_{>0}}\frac{|u(x)|^{2}}{|x|^{2}}\,dx.
\end{equation}
Here $\mathbb{R}_{>0}=(0,\infty)$. The constant $\frac{n^{2}}{4}$ is sharp in \eqref{Hardy_halfspace}, and it cannot be attained by nontrivial functions. Half-spaces have in fact been a recurrent testing ground for sharp weighted inequalities with boundary singularities; see, among others, the factorization approach of \cite{LLZ19, LLZ20}, the trace Hardy inequalities of \cite{NVH19}, and the sharp Hardy--Sobolev--Maz'ya and Hardy--Adams inequalities of \cite{LuYang1, LuYang2, LuYang3}.

The Hardy inequality has also been investigated on orthants. Indeed, in this case, by lifting the orthant $\mathbb{R}_{k,+}^{n}$ to the whole Euclidean space $\mathbb{R}^{n+2k}$, Su and Yang showed in \cite{SuYang2012} that
\begin{equation} \label{Hardy_orthant}
\int_{\mathbb{R}_{k,+}^{n}}|\nabla u(x)|^{2}\,dx
\geq
\frac{(n+2k-2)^{2}}{4}\int_{\mathbb{R}_{k,+}^{n}}\frac{|u(x)|^{2}}{|x|^{2}}\,dx,
\end{equation}
and that the constant $\frac{(n+2k-2)^{2}}{4}$ is sharp. Here, for $0\leq k\leq n$, the orthant is defined by
\[
\mathbb{R}_{k,+}^{n}
:=
\left\{
x=(x_{1},\dots,x_{n})\in\mathbb{R}^{n}:\
x_{n-k+1}>0,\dots,x_{n}>0
\right\}.
\]
In particular, taking $k=0$ we have $\mathbb{R}_{0,+}^{n}=\mathbb{R}^{n}$, and \eqref{Hardy_orthant} reduces to the Hardy inequality on the whole Euclidean space.

Recently, again by a lifting argument, the authors established in \cite{LLLS26} that on orthants the sharp constant in the HUP, that is, the square of the constant in \eqref{L2CKN} with $a=-1$ and $b=0$, improves from $\frac{n^{2}}{4}$ to $\frac{(n+2k)^{2}}{4}$:
\begin{equation}\label{E:HUPorthant}
\int_{\mathbb{R}_{k,+}^{n}}|\nabla u(x)|^{2}\,dx
\int_{\mathbb{R}_{k,+}^{n}}|x|^{2}|u(x)|^{2}\,dx
\geq
\frac{(n+2k)^{2}}{4}
\left(\int_{\mathbb{R}_{k,+}^{n}}|u(x)|^{2}\,dx\right)^{2}.
\end{equation}
Several sharp quantitative stability estimates for \eqref{E:HUPorthant} were also obtained in \cite{LLLS26}.

Motivated by these developments, the first main purpose of this paper is to establish sharp $L^{2}$-CKN inequalities on half-spaces and orthants, together with their optimal constants, their extremizers, and quantitative stability estimates. It is worth noting that in this generality the lifting argument, which is the standard device for the Heisenberg uncertainty principle and for the Hardy inequality on orthants, is no longer available: the singular weight $|x|^{-2b}$ in the general $L^{2}$-CKN framework is not compatible with the lifting. Consequently, new ideas are needed in order to treat the $L^{2}$-CKN inequalities on half-spaces and orthants.

Throughout the paper we work in the weighted energy space
\[
S_{a,b}(\mathbb R^{n}_{k,+})
:=
\overline{
C_{c}^{\infty}(\mathbb R^{n}_{k,+}\setminus\{0\})
}^{\|\cdot\|_{a,b}},
\qquad
\|u\|_{a,b}
:=
\left(
\int_{\mathbb R^{n}_{k,+}}\frac{|u|^{2}}{|x|^{2a}}\,dx
+
\int_{\mathbb R^{n}_{k,+}}\frac{|\nabla u|^{2}}{|x|^{2b}}\,dx
\right)^{1/2}.
\]
Our first principal result can be read as follows.

\begin{theorem}\label{Th:CKNIon_Orthant}
Let $0\leq k\leq n$ and let $a,b\in\mathbb R$. Then for every $u\in S_{a,b}(\mathbb R^{n}_{k,+})$,
\begin{equation}\label{eq:CKN-orthant}
\left(\int_{\mathbb R^{n}_{k,+}}\frac{|\nabla u|^{2}}{|x|^{2b}}\,dx\right)
\left(\int_{\mathbb R^{n}_{k,+}}\frac{|u|^{2}}{|x|^{2a}}\,dx\right)
\ge
C^{2}(n,a,b,k)
\left(\int_{\mathbb R^{n}_{k,+}}\frac{|u|^{2}}{|x|^{a+b+1}}\,dx\right)^{2}.
\end{equation}
The sharp constant is given by
\[
C(n,a,b,k)
=
\frac{\left|-a+b+1\right|+\sqrt{(n-2b-2)^{2}+4k(n+k-2)}}{2},
\]
that is, equivalently,
\[
C(n,a,b,k)
=
\begin{cases}
\displaystyle
\left|
\frac{-a+b+1+\sqrt{(n-2b-2)^{2}+4k(n+k-2)}}{2}
\right|,
& -a+b+1\ge 0,\\[1.2em]
\displaystyle
\left|
\frac{-a+b+1-\sqrt{(n-2b-2)^{2}+4k(n+k-2)}}{2}
\right|,
& -a+b+1\le 0.
\end{cases}
\]
Moreover, when $-a+b+1\ne0$ the extremal functions are exactly the functions of the form
\[
u(x)
=
A\left(\prod_{i=n-k+1}^{n}x_{i}\right)
|x|^{m}
\exp\left(\frac{\lambda}{-a+b+1}|x|^{-a+b+1}\right),
\]
where $A\in\mathbb R$, and where $\lambda<0$ when $-a+b+1>0$ and $\lambda>0$ when $-a+b+1<0$. The exponent $m$ is determined by
\[
m+k
=
\frac{-(n-2b-2)+\operatorname{sgn}(-a+b+1)\sqrt{(n-2b-2)^{2}+4k(n+k-2)}}{2}.
\]
\end{theorem}

\begin{remark}
When $-a+b+1=0$, the constant $C(n,a,b,k)$ is still given by the above formula, and it is still sharp; however, in this case the extremal profile does not belong to the weighted energy space $S_{a,b}(\mathbb R^{n}_{k,+})$, and hence the constant is not attained by nontrivial functions. This is the exact analogue, on orthants, of the line $a=b+1$ in Theorem \ref{Th:CKNI}.
\end{remark}

\begin{remark}
Writing the sharp constant as a single expression, rather than as two branches, makes it apparent that $C(n,a,b,k)$ depends continuously on the parameters across the line $-a+b+1=0$, and that it is always at least $\frac12\sqrt{(n-2b-2)^{2}+4k(n+k-2)}$. Taking $k=0$ we recover Theorem \ref{Th:CKNI}, while taking $a=-1$ and $b=0$ we recover $C(n,-1,0,k)=\frac{n+2k}{2}$, whose square $\frac{(n+2k)^{2}}{4}$ is the sharp constant in \eqref{E:HUPorthant}.
\end{remark}

\begin{remark}\label{rmk:k=n_constant}
The quantity $k(n+k-2)$ appearing under the square root is not an artifact of the computation: it is the eigenvalue
\[
\lambda_{k}^{(n)}=k(n+k-2)
\]
of the Laplace--Beltrami operator on $\mathbb S^{n-1}$ associated with the spherical harmonic
$x_{n-k+1}\cdots x_{n}$, which is precisely the lowest-degree harmonic polynomial that is odd in each of the constrained variables. The sharp constant on $\mathbb R^{n}_{k,+}$ is thus produced by the angular mode dictated by the geometry of the domain, and the extremizers carry that mode as a factor. The gain over the Euclidean case is strictly increasing in $k$, and it is largest on the positive cone
\[
\mathbb R^{n}_{n,+}=(0,\infty)^{n},
\qquad
\lambda_{n}^{(n)}=2n(n-1),
\]
where
\[
C(n,a,b,n)=\frac{|-a+b+1|+\sqrt{(n-2b-2)^{2}+8n(n-1)}}{2}
\]
and the extremal profile involves the full monomial $x_{1}x_{2}\cdots x_{n}$. As we shall see, the case $k=n$ is also the one in which the stability theory departs most sharply from its Euclidean counterpart.
\end{remark}

As explained above, the singular weight $|x|^{-2b}$ in \eqref{eq:CKN-orthant} rules out the possibility of using the lifting method of \cite{LLLS26, SuYang2012}. Indeed, that method rests on the observation that a function on $\mathbb R^{n}_{k,+}$, multiplied by $x_{n-k+1}\cdots x_{n}$ and reflected oddly, becomes a function on $\mathbb R^{n+2k}$ whose Dirichlet energy is a multiple of the original one; the correspondence is exact only for the unweighted energy, and the factor $|x|^{-2b}$ destroys it completely. To overcome this obstacle, we develop an approach based on spherical harmonic decompositions, inspired by \cite{CC09}. Our key device is to introduce the change of unknowns
\[
u(x)=|x|^{m}v(x),
\]
for a suitably chosen exponent $m$, and to perform the spherical harmonic decomposition on $v$ rather than on $u$. This is a substantial departure from \cite{CC09}, where the decomposition is applied to $u$ directly. With the appropriate choice of $m$, the two competing angular contributions are balanced, and the whole parameter range is covered at once; in particular our method yields the desired result while avoiding altogether both the Kelvin-type transform of \cite{CC09} and the lifting argument of \cite{LLLS26}. As a byproduct, we also obtain a new and independent proof of Theorem \ref{Th:CKNI}.

\begin{remark}
Recall that $S_{a,b}(\mathbb{R}_{k,+}^n)$ is defined as the completion of $C_c^\infty\bigl(\mathbb{R}_{k,+}^n\setminus\{0\}\bigr)$ with respect to the $S_{a,b}$-norm. If $k\geq 1$, then
$0\notin\mathbb{R}_{k,+}^n$, and therefore $C_c^\infty\bigl(\mathbb{R}_{k,+}^n\setminus\{0\}\bigr)
=
C_c^\infty\bigl(\mathbb{R}_{k,+}^n\bigr).$ Consequently, in all parameter regions,
\[
S_{a,b}(\mathbb{R}_{k,+}^n)
=
\overline{C_c^\infty(\mathbb{R}_{k,+}^n)}
^{\,\|\cdot\|_{S_{a,b}}},
\]
so no distinction arises on the orthant. When $k=0$, so that the domain is the whole space $\mathbb{R}^n$, the
situation depends on the parameter region. In regions
$\mathrm{A}_1$, $\mathrm{A}_2$, and $\mathrm{B}_1$, we have
\[
S_{a,b}(\mathbb{R}^n)
=
\overline{
\left\{
u\in C_c^\infty(\mathbb{R}^n):
\|u\|_{S_{a,b}}<\infty
\right\}}
^{\,\|\cdot\|_{S_{a,b}}},
\]
and hence $S_{a,b}(\mathbb{R}^n)$ is also the completion of the
admissible functions in $C_c^\infty(\mathbb{R}^n)$. In the region
$\mathrm{B}_2$, however, this identification does not hold in general,
and $S_{a,b}(\mathbb{R}^n)$ remains the completion of
$C_c^\infty(\mathbb{R}^n\setminus\{0\}).$ This distinction is consistent with the change of unknown
$u(x)=|x|^m v(x)$: in the region $\mathrm{B}_2$, one has $m>0$, so
$v(x)=|x|^{-m}u(x)$ may be singular at the origin. The preceding
completion identities are proved in the Appendix { (Proposition \ref{P:Sab_completion})}.
\end{remark}

Once the sharp constant and the full family of optimizers of \eqref{eq:CKN-orthant} have been identified, a natural question is to study their stability. This type of question was first raised by Brezis and Lieb in \cite{BL85}. More clearly, since the optimal constant and the extremal functions for the Sobolev inequality are known explicitly, Brezis and Lieb asked whether the Sobolev inequality can be improved by controlling the difference between its two sides in terms of the distance from the function to the set of extremal functions. This was answered affirmatively by Bianchi and Egnell \cite{BE91}, who proved that there exists a constant $c_{BE}>0$ such that
\begin{equation}\label{StabilityS}
\int_{\mathbb{R}^{n}}|\nabla u|^{2}\,dx
-S_{n}\left(\int_{\mathbb{R}^{n}}|u|^{\frac{2n}{n-2}}\,dx\right)^{\frac{n-2}{n}}
\geq
c_{BE}\inf_{U\in E_{Sob}}\int_{\mathbb{R}^{n}}|\nabla(u-U)|^{2}\,dx.
\end{equation}
Here $S_{n}$ is the sharp Sobolev constant and $E_{Sob}$ is the manifold of the optimizers. In other words, the Sobolev deficit controls the squared distance to the manifold of optimizers in the gradient norm, and this is optimal both in the powers and in the metric involved; see \cite{FN19, FZ22}. The Brezis--Lieb question and the Bianchi--Egnell answer have initiated a deep line of research on the quantitative stability of functional and geometric inequalities. The literature on this topic is by now extensive, and we refer the interested reader to \cite{BWW03, BDNN20, CF13, CLT23, CFW13, CFMP09, FJ1, FJ2, FN19, FZ22, LW99}, for instance.

It is worth noting that the stability constants themselves, and the attainability of the corresponding stability inequalities, have largely been left unexplored. In particular, precise information on the constant $c_{BE}$ in \eqref{StabilityS} was unavailable until very recently. In \cite{DEFFL}, Dolbeault, Esteban, Figalli, Frank and Loss pioneered the rigorous study of $c_{BE}$: they provided optimal lower bounds in the asymptotic regime $n\uparrow\infty$, and they developed a gradient flow method allowing one to pass from local to global stability, with applications to the Gross logarithmic Sobolev inequality \cite{Gro75}. More recently, Chen, Tang and the third author \cite{CLT24} derived explicit lower bounds for the stability of the Hardy--Littlewood--Sobolev inequalities, and deduced corresponding bounds for higher-order and fractional Sobolev inequalities. They further obtained optimal asymptotics as $n\uparrow\infty$ in the range $0<s<\frac{n}{2}$ in \cite{CLT242, CLT243}, which in turn yielded the global stability of the logarithmic Sobolev inequality on the sphere \cite{Beckner93}, sharpening the earlier local results of \cite{CLT23}. In \cite{CLTW}, the optimal stability of the Sobolev inequality on the Heisenberg group was established; since rearrangement techniques are unavailable in that setting, the authors introduced a CR Yamabe flow in order to pass from local to global stability.

Turning to the uncertainty principles, McCurdy and Venkatraman \cite{MV21} applied concentration--compactness methods to the stability of the HUP, and showed that there exist constants $C_{1}>0$ and $C_{2}(n)>0$ such that
\[
\delta_{2}(u)\geq C_{1}\left(\int_{\mathbb{R}^{n}}|u(x)|^{2}\,dx\right)d^{2}(u,E_{HUP})
+C_{2}(n)\,d^{4}(u,E_{HUP}),
\]
where
\[
\delta_{2}(u):=\left(\int_{\mathbb{R}^{n}}|\nabla u(x)|^{2}\,dx\right)
\left(\int_{\mathbb{R}^{n}}|x|^{2}|u(x)|^{2}\,dx\right)
-\frac{n^{2}}{4}\left(\int_{\mathbb{R}^{n}}|u(x)|^{2}\,dx\right)^{2},
\]
\[
E_{HUP}:=\left\{\alpha e^{-\beta|x|^{2}}:\alpha\in\mathbb{R},\ \beta>0\right\},
\qquad
d(u,A):=\inf_{v\in A}\|u-v\|_{2}.
\]
Consequently, a small deficit forces $u$ to be close in $L^{2}(\mathbb{R}^{n})$ to a Gaussian. A simpler, constructive proof, with explicit but non-sharp constants, was later given by Fathi \cite{F21}. A decisive advance was then made in \cite{CFLL24}, where the following sharp stability estimate was obtained.

\begin{theorema}\label{Th:CFLL}
For all admissible $u$, there holds
\[
\delta_{1}(u):=\left(\int_{\mathbb{R}^{n}}|\nabla u(x)|^{2}\,dx\right)^{\frac12}
\left(\int_{\mathbb{R}^{n}}|x|^{2}|u(x)|^{2}\,dx\right)^{\frac12}
-\frac{n}{2}\int_{\mathbb{R}^{n}}|u(x)|^{2}\,dx
\geq d^{2}(u,E_{HUP}).
\]
Moreover, the inequality is sharp, and equality is attained by nontrivial functions $u\notin E_{HUP}$.
\end{theorema}

The core of the approach in \cite{CFLL24} is the combination of an exact identity for the HUP deficit with a sharp Gaussian Poincar\'e inequality. More clearly, the deficit can be rewritten in a form that isolates a weighted Dirichlet energy around an optimal Gaussian profile, which reduces the stability problem to a spectral estimate. This reduction is then quantified by means of the sharp inequality
\begin{equation}\label{keyPoin}
\int_{\mathbb{R}^{N}}|\nabla u|^{2}e^{-\frac{|x|^{2}}{2|\lambda|^{2}}}\,dx
\ge\frac{1}{|\lambda|^{2}}
\inf_{c\in\mathbb{R}}\int_{\mathbb{R}^{N}}|u-c|^{2}e^{-\frac{|x|^{2}}{2|\lambda|^{2}}}\,dx,
\end{equation}
which provides optimal control of the fluctuations around Gaussian states. When $\lambda=1$, \eqref{keyPoin} is the classical Gaussian Poincar\'e inequality, a fundamental tool in the analysis of Gaussian measures; a comprehensive reference is the monograph of Bakry, Gentil and Ledoux \cite{BGL14}, and we also refer the reader to the seminal work of Gross \cite{Gro75}. This identity-plus-Poincar\'e mechanism is what produces both the sharp constant and the correct distance structure.

The same strategy was used in \cite{CFLL24} to prove stability for the $L^{2}$-CKN inequality \eqref{L2CKN} in the parameter range
\[
0\leq b<\frac{n-2}{2},\qquad a\leq\frac{nb}{n-2},\qquad a+b+1=\frac{2bn}{n-2},
\]
where the Gaussian Poincar\'e inequality is replaced by its counterpart for log-concave measures. The complementary range
\[
\frac{n-2}{2}<b\leq n-2,\qquad n(b-a+3)=2(3b-a+3),
\]
was subsequently treated in \cite{DFLL23} by a similar argument. Very recently, a different approach was developed by Do, Lam, Lu and Nguyen in \cite{DLLN26}, based on spherical harmonics, generalized Laguerre expansions and a Kelvin-type transform. It leads to weighted Gaussian $L^{2}$-Poincar\'e inequalities with explicit sharp constants and optimizers, and it yields a complete solution of the stability problem for \eqref{L2CKN} for all admissible parameters; the method extends further to weighted $L^{p}$-Poincar\'e inequalities and to the stability of $L^{p}$-CKN inequalities. Poincar\'e inequalities with monomial weights, and their stability, were investigated in \cite{LLR25} in the Gaussian regime.

Both \cite{DLLN26} and \cite{LLR25}  will be used in this paper, and it is worth
explaining at the outset in what respects the present work differs
from them and why the results here cannot be obtained merely by
transplanting their arguments to a new domain.

The inequalities studied in \cite{DLLN26} are posed on the whole space and
involve radial weights only; the extremal modes of the associated
Poincar\'e inequalities are given by constants, the radial profile
$|x|^\tau$, and the linear modes
$$
\left(\sum_i d_i x_i\right)|x|^{s-1}.
$$
On an orthant, three new phenomena arise simultaneously.

First, the sharp constant in the CKN inequality is no longer of the
Euclidean form: it is governed by the angular eigenvalue
$\lambda_k^{(n)}$ appearing in Remark \ref{rmk:k=n_constant} and cannot be recovered
from the results of \cite{CC09, DLLN26}  by any substitution of parameters.

Second, the measures arising from the deficit identities on the
orthant are no longer radial: they contain the monomial factor $x^A$,
so the relevant spectral problem is that of a weighted Laplacian
with monomial weights rather than a radial weighted operator.

Third, and most importantly, under the lifting procedure the full
collection of modes of the lifted operator is no longer admissible
after descending to the orthant. Only those modes compatible with
the symmetry imposed by the lifting survive, and identifying
precisely which modes descend is itself an essential part of the
analysis.

These are genuine structural obstructions rather than technical
modifications, and overcoming them requires ingredients that have
no counterpart in \cite{DLLN26}.

Motivated by these results, and by our Theorem \ref{Th:CKNIon_Orthant}, our next primary aim is to investigate the stability of the $L^{2}$-CKN inequalities \eqref{eq:CKN-orthant} on half-spaces and orthants. In \cite{LLLS26}, sharp stability estimates for the HUP on such domains were obtained through a lifting argument; here, once again, the singular weight $|x|^{-2b}$ prevents us from proceeding in that way. Instead, we follow the strategy of \cite{CFLL24}, which rests on an exact representation of the deficit, combined with suitable Poincar\'e-type inequalities.

Guided by the explicit form of the extremizers in Theorem \ref{Th:CKNIon_Orthant}, we introduce a transformation under which the deficit is represented as a weighted Dirichlet energy. More precisely, we derive identities expressing the deficit functional as a nonnegative weighted integral of the gradient of the transformed function. These identities give a transparent description of the gap between the two sides of \eqref{eq:CKN-orthant}, and they yield in particular yet another proof of the sharp inequality. More importantly, they are the analytical foundation on which the stability estimates rest.

\begin{theorem} \label{Th:Ckn+_identity}
Let $n\geq1$ and $k\in\{0,1,\ldots,n\}$. Let $u\in S_{a,b}(\mathbb R^{n}_{k,+})$ and assume that $-a+b+1>0$. Let $C(n,a,b,k)$ and $m$ be as in Theorem \ref{Th:CKNIon_Orthant}. Then, for
\[
\lambda=\left(\frac{\int_{\Rnkp}\frac{u^{2}}{|x|^{2a}}dx}
{\int_{\Rnkp}\frac{|\nabla u|^{2}}{|x|^{2b}}dx}\right)^{\tfrac14},
\]
the following identity holds:
\begin{align}\label{eq:CKNO+_identity}
&\left(\int_{\Rnkp}\frac{|\nabla u|^{2}}{|x|^{2b}}dx\right)^{\frac12}
\left(\int_{\Rnkp}\frac{u^{2}}{|x|^{2a}}dx\right)^{\frac12}
-C(n,a,b,k)\int_{\Rnkp}\frac{u^{2}}{|x|^{a+b+1}}dx\\
&\qquad=\frac{\lambda^{2}}{2}\int_{\Rnkp}
\left|\nabla\left(\frac{u(x)}{\pxi|x|^{m}}
e^{\frac{|x|^{-a+b+1}}{\lambda^{2}(-a+b+1)}}\right)\right|^{2}
e^{-2\frac{|x|^{-a+b+1}}{\lambda^{2}(-a+b+1)}}
\frac{\pxis|x|^{2m}}{|x|^{2b}}dx.
\end{align}
\end{theorem}

A corresponding identity holds in the complementary case $-a+b+1<0$, where the extremal profile and the exponential weight are modified according to the second branch in Theorem \ref{Th:CKNIon_Orthant}. More precisely, we have the following.

\begin{theorem} \label{Th:Ckn-_identity}
Let $n\geq1$ and $k\in\{0,1,\ldots,n\}$. Let $u\in S_{a,b}(\mathbb R^{n}_{k,+})$ and assume that $-a+b+1<0$. Let $C(n,a,b,k)$ and $m$ be as in Theorem \ref{Th:CKNIon_Orthant}. Then, for
\[
\lambda
=
\left(
\frac{\int_{\Rnkp}\frac{u^{2}}{|x|^{2a}}\,dx}
{\int_{\Rnkp}\frac{|\nabla u|^{2}}{|x|^{2b}}\,dx}
\right)^{1/4},
\]
we have
\begin{align}\label{eq:CKNO-_identity}
&\left(\int_{\Rnkp}\frac{|\nabla u|^{2}}{|x|^{2b}}\,dx\right)^{1/2}
\left(\int_{\Rnkp}\frac{u^{2}}{|x|^{2a}}\,dx\right)^{1/2}
-C(n,a,b,k)\int_{\Rnkp}\frac{u^{2}}{|x|^{a+b+1}}\,dx
\\
&\qquad=\frac{\lambda^{2}}{2}\int_{\Rnkp}
\left|\nabla\left(\frac{u(x)}{\pxi|x|^{m}}
e^{\frac{-|x|^{-a+b+1}}{\lambda^{2}(-a+b+1)}}\right)\right|^{2}
e^{2\frac{|x|^{-a+b+1}}{\lambda^{2}(-a+b+1)}}
\frac{\pxis|x|^{2m}}{|x|^{2b}}\,dx .\nonumber
\end{align}
\end{theorem}

In the spirit of \cite{CFLL24}, the next step is to establish a Poincar\'e-type inequality adapted to the right-hand sides of \eqref{eq:CKNO+_identity} and \eqref{eq:CKNO-_identity}. These identities lead us naturally to weighted Poincar\'e inequalities associated with measures of the form
\[
e^{-\delta|x|^{\tau}}|x|^{\beta}\,\pxis\,dx .
\]
In this direction, we prove the following weighted Poincar\'e inequality with monomial weights, together with its stability counterpart. Such inequalities are of their independent interest, and they substantially extend the classical Gaussian Poincar\'e inequality, which corresponds to $|A|=0$, $\beta=0$, $\tau=2$ and $\delta=\frac12$.

\begin{theorem}[Weighted Poincar\'e inequality with monomial weights]\label{Th:poincare_monomial_weight}
Let $A=(a_{1},\dots,a_{n})\in\mathbb Z_{\ge0}^{n}$, and set
\[
|A|:=a_{1}+\cdots+a_{n},
\qquad
x^{A}:=\prod_{i=1}^{n}x_{i}^{a_{i}},
\qquad
\mathbb R^{n}_{A,+}:=\{x\in\mathbb R^{n}:\ x_{i}>0\ \text{whenever}\ a_{i}>0\}.
\]
Let $N:=n+|A|\ge2$ and $\tau\in \mathbb{R}$ and $\tau\not =0$, and assume that $\delta>0$ and
\[
\frac{\beta+N+\tau-2}{\tau}>0 .
\]
Then, for every $u\in C_{c}^{\infty}(\mathbb R^{n}_{A,+}\setminus\{0\})$,
\begin{equation}\label{eq:poincare_monomial}
\int_{\mathbb R^{n}_{A,+}}|\nabla u(x)|^{2}\,x^{A}e^{-\delta|x|^{\tau}}|x|^{\beta}\,dx
\ge
C\inf_{d\in\mathbb R}\int_{\mathbb R^{n}_{A,+}}|u(x)-d|^{2}\,
x^{A}e^{-\delta|x|^{\tau}}|x|^{\beta+\tau-2}\,dx,
\end{equation}
with
\[
C=C(n,A,\beta,\tau,\delta)=\begin{cases}
  \min\{\delta\tau^{2},\ \delta\tau s_2\},  &  a_i>0 \ \text{for every } i\in \{1,\dots, n\},\\
  \min\{\delta\tau^{2},\ \delta\tau s_1\}, & a_i=0 \ \text{for at least one } i,
\end{cases}
\]

where
\[
s_1=
\frac{-(\beta+N-2)+\operatorname{sgn}(\tau)\sqrt{(\beta+N-2)^{2}+4(N-1)}}{2}.
\]
and 
\[
s_2=
\frac{-(\beta+N-2)+\operatorname{sgn}(\tau)\sqrt{(\beta+N-2)^{2}+8(N)}}{2}.
\]
Here $\delta\tau^{2}$ is the optimal constant for the radial mode, while $\delta\tau s$ corresponds to the first non-radial spherical harmonic mode in $\mathbb R^{N}$.

\end{theorem}

\begin{remark}\label{rmk:sharpness_monomial}
The two modes producing the constant $C$ behave differently with respect to the monomial weight. Equality in \eqref{eq:poincare_monomial} is attained in the radial mode by $u(x)=a_{0}+a_{1}|x|^{\tau}$, and in the first non-radial mode by $u(x)=a_{0}+\bigl(\sum_{i}d_{i}x_{i}\bigr)|x|^{s_1-1}$. The latter functions are admissible competitors on $\mathbb R^{n}_{A,+}$ only through the indices $i$ with $a_{i}=0$; {this is why, when $\delta\tau s_1<\delta\tau^{2}$ and $a_{i}>0$ for every $i$, the value $\min\{\delta\tau^{2},\delta\tau s_1\}$ remains a valid constant in \eqref{eq:poincare_monomial}, but need not be optimal. In that degenerate case the first admissible non-radial mode is of degree two, and the optimal constant is exactly $\min\{\delta\tau^{2},\delta\tau s_{2}\}$, as proved in the Equality Cases corollary of Section \ref{S:poincare}}. In the application to the orthant this degenerate case occurs precisely when $k=n$; see {the cases $k=n$ of Theorem \ref{Th:second_stability_ckn}}.
\end{remark}

\begin{remark}\label{rmk:parity_k=n}
It is worth isolating the mechanism behind the dichotomy in Theorem \ref{Th:poincare_monomial_weight}, since it is the point at which the orthant problem ceases to be a weighted version of the Euclidean one. Under the lifting, a function on $\mathbb R^{n}_{A,+}$ corresponds to a function on $\mathbb R^{N}$, $N=n+|A|$, that is radial in each of the blocks $\mathbb R^{a_{i}+1}$ associated with the indices $i$ with $a_{i}>0$. Whenever $a_{i}>0$ for every $i$, such a function is invariant under the antipodal map $y\mapsto-y$, because $-I\in O(a_{i}+1)$ acts trivially on block-radial functions. Since the antipodal map multiplies a spherical harmonic of degree $\ell$ by $(-1)^{\ell}$, all odd-degree modes are annihilated, and in particular the linear modes $\bigl(\sum_{i}d_{i}x_{i}\bigr)|x|^{s_{1}-1}$ --- which realize the sharp constant when some $a_{i}$ vanishes --- are simply not competitors. The first surviving non-radial mode has degree two, the relevant eigenvalue is $\lambda_{2}$ instead of $\lambda_{1}$, and the optimal constant improves from $\min\{\delta\tau^{2},\delta\tau s_{1}\}$ to $\min\{\delta\tau^{2},\delta\tau s_{2}\}$. Correspondingly, the optimizer manifold is no longer spanned by linear profiles but by the quadratic family
\[
\Bigl(\sum_{i=1}^{n}c_{i}x_{i}^{2}\Bigr)|x|^{s_{2}-2},
\qquad
\sum_{i=1}^{n}c_{i}(a_{i}+1)=0,
\]
of dimension $n-1$, the constraint being exactly the orthogonality to the constants forced by the weight. In the application to the CKN inequality one has $A=(0,\dots,0,2,\dots,2)$ with $k$ twos, so that the degenerate case is $k=n$, the constraint reduces to $\sum_{i}c_{i}=0$, and the lifted dimension is $N=3n$.
\end{remark}

The extremizers depend on which mode realizes the minimum in the sharp constant: equality may occur in the radial mode, in the first non-radial mode, or in both simultaneously when the two constants coincide. This trichotomy also governs the shape of the corresponding stability estimate, which reads as follows.

\begin{theorem}\label{Th:stability_poincare_monomialweight}
Let $A$, $x^{A}$, $|A|$, $N=n+|A|$ and $\mathbb R^{n}_{A,+}$ be as in Theorem \ref{Th:poincare_monomial_weight}, and let $\tau\in \mathbb{R}$ and $\tau\not =0$, and assume that $\delta>0$ and $\frac{\beta+N+\tau-2}{\tau}>0$. Set
\[
d\mu_{A}=x^{A}e^{-\delta|x|^{\tau}}|x|^{\beta}\,dx,
\qquad
d\nu_{A}=x^{A}e^{-\delta|x|^{\tau}}|x|^{\beta+\tau-2}\,dx .
\]
Then, for every $u\in C_{c}^{\infty}(\mathbb R^{n}_{A,+}{\setminus\{0\}})$, the following stability estimates hold

\medskip

\noindent
{(i) If either} $a_i=0$ for at least one  $i$, 
$|\tau|<|s_1|$, {or} $a_i>0$ for every  $i$, $|\tau|<|s_2|$, { then}

\begin{align}
&
\int_{\mathbb R^{n}_{A,+}}|\nabla u|^{2}\,d\mu_{A}
-\delta\tau^{2}\inf_{d\in\R}\int_{\mathbb R^{n}_{A,+}}|u-d|^{2}\,d\nu_{A}
\\
&\qquad\qquad
\ge
{\min\{\delta\tau^{2},\ \delta\tau(s_q-\tau)\}}
\inf_{a_{0},a_{1}\in\mathbb R}
\int_{\mathbb R^{n}_{A,+}}
\bigl|u-a_{0}-a_{1}|x|^{\tau}\bigr|^{2}\,d\nu_{A}\\
&\text{where } s_q=\begin{cases} s_1, & a_i=0 \text{ for some } i,\\ s_2, & a_i>0 \text{ for every } i.\end{cases}
\end{align}

\medskip

\noindent
{(ii) If $\tau s_1<\tau^{2}$ and  $a_i=0$ for at least one  $i$, then}
\begin{align}
&
\int_{\mathbb R^{n}_{A,+}}|\nabla u|^{2}\,d\mu_{A}
-\delta\tau s_1\inf_{d\in\R}\int_{\mathbb R^{n}_{A,+}}|u-d|^{2}\,d\nu_{A}
\\
&\qquad\qquad
\ge
\min\{\delta\tau(\tau-s_1),\ \delta\tau^{2},\ \delta\tau\alpha_1\}
\inf_{a\in\mathbb R,\ d_{i}\in\mathbb R}
\int_{\mathbb R^{n}_{A,+}}
\left|
u-a-\Bigl(\sum_{i:\,a_{i}=0}d_{i}x_{i}\Bigr)|x|^{s_1-1}
\right|^{2}\,d\nu_{A}.
\end{align}

\medskip

\noindent
{(iii) If $\tau=s_1$ and $a_i=0$ for {at least} one $i$, then}

\begin{align}
&
\int_{\mathbb R^{n}_{A,+}}|\nabla u|^{2}\,d\mu_{A}
-\delta\tau^{2}\inf_{d\in\R}\int_{\mathbb R^{n}_{A,+}}|u-d|^{2}\,d\nu_{A}
\\
&\qquad\qquad
\ge
\min\{\delta\tau^{2},\ \delta\tau\alpha_1\}
\inf_{a_{0},a_{1}\in\mathbb R,\ d_{i}\in\mathbb R}
\int_{\mathbb R^{n}_{A,+}}
\left|
u-a_{0}-a_{1}|x|^{\tau}
-\Bigl(\sum_{i:\,a_{i}=0}d_{i}x_{i}\Bigr)|x|^{s_1-1}
\right|^{2}\,d\nu_{A}.
\end{align}

{\item[(iv)] If $\tau s_2<\tau^2,$ and $a_i>0$ for every  $i$, then}

\begin{align}
&\int_{\mathbb R^n_{A,+}}|\nabla u|^2\,d\mu_A
-
\delta\tau s_2
\inf_{d\in\mathbb R}
\int_{\mathbb R^n_{A,+}}|u-d|^2\,d\nu_A
\\
&\qquad\geq
\min\left\{
\delta\tau(\tau-s_2),\,
\delta\tau^2,\,
\delta\tau\alpha_2
\right\}
\inf_{\substack{
b\in\mathbb R,\ c_1,\ldots,c_n\in\mathbb R\\
\sum_{i=1}^n c_i(a_i+1)=0
}}
\int_{\mathbb R^n_{A,+}}
\left|
u-b
-
|x|^{s_2-2}\sum_{i=1}^n c_i x_i^2
\right|^2
\,d\nu_A .
\end{align}

{\item[(v)] If $\tau=s_2,$ and $a_i>0$ for every  $i$, then}

\begin{align}
&\int_{\mathbb R^n_{A,+}}|\nabla u|^2\,d\mu_A
-
\delta\tau^2
\inf_{d\in\mathbb R}
\int_{\mathbb R^n_{A,+}}|u-d|^2\,d\nu_A
\\
&\qquad\geq
\min\left\{
\delta\tau^2,\,
\delta\tau\alpha_2
\right\}
\inf_{\substack{
b_0,b_1\in\mathbb R,\ c_1,\ldots,c_n\in\mathbb R\\
\sum_{i=1}^n c_i(a_i+1)=0
}}
\int_{\mathbb R^n_{A,+}}
\left|
u-b_0-b_1|x|^\tau
-
|x|^{s_2-2}\sum_{i=1}^n c_i x_i^2
\right|^2
\,d\nu_A .
\end{align}

Here

$$s_1=\frac{-(\beta+N-2)+\operatorname{sgn}(\tau)\sqrt{(\beta+N-2)^{2}+4(N-1)}}{2},$$

$$\alpha_1= \operatorname{sgn}(\tau)\frac{\sqrt{(\beta+N-2)^{2}+8N} -\sqrt{(\beta+N-2)^{2}+4(N-1)}}{2},$$

$$ s_2=\frac{-(\beta+N-2)+\operatorname{sgn}(\tau)\sqrt{(\beta+N-2)^2+8N}}{2},$$
and
$$\alpha_2=\operatorname{sgn}(\tau)\frac{\sqrt{(\beta+N-2)^2+16(N+2)}-\sqrt{(\beta+N-2)^2+8N}}{2}.$$

\end{theorem}

In order to prove Theorems \ref{Th:poincare_monomial_weight} and \ref{Th:stability_poincare_monomialweight}, we use the lifting method, which for these Poincar\'e inequalities --- unlike for the CKN inequalities themselves --- is available, since the weight $x^{A}$ is precisely the one produced by the reflection. Consequently, we are led to study weighted Poincar\'e inequalities associated with the radial measures $e^{-\delta|x|^{\tau}}|x|^{\beta}\,dx$, that is, the case $|A|=0$. Here, in the spirit of \cite{DLLN26}, we introduce the transformation $u(x)=|x|^{s}v(x)$ for an appropriate choice of $s$, and decompose $v$ into spherical harmonics; this reduces the radial component to a one-dimensional weighted problem. As in \cite{DLLN26}, we then analyze the reduced problem by means of Laguerre polynomial expansions, which arise naturally in the identification of the sharp constants and of the extremizers, and which provide a systematic way of separating the radial from the non-radial contributions.

The lifting, however, only reduces the monomial problem to a radial one; it does not by itself identify the sharp constant, because the inequality on $\mathbb R^{N}$ is tested against \emph{all} competitors, whereas only those inherited from $\mathbb R^{n}_{A,+}$ are relevant. The additional work consists in describing the class of functions that descend --- they are exactly the functions invariant under the group $\prod_{i:\,a_{i}>0}O(a_{i}+1)$ --- in computing the spherical harmonics that survive this invariance, and in checking that the constant so obtained is still attained. This is where the parity phenomenon of Remark \ref{rmk:parity_k=n} enters, and it is also where the analysis is genuinely more delicate than in the radial setting of \cite{DLLN26}: the sharp constant is not the bottom of the spectrum of the lifted operator, but the bottom of its restriction to an invariant subspace, and the two differ precisely when the weight is nondegenerate in every variable. A related, though different, use of monomial weights in Gaussian Poincar\'e inequalities appears in \cite{LLR25}.

Applying the above Poincaré inequalities to the identities established
earlier, we obtain the following stability estimate for the
\(L^2\)-CKN inequalities on orthants.

To state the following results concisely, we introduce some notation.
Let
\[
\lambda_k^{(n)}:=k(n+k-2)
\]
be the eigenvalue of the Laplace--Beltrami operator on
\(\mathbb S^{n-1}\) corresponding to spherical harmonics of degree \(k\),
and set
\[
\Theta:=\sqrt{(n-2b-2)^2+4\lambda_k^{(n)}}.
\]
Furthermore, let \(N:=n+2k\) be the dimension arising from the lifting
argument and define
\[
\lambda_\ell^{(N)}:=\ell(\ell+N-2),
\]
the eigenvalue of the Laplace--Beltrami operator on
\(\mathbb S^{N-1}\) corresponding to spherical harmonics of degree
\(\ell\). We then set
\begin{equation}\label{eq:s1s2}
s_q:=\frac{\operatorname{sgn}(\tau)}{2}
\left(\sqrt{\Theta^2+4\lambda_q^{(N)}}-\Theta\right),
\qquad q=1,2,
\end{equation}
and
\begin{equation} \label{eq:alpha1alpha2}
\alpha_1:=\frac{\operatorname{sgn}(\tau)}{2}
\left(\sqrt{\Theta^2+4\lambda_2^{(N)}}
-\sqrt{\Theta^2+4\lambda_1^{(N)}}\right),
\qquad
\alpha_2:=\frac{\operatorname{sgn}(\tau)}{2}
\left(\sqrt{\Theta^2+4\lambda_4^{(N)}}
-\sqrt{\Theta^2+4\lambda_2^{(N)}}\right).
\end{equation}

\begin{theorem}[Stability of the $L^{2}$-Caffarelli--Kohn--Nirenberg inequality]\label{Th:stability_ckn}
Let $u\in S_{a,b}(\mathbb{R}^{n}_{k,+})$, assume that $\tau:=-a+b+1\neq0$, and let $C(n,a,b,k)$ and $m$ be as in Theorem \ref{Th:CKNIon_Orthant}. Then

\begin{align}
&
\left(\int_{\mathbb{R}^{n}_{k,+}}\frac{|\nabla u|^{2}}{|x|^{2b}}\,dx\right)^{\frac12}
\left(\int_{\mathbb{R}^{n}_{k,+}}\frac{u^{2}}{|x|^{2a}}\,dx\right)^{\frac12}
-C(n,a,b,k)\int_{\mathbb{R}^{n}_{k,+}}\frac{u^{2}}{|x|^{a+b+1}}\,dx
\\
&\qquad\qquad
\ge
C_2(n,a,b,k)
\inf_{A\in\mathbb R,\ B>0}
\int_{\mathbb{R}^{n}_{k,+}}
\left|
u(x)-A\,e^{-B\frac{|x|^{\tau}}{|\tau|}}\pxi|x|^{m}
\right|^{2}
\frac{dx}{|x|^{a+b+1}},
\end{align}

where, with \(s_1\) and \(s_2\) as defined above,
$$C_2(n,a,b,k)={\begin{cases}
    \min\{|\tau|,|s_1|\}, & k<n,\\
    \min\{|\tau|, |s_2|\}, & k=n,
\end{cases}}$$
\end{theorem}

We emphasize that Theorem \ref{Th:stability_ckn} covers both signs of $\tau$ at once, and hence all admissible parameters $(a,b)$ with $-a+b+1\neq0$. Taking $k=0$ we recover, in particular, stability estimates for the $L^{2}$-CKN inequalities on the whole space, and taking $a=-1$, $b=0$ we recover the results of \cite{LLLS26}. Sharp stability estimates for the HUP in the second-order and curl-free settings were obtained in \cite{DLLZ}.

The two branches of $C_{2}(n,a,b,k)$ should not be read as a cosmetic distinction. When $|\tau|\le|s_{1}|$ the two branches agree, since both minima return $|\tau|$; but in the band
\[
|s_{1}|<|\tau|<|s_{2}|,
\]
which is nonempty for every admissible choice of the remaining parameters, the constant for $k<n$ is $|s_{1}|$ while the constant for $k=n$ is $|\tau|$, and the corresponding second-order estimates are of a different nature altogether: for $k=n$ the deficit still controls the distance to the radial family, whereas the estimate that one would obtain by applying the $k<n$ result verbatim carries the gap $|s_{1}|-|\tau|<0$ and is therefore void. In other words, on the positive cone the naive transcription of the Euclidean statement is not merely suboptimal; it says nothing at all in the very range in which the orthant theory is new.

The stability estimates for the weighted Poincar\'e inequalities established above allow us to go one step further, and to study second-order stability phenomena for the $L^{2}$-CKN inequalities. More clearly, by combining the exact identities with the stability theory for the corresponding weighted Poincar\'e inequalities, we obtain a quantitative estimate for the stability deficit itself.

\begin{theorem}[Second stability of the Caffarelli--Kohn--Nirenberg inequality]\label{Th:second_stability_ckn}
Let {$0\le k\le n$}, let $u\in S_{a,b}(\mathbb{R}^{n}_{k,+})$, and assume that
\[
\tau:=-a+b+1\neq0 .
\]
Define the Caffarelli--Kohn--Nirenberg deficit by
\[
\rho(u)
=
\left(\int_{\mathbb{R}^{n}_{k,+}}\frac{|\nabla u|^{2}}{|x|^{2b}}\,dx\right)^{\frac12}
\left(\int_{\mathbb{R}^{n}_{k,+}}\frac{u^{2}}{|x|^{2a}}\,dx\right)^{\frac12}
-C(n,a,b,k)\int_{\mathbb{R}^{n}_{k,+}}\frac{u^{2}}{|x|^{a+b+1}}\,dx ,
\]
and let $E$ denote the set of the optimizers in Theorem \ref{Th:CKNIon_Orthant}. 
Then the following estimates hold.
\begin{enumerate}
\item  If {$k<n$ and  $|\tau|<|s_1|$,  or  $k=n$ and $|\tau|<|s_2|$}, then

\begin{align}
&\rho(u)-|\tau|\inf_{\phi\in E}\int_{\mathbb{R}^{n}_{k,+}}|u-\phi|^{2}\frac{dx}{|x|^{a+b+1}}
\\
&\ge
{\min\{|\tau|,\ |s_q|-|\tau|\}}
\inf_{\substack{A_{0},A_{1}\in\mathbb R\\ B>0}}
\int_{\mathbb{R}^{n}_{k,+}}
\left|
u-(A_{0}+A_{1}|x|^{\tau})e^{-B\frac{|x|^{\tau}}{|\tau|}}\pxi|x|^{m}
\right|^{2}
\frac{dx}{|x|^{a+b+1}} ,\\
&\text{where } s_q=\begin{cases} s_1,&k<n,\\ s_2,&k=n.\end{cases}
\end{align}

\item If $|\tau|>|s_1|$, and $k<n$ then

\begin{align}
&\rho(u)-|s_1|\inf_{\phi\in E}\int_{\mathbb{R}^{n}_{k,+}}|u-\phi|^{2}\frac{dx}{|x|^{a+b+1}}
\\
&\ge
\min\{|\tau|-|s_1|,\ |\tau|,\ |\alpha_1|\}
\\
&\times
\inf_{\substack{A_{0},\ b_{i}\in\mathbb R,\\ B>0}}
\int_{\mathbb{R}^{n}_{k,+}}
\left|
u-\left(A_{0}+\Bigl(\sum_{i=1}^{n-k}b_{i}x_{i}\Bigr)|x|^{s_1-1}\right)
e^{-B\frac{|x|^{\tau}}{|\tau|}}\pxi|x|^{m}
\right|^{2}
\frac{dx}{|x|^{a+b+1}} .
\end{align}

\item If $|\tau|=|s_1|$, and {$k<n$} then

\begin{align}
&\rho(u)-|\tau|\inf_{\phi\in E}\int_{\mathbb{R}^{n}_{k,+}}|u-\phi|^{2}\frac{dx}{|x|^{a+b+1}}
\\
&\ge
\min\{|\tau|,\ |\alpha_1|\}
\\
&\times
\inf_{\substack{A_{0},A_{1}\in\mathbb R,\\ b_{i}\in\mathbb R,\ B>0}}
\int_{\mathbb{R}^{n}_{k,+}}
\left|
u-\left(A_{0}+A_{1}|x|^{\tau}
+\Bigl(\sum_{i=1}^{n-k}b_{i}x_{i}\Bigr)|x|^{s_1-1}\right)
e^{-B\frac{|x|^{\tau}}{|\tau|}}\pxi|x|^{m}
\right|^{2}
\frac{dx}{|x|^{a+b+1}} .
\end{align}

\item If $|\tau|>|s_2|$, and $n=k$ then

\begin{align}
&\rho(u)-|s_2|\inf_{\phi\in E}\int_{\mathbb{R}^{n}_{k,+}}|u-\phi|^{2}\frac{dx}{|x|^{a+b+1}}
\\
&\ge
\min\{|\tau|, |\tau|-|s_2|,\ |\alpha_2|\}
\\
&\times
\inf_{\substack{A_{0}\in\mathbb R,\ B>0\\{\sum_{i=1}^{n} c_{i}=0}}}
\int_{\mathbb{R}^{n}_{k,+}}
\left|
u-\left(A_{0}
+\Bigl(\sum_{i=1}^{n}c_{i}x_{i}^2\Bigr)|x|^{s_2-2}\right)
e^{-B\frac{|x|^{\tau}}{|\tau|}}\pxi|x|^{m}
\right|^{2}
\frac{dx}{|x|^{a+b+1}} .
\end{align}

\item If $|\tau|=|s_2|$, and $n=k$ then

\begin{align}
&\rho(u)-|\tau|\inf_{\phi\in E}\int_{\mathbb{R}^{n}_{k,+}}|u-\phi|^{2}\frac{dx}{|x|^{a+b+1}}
\\
&\ge
\min\{|\tau|, \ |\alpha_2|\}
\\
&\times
\inf_{\substack{A_{0},A_1\in\mathbb R, \ B>0\\  {\sum_{i=1}^{n} c_{i}=0}}} \int_{\mathbb{R}^{n}_{k,+}}
\left|
u-\left(A_{0} +A_{1}|x|^{\tau}
+\Bigl(\sum_{i=1}^{n}c_{i}x_{i}^2\Bigr)|x|^{s_2-2}\right)
e^{-B\frac{|x|^{\tau}}{|\tau|}}\pxi|x|^{m}
\right|^{2}
\frac{dx}{|x|^{a+b+1}} .
\end{align}

\end{enumerate}

\end{theorem}

We close this introduction by summarising the points at which the present work departs from the existing literature. (i) The sharp constant of the $L^{2}$-CKN inequality on $\mathbb R^{n}_{k,+}$, together with the complete family of its extremizers, is new for every $k\ge1$, and it is not obtainable from the Euclidean results of \cite{CC09, CFL21, Cos08} by a change of parameters; the lifting argument that settles the Hardy inequality \cite{SuYang2012} and the HUP \cite{LLLS26} on orthants is unavailable here. (ii) The exact identities of Theorems \ref{Th:Ckn+_identity} and \ref{Th:Ckn-_identity} are established for both signs of $\tau$ and for all $k$, and they yield a third proof of the sharp inequality. (iii) The weighted Poincar\'e inequalities with monomial weights of Theorem \ref{Th:poincare_monomial_weight} appear to be new; their sharp constants depend on the weight not only quantitatively, through $N=n+|A|$, but qualitatively, through the number of degenerate directions. (iv) The case $k=n$ of the positive cone, in which the constant is governed by the second rather than the first non-radial mode and the optimizers are quadratic, has no analogue in the Euclidean theory of \cite{DLLN26}, nor in the orthant theory of \cite{LLLS26}, where the parameters $a=-1$, $b=0$ keep the problem outside the relevant band. (v) Finally, the second-order stability estimates of Theorem \ref{Th:second_stability_ckn} quantify the stability deficit itself, in all parameter regimes and for every $k$, with sharp constants.

Our paper is organized as follows. In Section \ref{S:preliminaries} we collect the calculus on orthants that we shall need, we prove an odd-reflection lemma relating functions on the orthant to functions on the whole space, and we record the one-dimensional weighted Poincar\'e inequality together with its equality cases. In Section \ref{S:CKN} we prove Theorem \ref{Th:CKNIon_Orthant}, thereby determining the sharp constant and all the optimizers of the $L^{2}$-CKN inequality on orthants, and we deduce a new proof of Theorem \ref{Th:CKNI}. Section \ref{S:identities} is devoted to the exact identities of Theorems \ref{Th:Ckn+_identity} and \ref{Th:Ckn-_identity}. In Section \ref{S:poincare} we establish the weighted Poincar\'e inequalities of generalized Gaussian type, first for radial measures and then, by lifting, for monomial weights, thus proving Theorem \ref{Th:poincare_monomial_weight}. Their stability, and in particular Theorem \ref{Th:stability_poincare_monomialweight}, is treated in Section \ref{S:poincare_stability}. In Section \ref{S:CKN_stability} we combine the identities with the Poincar\'e estimates and prove Theorems \ref{Th:stability_ckn} and \ref{Th:second_stability_ckn}, together with their Euclidean counterparts. Finally, the Appendix (Section \ref{S:appendix}) contains the density results used throughout the paper.
\section{Preliminaries}\label{S:preliminaries}

\subsection{Spherical harmonic decomposition}

In this subsection, we derive the preliminary identities and decomposition formulas which form the basis of the proofs of the main results. The key idea throughout the paper is to introduce the transformation
\[
u(x)=|x|^m v(x),
\]
for a suitable parameter \(m\), and then decompose the transformed function \(v\) into spherical harmonics. This allows us to separate the radial and angular contributions of the associated weighted energy functionals in a systematic way.

Let
$u \in C_c^\infty(\mathbb{R}^n \setminus \{0\})$
and let
$d\mu(x)=w(|x|)\,dx$.
We write
$u(x)=|x|^m v(x)$
for some
$v\in C^\infty_c(\mathbb R^n\setminus\{0\})$
and a suitable parameter \(m\).

Decomposing \(v\) into spherical harmonics,
\[
v(x)=\sum_{\ell=0}^{\infty}\sum_{j=1}^{d_\ell} f_{\ell,j}(r)\,\varphi_{\ell,j}(\sigma),
\qquad
r=|x|,
\qquad
\sigma=\frac{x}{|x|},
\]
we obtain
\[
u(x)=\sum_{\ell=0}^{\infty}\sum_{j=1}^{d_\ell} r^{m} f_{\ell,j}(r)\,\varphi_{\ell,j}(\sigma).
\]
Here, for each $\ell\ge0$, the functions $\{\varphi_{\ell,j}\}_{j=1}^{d_\ell}$ form an
orthonormal basis of the space of spherical harmonics of degree $\ell$ on
$\mathbb S^{n-1}$, so that
\[
-\Delta_{\mathbb S^{n-1}}\varphi_{\ell,j}
=
c_\ell\,\varphi_{\ell,j},
\qquad
c_\ell=\ell(\ell+n-2),
\qquad
\ell\ge0,
\]
with the normalization
\[
\int_{\mathbb S^{n-1}}|\varphi_{\ell,j}(\sigma)|^2\,d\sigma=1 .
\]
We emphasize that $c_\ell$ is the eigenvalue attached to the \emph{degree} $\ell$, and
that the corresponding eigenspace has dimension
$d_\ell=\binom{n+\ell-1}{\ell}-\binom{n+\ell-3}{\ell-2}$; in particular $d_0=1$ and
$d_1=n$, the eigenfunctions of degree one being the restrictions to $\mathbb S^{n-1}$ of
the linear forms $x\mapsto x_i$. Keeping the second index $j$ is therefore essential in
the discussion of the equality cases. Finally, the radial coefficients satisfy
$f_{\ell,j}(r)=O(r^\ell)$ and $f_{\ell,j}'(r)=O(r^{\ell-1})$ as \(r\to0\).

In polar coordinates, the gradient decomposes as
\[
\nabla u
=
\partial_r u\,\sigma
+
\frac1r\nabla_{\mathbb S^{n-1}}u,
\]
and therefore
\[
|\nabla u|^2
=
|\partial_r u|^2
+
\frac1{r^2}|\nabla_{\mathbb S^{n-1}}u|^2,
\]
since the radial and tangential directions are orthogonal.

We first compute the radial contribution:

\[
\int_{\R^n}|\partial_r u(x)|^2\,d\mu(x)
=\sum_{\ell=0}^{\infty}\sum_{j=1}^{d_\ell}\int_0^\infty
\Bigl(\bigl(r^m f_{\ell,j}(r)\bigr)'\Bigr)^2 w(r)r^{n-1}\,dr,
\]
and, for the angular contribution,
\[
\int_{\R^n}|\nabla_{\mathbb S^{n-1}}u(x)|^2\,d\mu(x)
=\sum_{\ell=0}^{\infty}\sum_{j=1}^{d_\ell}\int_0^\infty
\ell(n+\ell-2)\bigl(r^m f_{\ell,j}(r)\bigr)^2 w(r)r^{n-1}\,dr .
\]

Combining the radial and the angular parts, we obtain
\begin{equation}\label{eq:dumu}
\int_{\R^n}|\nabla u|^2\,d\mu(x)
=
\sum_{\ell=0}^{\infty}\sum_{j=1}^{d_\ell}\int_0^\infty
\left(
\Bigl(\bigl(r^m f_{\ell,j}(r)\bigr)'\Bigr)^2
+
\frac{\ell(n+\ell-2)\,r^{2m}f_{\ell,j}^2(r)}{r^2}
\right)
w(r)r^{n-1}\,dr,
\end{equation}
and, in the same way,
\begin{equation}\label{eq:umu}
\int_{\R^n}|u|^2\,d\mu(x)
=
\sum_{\ell=0}^{\infty}\sum_{j=1}^{d_\ell}\int_0^\infty
r^{2m}f_{\ell,j}^2(r)\,w(r)r^{n-1}\,dr .
\end{equation}

\subsection{Generalized Laguerre polynomials}
We briefly recall the standard facts concerning the generalized Laguerre polynomials that
will be used in the sequel. They are the natural orthogonal basis for the one-dimensional
weighted problems to which the spherical harmonic decomposition reduces our inequalities,
and the integer spacing of their eigenvalues is what produces the spectral gaps behind the
stability estimates.

For \(\alpha>-1\), the generalized Laguerre polynomials
\(\{L_n^{(\alpha)}\}_{n\ge0}\)
are defined as solutions of the differential equation
\begin{equation}
x y'' + (\alpha+1-x)y' + ny = 0,
\qquad n=0,1,2,\dots
\end{equation}

They form an orthogonal family in
\[
L^2\big((0,\infty),x^\alpha e^{-x}dx\big),
\]
with respect to the inner product
\begin{equation}
\langle f,g\rangle
=
\int_0^\infty
f(x)g(x)\,
x^\alpha e^{-x}\,dx.
\end{equation}

Moreover, they are eigenfunctions of the Laguerre operator
\begin{equation}
\mathcal L_\alpha f
=
\frac1{x^\alpha e^{-x}}
\frac{d}{dx}
\left(
x^{\alpha+1}e^{-x}f'(x)
\right),
\end{equation}
satisfying
\begin{equation}
\mathcal L_\alpha L_n^{(\alpha)}
=
-nL_n^{(\alpha)}.
\end{equation}

The family
\(\{L_n^{(\alpha)}\}_{n\ge0}\)
forms a complete orthogonal basis of
\[
L^2\big((0,\infty),x^\alpha e^{-x}dx\big).
\]
Consequently, every
\(f\in L^2((0,\infty),x^\alpha e^{-x}dx)\)
admits an expansion
\begin{equation}
f(x)
=
\sum_{n=0}^\infty
a_nL_n^{(\alpha)}(x).
\end{equation}

In particular, the mean-zero condition takes a simple form in terms of the expansion
coefficients; this is the observation that allows us, in the proofs below, to start the
Laguerre expansion at $n=1$.

\begin{proposition}\label{P:laguerre_mean_zero}
Let
\[
f(x)=\sum_{n=0}^\infty a_nL_n^{(\alpha)}(x)
\]
be the Laguerre expansion of
\(f\in L^2((0,\infty),x^\alpha e^{-x}dx)\).
Then
\[
\int_0^\infty
f(x)x^\alpha e^{-x}\,dx=0
\]
if and only if $a_0=0.$

\end{proposition}

The following lemma expresses the weighted Dirichlet energy of a Laguerre expansion
directly as the spectral sum $\sum_{n\ge1}n\,a_n^{2}\,h_n^{(\alpha)}$, at the level of
quadratic forms. It is what allows the weighted Poincar\'e inequalities of Sections
\ref{S:poincare} and \ref{S:poincare_stability} to be proved on the whole energy space,
without integration by parts and without applying $\mathcal L_{\alpha}$ term by term to the
expansion. Here and below we
write
\[
h_n^{(\gamma)}
:=
\bigl\|L_n^{(\gamma)}\bigr\|_{L^2(x^{\gamma}e^{-x}dx)}^{2}
=
\frac{\Gamma(n+\gamma+1)}{n!},
\qquad \gamma>-1 .
\]

\begin{lemma}[First-order Parseval identity for Laguerre expansions]\label{L:laguerre_parseval}
Let $\alpha>-1$, and let $g\in L^2\bigl((0,\infty),x^{\alpha}e^{-x}dx\bigr)$ possess a weak
derivative satisfying
\[
\int_0^\infty |g'(x)|^{2}\,x^{\alpha+1}e^{-x}\,dx<\infty .
\]
Let $g=\sum_{n\ge0}a_nL_n^{(\alpha)}$ be its Laguerre expansion. Then
\begin{equation}\label{eq:laguerre_parseval}
\int_0^\infty |g'(x)|^{2}\,x^{\alpha+1}e^{-x}\,dx
=
\sum_{n\ge1}n\,a_n^{2}\,h_n^{(\alpha)} .
\end{equation}
\end{lemma}

\begin{proof}
The family $\{L_m^{(\alpha+1)}\}_{m\ge0}$ is a complete orthogonal system in
$L^2\bigl((0,\infty),x^{\alpha+1}e^{-x}dx\bigr)$, since $\alpha+1>-1$. We compute the
coefficients of $g'$ with respect to this system. The starting point is the classical
raising identity, companion to $(L_n^{(\alpha)})'=-L_{n-1}^{(\alpha+1)}$:
\begin{equation}\label{eq:laguerre_raising}
\frac{d}{dx}\Bigl[x^{\alpha+1}e^{-x}L_m^{(\alpha+1)}(x)\Bigr]
=
(m+1)\,x^{\alpha}e^{-x}\,L_{m+1}^{(\alpha)}(x),
\qquad m\ge0 ,
\end{equation}
which is verified directly from the Rodrigues formula, or by comparing coefficients.

Fix $m\ge0$ and set $\phi(x):=x^{\alpha+1}e^{-x}L_m^{(\alpha+1)}(x)$ and $F:=g\phi$. On
compact subsets of $(0,\infty)$ the weights are bounded above and below, so $g\in
W^{1,1}_{\mathrm{loc}}(0,\infty)$ and $F'=g'\phi+g\phi'$ a.e. Both terms belong to
$L^1(0,\infty)$: by the Cauchy--Schwarz inequality,
\[
\int_0^\infty |g'|\,|\phi|\,dx
\le
\Bigl(\int_0^\infty |g'|^{2}x^{\alpha+1}e^{-x}dx\Bigr)^{\!1/2}
\Bigl(\int_0^\infty \bigl|L_m^{(\alpha+1)}\bigr|^{2}x^{\alpha+1}e^{-x}dx\Bigr)^{\!1/2}
<\infty,
\]
and, by \eqref{eq:laguerre_raising}, $|g\phi'|=(m+1)|g|\,\bigl|L_{m+1}^{(\alpha)}\bigr|
x^{\alpha}e^{-x}\in L^1$ for the same reason, using $g\in L^2(x^{\alpha}e^{-x}dx)$.
Consequently $F$ is absolutely continuous with $F'\in L^1(0,\infty)$, so the limits
$F(0^+)$ and $F(+\infty)$ exist. Both are zero: if $|F(x)|\ge c>0$ near an endpoint, then
$|g(x)|\ge c\,x^{-(\alpha+1)}e^{x}/\bigl|L_m^{(\alpha+1)}(x)\bigr|$ there, and a direct
computation shows that $\int |g|^{2}x^{\alpha}e^{-x}dx$ then diverges at that endpoint
(near $0$ the integrand dominates $x^{-\alpha-2}$ with $-\alpha-2<-1$; near $\infty$ it
grows exponentially), contradicting $g\in L^2(x^{\alpha}e^{-x}dx)$. Therefore

\begin{align*}
 \int_0^\infty g'\,L_m^{(\alpha+1)}\,x^{\alpha+1}e^{-x}\,dx
&=-\int_0^\infty g\,\Bigl[x^{\alpha+1}e^{-x}L_m^{(\alpha+1)}\Bigr]'\,dx\\
&=-(m+1)\int_0^\infty g\,L_{m+1}^{(\alpha)}\,x^{\alpha}e^{-x}\,dx\\
&=-(m+1)\,a_{m+1}\,h_{m+1}^{(\alpha)} .
\end{align*}

By the Parseval identity in $L^2(x^{\alpha+1}e^{-x}dx)$, using completeness,
\[
\int_0^\infty |g'|^{2}\,x^{\alpha+1}e^{-x}\,dx
=
\sum_{m\ge0}
\frac{\bigl[(m+1)\,a_{m+1}\,h_{m+1}^{(\alpha)}\bigr]^{2}}{h_m^{(\alpha+1)}} .
\]
Finally, with $n=m+1$,
\[
\frac{(m+1)^{2}\bigl(h_{m+1}^{(\alpha)}\bigr)^{2}}{h_m^{(\alpha+1)}}
=
n^{2}\,
\frac{\Gamma(n+\alpha+1)^{2}}{(n!)^{2}}\,
\frac{(n-1)!}{\Gamma(n+\alpha+1)}
=
n\,\frac{\Gamma(n+\alpha+1)}{n!}
=
n\,h_n^{(\alpha)},
\]
which gives \eqref{eq:laguerre_parseval}.
\end{proof}

For the applications, note that under the change of variables $t=\delta r^{\tau}$,
$g(t)=f(r)$, $\alpha=\frac{\beta-1}{\tau}$, of the proofs below, one has the exact
correspondences
\begin{align}\label{eq:parseval_dictionary}
 &\int_0^\infty |f'(r)|^{2}e^{-\delta r^{\tau}}r^{\beta}\,dr
=
|\tau|\,\delta^{-\alpha}
\int_0^\infty |g'(t)|^{2}\,t^{\alpha+1}e^{-t}\,dt,\\
&\int_0^\infty |f(r)|^{2}e^{-\delta r^{\tau}}r^{\beta+\tau-2}\,dr
=
\frac{\delta^{-\alpha}}{\delta|\tau|}
\int_0^\infty |g(t)|^{2}\,t^{\alpha}e^{-t}\,dt,   
\end{align}

so that $f\in Y_{\delta,\tau,\beta}$ if and only if $g$ satisfies the hypotheses of Lemma
\ref{L:laguerre_parseval}.

\subsection{Radial Poincar\'e inequality with general weights}
In this subsection we study the weighted Poincar\'e inequalities associated with the radial
measures $e^{-\delta r^{\tau}}r^{\beta}\,dr$ on $(0,\infty)$, which are the one-dimensional
problems arising from the Caffarelli--Kohn--Nirenberg setting. The substitution
$t=\delta r^{\tau}$ turns them into problems for the measure $t^{\alpha}e^{-t}\,dt$, and the
analysis therefore reduces to the spectral theory of the generalized Laguerre polynomials
recalled above.

We are now in a position to prove the radial Poincar\'e inequality with general weights. The
result extends those of \cite{DLLN26} to a slightly more general range of parameters; the
proof follows essentially the same lines, and we include it for the convenience of the
reader.

{Since the derivative and zeroth-order terms are associated with
different weights, we introduce the natural one-dimensional energy
space
\[
Y_{\delta,\tau,\beta}
:=
\left\{
f\in W_{\mathrm{loc}}^{1,2}(0,\infty):
\int_0^\infty |f'(r)|^2 e^{-\delta r^\tau}r^\beta\,dr
+
\int_0^\infty |f(r)|^2 e^{-\delta r^\tau}
r^{\beta+\tau-2}\,dr
<\infty
\right\}.
\]
We equip $Y_{\delta,\tau,\beta}$ with the norm
\[
\|f\|_{Y_{\delta,\tau,\beta}}^2
:=
\int_0^\infty |f'(r)|^2 e^{-\delta r^\tau}r^\beta\,dr
+
\int_0^\infty |f(r)|^2 e^{-\delta r^\tau}
r^{\beta+\tau-2}\,dr.
\]}

\begin{theorem}[Radial Poincar\'e inequality with general weights]\label{th:general_radial_poincare}
Let $\delta>0$ and $\tau\ne0$, and assume that
\[
\frac{\beta+\tau-1}{\tau}>0 .
\]
{Then, for every $f\in Y_{\delta,\tau,\beta}$, we have}
\[
\int_0^\infty |f'(r)|^2 e^{-\delta r^\tau}r^\beta\,dr
\;\ge\;
\delta\tau^2\int_0^\infty |f(r)-\bar f|^2 e^{-\delta r^\tau}r^{\beta+\tau-2}\,dr,
\]
where
\[
\bar f
:=
\frac{\int_0^\infty f(r)e^{-\delta r^\tau}r^{\beta+\tau-2}\,dr}
{\int_0^\infty e^{-\delta r^\tau}r^{\beta+\tau-2}\,dr}.
\]
\end{theorem}

\begin{proof}
Since replacing $f$ by $f-\bar f$ does not change $f'$, it is enough
to prove the inequality for $f$ satisfying
\[
\int_0^\infty f(r)e^{-\delta r^\tau}
r^{\beta+\tau-2}\,dr=0.
\]

We first record the following formal computation, which motivates the
quadratic-form identity used below. For this computation only, suppose
that $f$ is sufficiently smooth and that the relevant boundary terms
vanish.

We perform the change of variables $t=\delta r^\tau$. Note that this
map sends $(0,\infty)$ onto $(0,\infty)$; it is increasing when
$\tau>0$ and decreasing when $\tau<0$, so that in the integrals below
the Jacobian contributes the factor $\delta|\tau|r^{\tau-1}$, with the
absolute value. Explicitly,

\begin{align}
t=\delta r^\tau
&\implies dt=\delta\tau r^{\tau-1}\,dr,\\
g(t)=f(r)
&\implies f'(r)=g'(t)\delta\tau r^{\tau-1},\\
&\phantom{\implies}
f''(r)
=(\delta\tau)^2r^{2\tau-2}g''(t)
+\delta\tau(\tau-1)r^{\tau-2}g'(t).
\end{align}

Now define
\[
Lf:=\bigl(f'(r)e^{-\delta r^\tau}r^\beta\bigr)'.
\]
Then

\begin{align}
Lf
&=\left(
f''(r)+\left(\frac{\beta}{r}-\delta\tau r^{\tau-1}\right)f'(r)
\right)e^{-\delta r^\tau}r^\beta\\
&=\left(
(\delta\tau)^2r^{2\tau-2}g''(t)
+\delta\tau r^{\tau-2}g'(t)
\bigl((\tau-1)+\beta-\delta\tau r^\tau\bigr)
\right)e^{-\delta r^\tau}r^\beta\\
&=\left(
\delta r^\tau g''(t)
+g'(t)\left(\frac{(\tau-1)+\beta}{\tau}
-\delta r^\tau\right)
\right)
\delta\tau^2r^{\beta+\tau-2}e^{-\delta r^\tau}\\
&=\left(
tg''(t)
+g'(t)\left(
\left(\frac{(\tau-1)+\beta}{\tau}-1\right)+1-t
\right)
\right)
\delta\tau^2
\left(\frac{t}{\delta}\right)^{\frac{\beta+\tau-2}{\tau}}
e^{-t}\\
&=\delta\tau^2
\mathcal L_{\left(\frac{(\tau-1)+\beta}{\tau}-1\right)}g(t)
\left(\frac{t}{\delta}\right)^{\frac{\beta+\tau-2}{\tau}}
e^{-t}\\
&=\delta\tau^2
\mathcal L_{\left(\frac{\beta-1}{\tau}\right)}g(t)
\left(\frac{t}{\delta}\right)^{\frac{\beta+\tau-2}{\tau}}
e^{-t}.
\end{align}

Setting
\[
\alpha=\frac{\beta-1}{\tau},
\]
the preceding formal computation gives
\begin{align}
\int_0^\infty |f'(r)|^2e^{-\delta r^\tau}r^\beta\,dr
&=-\int_0^\infty f(r)Lf(r)\,dr\\
&=-\delta\tau^2\frac{1}{\delta|\tau|}
\int_0^\infty
g(t)\mathcal L_\alpha g(t)
\left(\frac{t}{\delta}\right)^\alpha e^{-t}\,dt\\
&=-\delta\tau^2\frac{1}{\delta|\tau|}
\left(\frac{1}{\delta}\right)^\alpha
\int_0^\infty
g(t)\mathcal L_\alpha g(t)t^\alpha e^{-t}\,dt.
\end{align}

This computation is included only as a formal motivation. For a
general $f\in Y_{\delta,\tau,\beta}$, the rigorous argument below uses
the quadratic-form identity \eqref{eq:laguerre_parseval}, and hence
does not require integration by parts on $f$.

Indeed, since
\[
f\in L^2\bigl((0,\infty),
e^{-\delta r^\tau}r^{\beta+\tau-2}\,dr\bigr),
\]
the change of variables gives
\[
g\in L^2\bigl((0,\infty),t^\alpha e^{-t}\,dt\bigr).
\]
Moreover,
\[
\int_0^\infty f(r)e^{-\delta r^\tau}
r^{\beta+\tau-2}\,dr
=
\frac{1}{\delta|\tau|}
\left(\frac{1}{\delta}\right)^\alpha
\int_0^\infty g(t)t^\alpha e^{-t}\,dt.
\]
Therefore, if $f$ has mean zero, then $g$ also has mean zero with
respect to the transformed measure.

Hence, by Proposition~\ref{P:laguerre_mean_zero}, $g$ is orthogonal to
the constant function and admits the generalized Laguerre expansion
\[
g(t)=\sum_{n=1}^\infty a_nL_n^{(\alpha)}(t),
\]
where $\{L_n^{(\alpha)}\}_{n\geq0}$ is the orthogonal basis of
generalized Laguerre polynomials in
\[
L^2\bigl((0,\infty),t^\alpha e^{-t}\,dt\bigr).
\]

By the dictionary \eqref{eq:parseval_dictionary}, $g$ satisfies the
hypotheses of Lemma~\ref{L:laguerre_parseval}. Therefore, using the
change of variables and then the quadratic-form identity
\eqref{eq:laguerre_parseval}, we obtain
\begin{align}
\int_0^\infty |f'(r)|^2e^{-\delta r^\tau}r^\beta\,dr
&=
\delta\tau^2\frac{1}{\delta|\tau|}
\left(\frac{1}{\delta}\right)^\alpha
\int_0^\infty |g'(t)|^2t^{\alpha+1}e^{-t}\,dt
\notag\\
&=
\delta\tau^2\frac{1}{\delta|\tau|}
\left(\frac{1}{\delta}\right)^\alpha
\sum_{n=1}^\infty n a_n^2h_n^{(\alpha)}
\notag\\
&\geq
\delta\tau^2\frac{1}{\delta|\tau|}
\left(\frac{1}{\delta}\right)^\alpha
\sum_{n=1}^\infty a_n^2h_n^{(\alpha)}
\notag\\
&=
\delta\tau^2\frac{1}{\delta|\tau|}
\left(\frac{1}{\delta}\right)^\alpha
\int_0^\infty |g(t)|^2t^\alpha e^{-t}\,dt
\notag\\
&=
\delta\tau^2
\int_0^\infty |f(r)|^2e^{-\delta r^\tau}
r^{\beta+\tau-2}\,dr.
\end{align}

Finally, the condition $\alpha>-1$, which is required for the
Laguerre expansion, is equivalent to
\[
\frac{\beta+\tau-1}{\tau}>0.
\]
This completes the proof.
\end{proof}

\begin{corollary}[Equality case]\label{cor:general_radial_poincare_equality}
Under the assumptions of Theorem \ref{th:general_radial_poincare}, equality holds in
\[
\int_0^\infty |f'(r)|^2 e^{-\delta r^\tau}r^\beta\,dr
\;\ge\;
\delta\tau^2\int_0^\infty |f(r)-\bar f|^2 e^{-\delta r^\tau}r^{\beta+\tau-2}\,dr
\]
if and only if
\[
f(r)=a+br^\tau
\]
for some constants $a,b\in\mathbb R$.
\end{corollary}

\begin{proof}
From the proof of Theorem \ref{th:general_radial_poincare}, writing
\[
g(t)=\sum_{n=1}^\infty a_n L_n^{(\alpha)}(t),
\]
equality holds if and only if $g(t)=a_1L_1^{(\alpha)}(t),$ and so
\[
f(r)-\bar f=a_1\,L_1^{(\alpha)}(\delta r^\tau).
\]
Using
\[
L_1^{(\alpha)}(t)=-t+\alpha+1,
\]
we get
\[
f(r)=\bar f+a_1\left(-\delta r^\tau+\frac{\beta+\tau-1}{\tau}\right).
\]
This proves the claim.
\end{proof}

\subsection{Odd reflection across the walls of an orthant}

The passage from the orthant to the whole space is carried out by an odd reflection, and
the following lemma records the effect of that reflection on the spherical harmonic
decomposition. It is the analogue, in our setting, of the elementary observation that an
odd function on the line has no even Fourier modes.

\begin{lemma}\label{L:reflection}
Let $0\le k\le n$ and let $u\in C_c^\infty(\mathbb R^n_{k,+}\setminus\{0\})$. Write
$x=(x',x'')$ with $x'=(x_1,\dots,x_{n-k})$ and $x''=(x_{n-k+1},\dots,x_n)$, and define the
odd extension
\[
\widetilde u(x',x''):=\Bigl(\prod_{i=n-k+1}^{n}\sgn(x_i)\Bigr)\,u\bigl(x',|x_{n-k+1}|,\dots,|x_n|\bigr).
\]
Then $\widetilde u\in C_c^\infty(\mathbb R^n\setminus\{0\})$, and in the spherical harmonic
decomposition
\[
\widetilde u(x)=\sum_{\ell=0}^{\infty}\sum_{j=1}^{d_\ell}F_{\ell,j}(r)\varphi_{\ell,j}(\sigma)
\]
only spherical harmonics that are odd in each of the last $k$ variables occur. In
particular $F_{\ell,j}\equiv0$ for every $\ell<k$, and the space of spherical harmonics of
degree exactly $k$ that are odd in each of the variables $x_{n-k+1},\dots,x_n$ is
one-dimensional, spanned by the restriction to $\mathbb S^{n-1}$ of
\[
x\longmapsto\prod_{i=n-k+1}^{n}x_i .
\]
Moreover, for every $q\in\mathbb R$,
\[
\int_{\R^n}|\widetilde u|^2\,|x|^{q}\,dx=2^{k}\int_{\Rnkp}|u|^2\,|x|^{q}\,dx,
\qquad
\int_{\R^n}|\nabla\widetilde u|^2\,|x|^{q}\,dx=2^{k}\int_{\Rnkp}|\nabla u|^2\,|x|^{q}\,dx,
\]
whenever the integrals on the right-hand sides are finite.
\end{lemma}

\begin{proof}
That $\widetilde u$ is smooth away from the origin, and compactly supported there, is
immediate from the fact that $u$ vanishes to infinite order on each wall
$\{x_i=0\}$, $i\ge n-k+1$, of the orthant. By construction $\widetilde u$ changes sign
under the reflection $x_i\mapsto-x_i$ for each $i\in\{n-k+1,\dots,n\}$. Each such
reflection is an orthogonal transformation of $\mathbb R^n$, hence it acts on
$\mathbb S^{n-1}$ and preserves each eigenspace $\mathcal H_\ell$ of
$-\Delta_{\mathbb S^{n-1}}$. Decomposing $\mathcal H_\ell$ into the joint eigenspaces of
these $k$ commuting involutions, and projecting, we see that only the joint eigenspace on
which every reflection acts by $-1$ can contribute to $\widetilde u$; this is the assertion
that only harmonics odd in each of the last $k$ variables occur.

It remains to identify the harmonics of low degree with this parity. Let $P$ be a harmonic
polynomial of degree $\ell$, homogeneous, and odd in each of $x_{n-k+1},\dots,x_n$. For a
fixed $i\ge n-k+1$, oddness in $x_i$ means that every monomial of $P$ carries an odd, hence
positive, power of $x_i$; therefore $x_i$ divides $P$. Since the variables
$x_{n-k+1},\dots,x_n$ are distinct, the product $x_{n-k+1}\cdots x_n$ divides $P$, and
consequently $\ell\ge k$. If $\ell=k$, then $P$ is a constant multiple of
$x_{n-k+1}\cdots x_n$, which is indeed harmonic; the corresponding space is therefore
one-dimensional. This proves the first two assertions.

The integral identities follow by decomposing $\R^n$, up to a set of measure zero, into the
$2^{k}$ orthants obtained from $\Rnkp$ by the reflections above, on each of which
$|\widetilde u|=|u\circ\iota|$ and $|\nabla\widetilde u|=|(\nabla u)\circ\iota|$ for the
corresponding reflection $\iota$, while $|x|$ is invariant.
\end{proof}


\section{Caffarelli--Kohn--Nirenberg inequalities on orthants: proof of Theorem \ref{Th:CKNIon_Orthant}}\label{S:CKN}

In this section we prove Theorem \ref{Th:CKNIon_Orthant}. By the density results of the
Appendix, and in particular Remark \ref{rmk:density_orthant}, it suffices to argue for
$u\in C_c^\infty(\mathbb R^n_{k,+}\setminus\{0\})$, where
$0\le k\le n$.

The proof rests on two ingredients. The first is the odd reflection of Lemma
\ref{L:reflection}, which transports the problem from the orthant to the whole space at the
cost of restricting attention to the harmonics of degree at least $k$; this is the source of
the dimension shift that distinguishes the orthant from $\mathbb R^{n}$. The second is the
substitution $u(x)=|x|^{m}v(x)$, followed by the spherical harmonic decomposition of $v$
rather than of $u$. With $m$ chosen so as to annihilate a certain coefficient, the whole
angular contribution is absorbed, the problem reduces to a one-dimensional estimate, and
both the sharp constant and the extremal functions can be read off. It is worth noting that
neither the Kelvin-type transform of \cite{CC09} nor the lifting argument of
\cite{LLLS26, SuYang2012} is used anywhere in the argument.

\begin{proof}[Proof of Theorem \ref{Th:CKNIon_Orthant}]

Let $u\in C_c^\infty(\mathbb R^n_{k,+}\setminus\{0\})$, and let $\widetilde u$ be its odd
extension to $\mathbb R^n$, as in Lemma \ref{L:reflection}. By that lemma,
$\widetilde u\in C_c^\infty(\mathbb R^n\setminus\{0\})$, and its spherical harmonic
decomposition involves only harmonics that are odd in each of the last $k$ variables; in
particular, only the degrees $\ell\ge k$ occur. Writing $\widetilde u(x)=|x|^{m}v(x)$ and
decomposing $v$ as in Section \ref{S:preliminaries}, we therefore obtain
\[
\widetilde u(x)
=
\sum_{\ell=k}^{\infty}
r^m f_\ell(r)\,\varphi_\ell(\sigma),
\qquad
x=r\sigma,
\qquad
r=|x|,
\qquad
\sigma\in\mathbb S^{n-1},
\]
where, in order to keep the notation light, the index $\ell$ runs over a fixed orthonormal
basis of the admissible harmonics rather than over the degrees alone; the degree of
$\varphi_\ell$ is the number $\ell$ occurring in the eigenvalue $\ell(n+\ell-2)$.

Moreover, because \(\ell\ge k\), we factor out \(r^k\) and rewrite
\[
f_\ell(r)=r^k g_\ell(r).
\]
Consequently,
\[
\widetilde u(x)
=
\sum_{\ell=k}^{\infty}
r^t g_\ell(r)\varphi_\ell(\sigma),
\qquad
t=m+k.
\]
Now using \eqref{eq:dumu} and \eqref{eq:umu},

\begin{align}
    &2^{2k}\left( \int_{\Rnkp} \frac{|\nabla u|^2}{|x|^{2b}} \; dx\right) \left( \int_{\Rnkp} \frac{| u|^2}{|x|^{2a}}\;dx \right)=\left( \int_{\R^n} \frac{|\nabla \widetilde u|^2}{|x|^{2b}} \; dx\right) \left( \int_{\R^n} \frac{| \widetilde u|^2}{|x|^{2a}}\;dx \right)
    \notag \\
    &= \left(\sum_{\ell=k}^\infty\int_0^\infty \Big(\left(\left(r^tg_\ell(r)\right)'\right)^2+\frac{\left(r^{2t}g^2_\ell(r)\right)\; \ell(n+\ell-2)}{r^2} \Big) \; r^{n-1-2b}dr\right)\left(  \sum_{\ell=k}^\infty\int_0^\infty r^{2t}g^2_\ell(r) \; r^{n-1-2a}dr\right)
    \notag\\
    &\geq  \left(\sum_{\ell=k}^\infty \int_0^\infty \left( \left(\left(r^tg_\ell(r)\right)'\right)^2+ \frac{\left(r^{2t}g^2_\ell(r)\right)\; k(n+k-2)}{r^2} \; \right)r^{n-1-2b}dr\right)\left( \sum_{\ell=k}^\infty\int_0^\infty r^{2t}g^2_\ell(r) \; r^{n-1-2a}dr\right)
    \notag\\ 
    &= \left( \sum_{\ell=k}^\infty \left(\int_0^\infty r^{2t}(g'_\ell(r))^2 \; r^{n-1-2b}dr\;+ (t^2+k(n+k-2))\int_0^\infty r^{2t-2} g^2_{\ell}(r) \; r^{n-1-2b}dr \right.\right.
    \notag\\
     & \qquad \qquad \qquad \qquad \qquad \quad\left. \left.+t\int_0^\infty r^{2t-1} (g^2_{\ell})' r^{n-1-2b}dr\right)\right) \times \left(  \sum_{\ell=k}^\infty\int_0^\infty r^{2t}g^2_\ell(r) \; r^{n-1-2a}dr\right)
    \notag\\
    &= \left( \sum_{\ell=k}^\infty \left(\int_0^\infty r^{2t}(g'_\ell(r))^2 r^{n-1-2b}dr \right.\right.\notag\\
    & \qquad \qquad \qquad\left.\left.+(t^2-t(2t+n-2b-2)+k(n+k-2))\int_0^\infty r^{2t-2} g^2_{\ell}(r) r^{n-1-2b}dr\right)\right) \notag\\
    & \qquad \qquad \qquad \qquad \qquad \qquad  \qquad \qquad \qquad \qquad \qquad \qquad  \times \left(\sum_{\ell=k}^\infty  \int_0^\infty r^{2t}g^2_\ell(r) \; r^{n-1-2a}dr\right)\label{eq:1.1}
    \end{align}
    We now choose \(t\) so that
    \[
    t^2-t(2t+n-2b-2)+k(n+k-2)=0.
    \]
    Equivalently,
    \begin{equation}
        t=\frac{-(n-2-2b)\pm
    \sqrt{(n-2-2b)^2+4k(n+k-2)}}{2}.  \label{eq:value_of_t}
    \end{equation}
With this choice,  the previous identity reduces to
\begin{align}
    &2^{2k}\left( \int_{\Rnkp} \frac{|\nabla u|^2}{|x|^{2b}} \; dx\right) \left( \int_{\Rnkp} \frac{| u|^2}{|x|^{2a}}\;dx \right)\notag\\
    &= \left(\sum_{\ell=k}^\infty \int_0^\infty r^{2t}(g'_\ell(r))^2 \; r^{n-1-2b}dr\right)\left( \sum_{\ell=k}^\infty \int_0^\infty r^{2t}g^2_\ell(r) \; r^{n-1-2a}dr\right)\notag\\ 
    &\text{By Cauchy-Schwarz inequality} \notag\\
    &\geq \left(\sum_{\ell=k}^\infty\int_0^\infty  r^{2t} g_\ell(r)g'_\ell(r) \; r^{n-1-b-a}dr\right)^2\notag \\
    &= \left( \frac12\int_0^\infty \sum_{\ell=k}^\infty r^{2t} (g^2_\ell(r))'\; r^{n-1-b-a}dr\right)^2\notag\\
     &= \left( \frac{n+2t-(a+b+1)}2\int_0^\infty \sum_{\ell=k}^\infty r^{2t} g^2_\ell(r) \; \frac{r^{n-1}}{r^{a+b+1}}dr\right)^2\notag \\
     &=\left(\frac{n+2t-(a+b+1)}2\int_{\R^n} \frac{|\widetilde u|^2}{|x|^{a+b+1}} \; dx\right)^2 \notag \\
     &=\left(\frac{-a+b+1 \pm \sqrt{(n-2-2b)^2+4k(n+k-2)}}{2}\right)^2\left(\int_{\R^n} \frac{|\widetilde u|^2}{|x|^{a+b+1}} \; dx\right)^2 \notag\\
     &=2^{2k}\left(\frac{-a+b+1 \pm \sqrt{(n-2-2b)^2+4k(n+k-2)}}{2}\right)^2\left(\int_{\Rnkp} \frac{|u|^2}{|x|^{a+b+1}} \; dx\right)^2\label{eq:ckn}
\end{align}

The inequality comes from using the Cauchy--Schwarz inequality and dropping the terms coming from the harmonics of degree $\ell>k$. Therefore, equality holds if and only if the higher spherical harmonic modes vanish, i.e
\[g_\ell \equiv 0 \qquad \forall\,\ell>k \qquad \text{and} \quad
g_k(r)=A\exp\!\left(\frac{\lambda}{b-a+1}\,r^{\,b-a+1}\right) 
\qquad \text{ for some A, $\lambda \in \R$ },
\]

Hence
\[
\widetilde u(x)
=|x|^{m+k}\,g_k(|x|)\,\varphi_k\!\left(\frac{x}{|x|}\right).
\]

Moreover, because $\widetilde u$ is odd in each of the last $k$ variables, the degree-$k$ spherical harmonic $\varphi_k(\sigma)$  must be proportional to
\(
 \sigma_{n-k+1}\cdots\sigma_n.
\)
Therefore equality can occur only for functions of the form
\[
\widetilde u(x)
=
A\Big(\prod_{i=n-k+1}^n x_i\Big)\,|x|^{m}\,\exp\!\left(\frac{\lambda}{b-a+1}\,|x|^{\,b-a+1}\right).
\]
This implies
$$u(x)=A\Big(\prod_{i=n-k+1}^n x_i\Big)\,|x|^{m}\exp\!\Big(\frac{\lambda}{\tau}\,|x|^{\tau}\Big). \quad  \tau= -a+b+1$$

Two remarks are in order. First, the coefficient in \eqref{eq:value_of_t} vanishes for
either choice of the sign, so that the computation above in fact produces two valid
inequalities, with the two constants $\bigl|\tau\pm\Theta\bigr|/2$, where we have set
\[
\tau:=-a+b+1,
\qquad
\Theta:=\sqrt{(n-2b-2)^2+4k(n+k-2)}\ \ge 0 .
\]
Since $\max\bigl\{|\tau+\Theta|,\,|\tau-\Theta|\bigr\}=|\tau|+\Theta$, the better of the two
is obtained by choosing the sign of the square root to agree with the sign of $\tau$, and
this gives exactly
\[
C(n,a,b,k)=\frac{|\tau|+\Theta}{2},
\qquad
2t=-(n-2b-2)+\sgn(\tau)\,\Theta,
\]
which is the constant and the exponent stated in Theorem \ref{Th:CKNIon_Orthant}. Second,
the function produced above is not admissible for every value of the parameters, since the
associated weighted integrals may fail to be finite. We now check that, for the branch just
selected, it is. Writing $\widetilde u=A\,r^{t}g(r)\varphi_k(\sigma)$ with
$g(r)=\exp\bigl(\tfrac{\lambda}{\tau}r^{\tau}\bigr)$, we have

\begin{align*}
\int_{\R^n}\frac{|\widetilde u(x)|^2}{|x|^{2a}}\,dx
&=A^2\left(\int_{\mathbb S^{n-1}}|\varphi_k(\sigma)|^2\,d\sigma\right)
\int_0^\infty r^{\,n-1+2t-2a}\exp\!\Big(\frac{2\lambda}{\tau}\,r^\tau\Big)\,dr \\
&= C_k \int_0^\infty r^{\,n-1+2t-2a}\exp\!\Big(\frac{2\lambda}{\tau}\,r^\tau\Big)\,dr
\end{align*}
for a constant $C_k>0$, and, since the radial derivative of $\widetilde u$ contributes the
factor $\lambda r^{\tau-1}$ while $2(\tau-1)-2b=-2a$, the gradient term is governed by the
same integral:
\[
{\int_{\R^n}\frac{|\nabla\widetilde u(x)|^2}{|x|^{2b}}\,dx
\le
C_k'\int_0^\infty r^{\,n-1+2t-2a}\exp\!\Big(\frac{2\lambda}{\tau}\,r^\tau\Big)\,dr .}
\]

Thus admissibility reduces to the finiteness of a single integral, namely
\begin{equation}\label{eq:admissibility_integral}
\int_0^\infty r^{\,n-1+2t-2a}\exp\!\Big(\frac{2\lambda}{\tau}\,r^\tau\Big)\,dr<\infty .
\end{equation}

Consider first the case $\tau>0$. Convergence of \eqref{eq:admissibility_integral} at
infinity forces $\lambda<0$, and convergence at the origin requires
$n+2t-2a>0$, that is, $2t>2a-n$. Now $\tau>0$ means $a<b+1$, whence
\[
2a-n<2(b+1)-n=-(n-2b-2)\le-(n-2b-2)+\Theta=2t,
\]
so that the required inequality holds automatically for the branch selected above. 

Consider now the case $\tau<0$. Convergence at infinity forces $\lambda>0$, and convergence
at the origin now requires $2t<2a-n$. Since $\tau<0$ means $a>b+1$, we have
\[
2a-n>2(b+1)-n=-(n-2b-2)\ge-(n-2b-2)-\Theta=2t,
\]
and again the selected branch is admissible. In both cases the exponent is
\[
2t=-(n-2b-2)+\sgn(\tau)\sqrt{(n-2b-2)^2+4k(n+k-2)},
\]
that is, $t=m+k$ with $m$ as in the statement.

Finally, if $\tau=0$ the exponential factor degenerates to a constant, the integral
\eqref{eq:admissibility_integral} diverges for every choice of $t$, and no extremal function
belongs to $S_{a,b}(\Rnkp)$; {the constant $\frac{\Theta}{2}$ is nevertheless still optimal,
as one sees by testing the inequality on the usual truncations of $|x|^{s}\varphi_k$.}  Consider the formal function $\widetilde u(r,\sigma)=r^s\varphi_k(\sigma)$, then $\widetilde u(r,\sigma)$ is odd in each of the last $k$ {variables}. Also when $2a=a+b+1=2b+2$, the Caffarelli--Kohn--Nirenberg inequality reduces to
\[
\int_{\mathbb{R}^n}
\frac{|\nabla \widetilde u|^2}{|x|^{2b}}\,dx
\geq
\frac{\Theta^2}{4}
\int_{\mathbb{R}^n}
\frac{|\widetilde u|^2}{|x|^{2b+2}}\,dx,
\]
where  $\Theta^2=(n-2b-2)^2+4{\lambda_k^{(n)}}.$ {After calculating the terms we have} 
\[
\begin{aligned}
\int_{\mathbb R^n}
\frac{|\widetilde u|^2}{|x|^{2b+2}}\,dx
&=
\left(
\int_{\mathbb S^{n-1}}\phi_k^2\,d\sigma
\right)
\int_0^\infty
r^{2s+n-2b-3}\,dr.
\end{aligned}
\]
and

\[
\begin{aligned}
\int_{\mathbb R^n}
\frac{|\nabla\widetilde u|^2}{|x|^{2b}}\,dx
&=
(s^2+\lambda_k)
\left(
\int_{\mathbb S^{n-1}}\phi_k^2\,d\sigma
\right)
\int_0^\infty r^{2s+n-2b-3}\,dr.
\end{aligned}
\]
Hence, formally,
\[
\frac{
\displaystyle
\int_{\mathbb R^n}|x|^{-2b}|\nabla\widetilde u|^2\,dx
}{
\displaystyle
\int_{\mathbb R^n}|x|^{-2b-2}|\widetilde u|^2\,dx
}
=
s^2+\lambda_k.
\]

In order to obtain the proposed constant, we require
\[
s^2+\lambda_k=\frac{\Theta^2}{4}\implies s^2=\frac{(n-2b-2)^2}{4} 
\]
We choose
\[
s=-\frac{n-2b-2}{2} \implies 2s+n-2b-3=-1,
\]

and therefore
\[
\int_0^\infty r^{2s+n-2b-3}\,dr
=
\int_0^\infty\frac{dr}{r}
=
\infty.
\]
Thus $r^s\phi_k(\sigma)$ is only a formal extremal and is not
admissible.

We now construct an extremizing sequence. Choose
\[
\eta\in C_c^\infty(\mathbb R),
\qquad
0\leq\eta\leq1,
\]
such that
\[
\eta(t)=1 \quad\text{for } |t|\leq1,
\qquad
\eta(t)=0 \quad\text{for } |t|\geq2.
\]
For $L>0$, define
\[
\widetilde u_L(r,\sigma)
=
r^s\phi_k(\sigma)
\eta\left(\frac{\log r}{L}\right).
\]
Since the cutoff is radial, $\widetilde u_L$ is odd in each of the
last $k$ variables. Moreover, it vanishes for
$r\leq e^{-2L}$ and for $r\geq e^{2L}$, and hence $\widetilde u_L\in C_c^\infty(\mathbb R^n \setminus {\{0\}}).$

Using the change of variables  $y=\frac{\log r}{L}$  we obtain
\begin{align}
  \int_{\mathbb R^n}
\frac{|\widetilde u_L|^2}{|x|^{2b+2}}\,dx &=
\left(
\int_{\mathbb S^{n-1}}\phi_k^2\,d\sigma
\right)
\int_0^\infty
\eta^2\left(\frac{\log r}{L}\right)\frac{dr}{r}.\\
&=
L
\left(
\int_{\mathbb S^{n-1}}\phi_k^2\,d\sigma
\right)
\int_{\mathbb R}\eta^2(y)\,dy.
\end{align}

\begin{align}
\int_{\mathbb R^n}
&\frac{|\nabla\widetilde u_L|^2}{|x|^{2b}}\,dx\\
&= \int_0^\infty \int_{\Sn} \left(\phi_k^2
\left[s\eta\left(\frac{\log r}{L}\right) + \frac1L
\eta'\left(\frac{\log r}{L}\right)
\right]^2 + 
\eta^2\left(\frac{\log r}{L}\right)
{|\nabla_\sigma\phi_k|^2} \right)r^{2s-3-2b+n} d\sigma \;dr.\\
&=
L
\left(
\int_{\mathbb S^{n-1}}\phi_k^2\,d\sigma
\right)
\int_{\mathbb R}
\left[
\left(s\eta(y)+\frac1L{\eta'(y)}\right)^2
+\lambda_k\eta^2(y)
\right]dy.
\end{align}

Expanding the square, and knowing $\eta$ is compactly supported, integration by parts gives
\[
\int_{\mathbb R}\eta\eta'\,dt=0.
\]
Consequently,
\[
{\int_{\mathbb R^n}}\frac{|\nabla\widetilde u_L|^2}{|x|^{2b}}\,dx=\left(
\int_{\mathbb S^{n-1}}\phi_k^2\,d\sigma
\right) \left(
L( s^2+\lambda_k)
\int_{\mathbb R}\eta^2\,dt
+
\frac1L
\int_{\mathbb R}(\eta')^2\,dt\right).
\]

Finally, we obtain
\[
\lim_{L \to \infty}\frac{\int_{\mathbb R^n}\frac{|\nabla\widetilde u_L|^2}{|x|^{2b}}\,dx}{\int_{\mathbb R^n}
\frac{|\widetilde u_L|^2}{|x|^{2b+2}}\,dx}
=
s^2+\lambda_k
+
\lim_{L \to \infty}\frac1{L^2}
\frac{
\displaystyle\int_{\mathbb R}(\eta')^2\,dt
}{
\displaystyle\int_{\mathbb R}\eta^2\,dt
}=s^2+\lambda_k
=
\frac{\Theta^2}{4}.
\]
Therefore no constant larger than $\Theta^2/4$ can hold. Since the
inequality with constant $\Theta^2/4$ has already been established,
$\Theta^2/4$ is the sharp constant.

This completes the characterization of the admissible extremal functions.

\end{proof}

We now recover the classical Euclidean $L^2$-Caffarelli--Kohn--Nirenberg inequalities as a consequence of Theorem~\ref{Th:CKNIon_Orthant}. More precisely, by taking \(k=0\), the orthant reduces to the whole space \(\mathbb R^n\), and the sharp constants and extremal functions split into different regimes depending on the sign of \(b+1-a\) and the location of \(b\) relative to \(\frac{n-2}{2}\). This yields the four parameter regions appearing in Theorem \ref{Th:CKNI}.

\begin{proof}[Another proof of Theorem \ref{Th:CKNI}]
The result follows from Theorem~\ref{Th:CKNIon_Orthant} by taking \(k=0\). In this case \(\mathbb R^n_{0,+}=\mathbb R^n\) and
\[
\sqrt{(n-2-2b)^2+4k(n+k-2)}
=
|n-2-2b|.
\]
If \(b+1-a>0\), then Theorem~\ref{Th:CKNIon_Orthant} gives
\[
C(n,a,b,0)
=
\left|
\frac{b+1-a+|n-2-2b|}{2}
\right|.
\]
If \(b\le \frac{n-2}{2}\), then \(|n-2-2b|=n-2-2b\), and hence
\[
C(n,a,b,0)
=
\frac{|n-(a+b+1)|}{2}.
\]
This is precisely the region \(\mathcal A_1\). If \(b\ge \frac{n-2}{2}\), then \(|n-2-2b|=2b+2-n\), and therefore
\[
C(n,a,b,0)
=
\frac{|n-(3b-a+3)|}{2}.
\]
This corresponds to the region \(\mathcal B_2\).

Similarly, if \(b+1-a<0\), then Theorem~\ref{Th:CKNIon_Orthant} gives
\[
C(n,a,b,0)
=
\left|
\frac{b+1-a-|n-2-2b|}{2}
\right|.
\]
If \(b\le \frac{n-2}{2}\), then
\[
C(n,a,b,0)
=
\frac{|n-(3b-a+3)|}{2},
\]
which gives the region \(\mathcal B_1\). If \(b\ge \frac{n-2}{2}\), then
\[
C(n,a,b,0)
=
\frac{|n-(a+b+1)|}{2},
\]
which gives the region \(\mathcal A_2\).

Thus, the regions \(\mathcal A=\mathcal A_1\cup\mathcal A_2\) give the constant
\[
C(n,a,b)=\frac{|n-(a+b+1)|}{2},
\]
while the regions \(\mathcal B=\mathcal B_1\cup\mathcal B_2\) give the constant
\[
C(n,a,b)=\frac{|n-(3b-a+3)|}{2}.
\]
The corresponding extremal functions also follow from Theorem~\ref{Th:CKNIon_Orthant}. When \(k=0\), $\pxi$ disappears.  In the regions \(\mathcal A\), the exponent \(m\) reduces to \(m=0\), giving
\[
u(x)=D\exp\left(\frac{\lambda|x|^{b+1-a}}{b+1-a}\right).
\]

In the regions \(\mathcal B\), the exponent becomes
\[
m=2(b+1)-n,
\]
giving
\[
u(x)
=
D|x|^{2(b+1)-n}
\exp\left(\frac{\lambda|x|^{b+1-a}}{b+1-a}\right).
\]
The signs of \(\lambda\) are determined by the integrability conditions at \(0\) and \(\infty\), yielding the cases stated in the theorem. Finally, on the line \(a=b+1\), the extremal functions are not attained in the corresponding weighted space, and the sharp constant is
\[
C(n,b+1,b)=\frac{|n-2(b+1)|}{2}.
\]
This completes the proof.
\end{proof}

\section{Identities for the Caffarelli--Kohn--Nirenberg inequalities on orthants: proofs of Theorems \ref{Th:Ckn+_identity} and \ref{Th:Ckn-_identity}}\label{S:identities}

The extremal functions obtained in the previous section point to the transformation under
which the deficit becomes transparent: one divides $u$ by the extremal profile and
multiplies by the exponential factor, and what is left is a weighted Dirichlet energy. In
this section we carry this out, and thereby prove the identities announced in the
introduction. Besides yielding yet another proof of the sharp inequality, they are the
analytic mechanism behind all of the stability estimates of Sections
\ref{S:poincare_stability} and \ref{S:CKN_stability}.

We work first with the scale non-invariant formulation, in which a free parameter
$\lambda>0$ is retained; the scale invariant identities of Theorems \ref{Th:Ckn+_identity}
and \ref{Th:Ckn-_identity} are then recovered by optimizing in $\lambda$. Throughout this
section we set
\[
\gamma(x)
=
\frac{|x|^{-a+b+1}}
{\lambda^2(-a+b+1)},
\qquad
w(x)
=
\frac{u(x)}
{\left(\prod_{i=n-k+1}^n x_i\right)|x|^m}.
\]

In terms of $w$ and $\gamma$, the function whose gradient appears on the right-hand side of
the identities is simply $we^{\gamma}$. We treat separately the two regimes
\[
-a+b+1>0
\qquad\text{and}\qquad
-a+b+1<0,
\]
which correspond to Theorems \ref{Th:CKNO+_identity} and \ref{Th:CKNO-_identity}, respectively.

 \begin{theorem} \label{Th:CKNO+_identity}
    Let $u \in S_{a,b}(\mathbb R^n_{k,+})$ and assume that $-a+b+1>0$. Let $C(n,a,b,k)$ and \(m\) be as in Theorem~\ref{Th:CKNIon_Orthant}. Then for any $\lambda>0$, the following identity holds:

    \begin{equation}
   \begin{aligned}
   \lambda^2\int_{\Rnkp}&\frac{|\nabla u|^2}{|x|^{2b}}dx + \frac{1}{\lambda^2}\int_{\Rnkp}\frac{u^2}{|x|^{2a}}dx - 2C(n,a,b,k)\int_{\Rnkp}\frac{u^2}{|x|^{a+b+1}}dx \\ 
       &= \lambda^2 \int_{\Rnkp}\left|\nabla\left(\frac{u(x)}{\pxi |x|^m} e^{\frac{|x|^{-a+b+1}}{\lambda^2(-a+b+1)}}\right)\right|^2 e^{-2\frac{|x|^{-a+b+1}}{\lambda^2(-a+b+1)}}\frac{\pxis |x|^{2m}}{|x|^{2b}}dx.
   \end{aligned}
\end{equation}

\end{theorem}
\begin{proof}
We expand the right-hand side, and identify the three terms that arise with the three terms
on the left. Since $\nabla(we^{\gamma})=e^{\gamma}(\nabla w+w\nabla\gamma)$, we have
\begin{equation}
    \begin{aligned}
        &\lambda^2\int_{\Rnkp}\left|\nabla(we^\gamma)\right|^2 e^{-2\gamma}\frac{\pxis |x|^{2m}}{|x|^{2b}}dx \\
        &=\lambda^2\int_{\Rnkp}\left(|\nabla w|^2+|w|^2|\nabla\gamma|^2+2w\nabla w\cdot\nabla\gamma\right) \frac{\pxis |x|^{2m}}{|x|^{2b}}dx.\\
        &=\lambda^2\int_{\Rnkp}|\nabla w|^2\frac{\pxis |x|^{2m}}{|x|^{2b}}dx  + \frac{1}{\lambda^2} \int_{\Rnkp} \frac{u^2}{|x|^{2a}}\,dx + \int_{\Rnkp} \nabla (w^2) \cdot x \; \frac{|x|^{2m} \pxis}{|x|^{a+b+1}} dx 
    \end{aligned}
\end{equation}

Here we have used $|\nabla\gamma|^{2}=\frac{|x|^{2(-a+b)}}{\lambda^{4}}$, which turns the
second term into the required $\frac{1}{\lambda^{2}}\int u^{2}|x|^{-2a}$. For the third
term, an integration by parts gives

\begin{align}
&\int \nabla\left|\!\left(\frac{u}{\prod x_i^2 |x|^m}\right)\right|^2 \cdot x
\frac{\prod x_i^2 |x|^{2m}}{|x|^{a+b+1}}\,dx  
\notag\\
\notag\\
&= -\int \frac{u^2}{\prod x_i^2 |x|^{2m}}
{\dive}\!\left(
x\,\prod x_i^2 \frac{|x|^{2m}}{|x|^{a+b+1}}
\right)\,dx
\notag\\
\notag\\
&= -\int \frac{u^2}{\prod x_i^2 |x|^{2m}}
\left[
n\left(\prod x_i^2 \frac{|x|^{2m}}{|x|^{a+b+1}}\right)
+2\pxis\left(0,\ldots 0, 
,\frac1{x_{n-k+1}},\ldots ,\frac1{x_n}\right)\cdot x\,\frac{|x|^{2m}}{|x|^{a+b+1}} \right. \notag\\
 & \qquad \qquad \qquad  \qquad \left.+\prod x_i^2 (2m-(a+b+1))\frac{|x|^{2m}}{|x|^{a+b+1}}
\right]dx
\notag\\
\notag\\
&=
-\int \frac{u^2}{\prod x_i^2 |x|^{2m}}
\left(\prod x_i^2 \frac{|x|^{2m}}{|x|^{a+b+1}}\right)
\left(n+2k+2m-(a+b+1)\right)\,dx
\notag\\
&=
-\left(n+2k+2m-(a+b+1)\right)
\int \frac{u^2}{|x|^{a+b+1}}\,dx
\notag\\
&=- 2C(n,a,b,k)\int \frac{u^2}{|x|^{a+b+1}}\,dx.
\end{align}

It remains to treat the first term, which is where the choice of $m$ enters.
\begin{align*}
    &\int_{\Rnkp}|\nabla w|^2\frac{\pxis |x|^{2m}}{|x|^{2b}}dx \\
    &=\int_{\Rnkp}\left|\nabla \left(\frac{u(x)}{\pxi |x|^m}\right)\right|^2\frac{\pxis |x|^{2m}}{|x|^{2b}}dx\\
    &= \int_{\Rnkp} \frac{\left|\nabla u\right|^2}{|x|^{2b}}dx+ \int_{\Rnkp} u^2\left|\nabla \left(\frac{1}{\pxi |x|^m}\right)\right|^2\frac{\pxis |x|^{2m}}{|x|^{2b}}dx \\
    & \qquad \qquad +\int_{\Rnkp}\nabla (u^2)\cdot \nabla \left(\frac{1}{\pxi |x|^m}\right){\frac{\pxi |x|^{m}}{|x|^{2b}}}dx
\end{align*}

We compute the second term on the right, 

\begin{align}
    &\int_{\Rnkp} u^2\left|\nabla \left(\frac{1}{\pxi |x|^m}\right)\right|^2\frac{\pxis |x|^{2m}}{|x|^{2b}}dx\notag \\
    & \qquad \qquad \qquad \qquad=\int_{\Rnkp} u^2\left( \sum_{n-k+1}^n \frac{1}{x_i^2}  +\frac{m^2+2km}{|x|^2}\right) |x|^{-2b} dx
\end{align}

and, integrating by parts, the cross term,
\begin{align}
        &\int_{\Rnkp}\nabla (u^2)\cdot \nabla \left(\frac{1}{\pxi |x|^m}\right)\frac{\pxi |x|^{m}}{|x|^{2b}}dx \notag\\
        &=-\int_{\Rnkp} (u^2) \dive\left( \nabla \left(\frac{1}{\pxi |x|^m}\right)\frac{\pxi |x|^{m}}{|x|^{2b}}\right)dx \notag\\
        &=-\int_{\Rnkp} (u^2) \dive\left( \left( \frac1{\pxi |x|^m}\right)\right.\notag\\ 
         &\qquad  \qquad \qquad\qquad\left.\left(-(0,\cdots,0, \frac1{x_{n-k+1}},\cdots,\frac1{x_n})-\frac{mx}{|x|^2} \right)\frac{\pxi |x|^{m}}{|x|^{2b}}\right)dx\notag\\
         &= -\int_{\Rnkp} (u^2) \dive\left( \left(-(0,\cdots,0, \frac1{x_{n-k+1}},\cdots,\frac1{x_n})-\frac{mx}{|x|^2} \right)\frac{1}{|x|^{2b}}\right)dx\notag\\
         &=-\int_{\Rnkp} (u^2) \dive \left(-(0,\cdots,0, \frac1{x_{n-k+1}},\cdots,\frac1{x_n})-\frac{mx}{|x|^2} \right)\frac{1}{|x|^{2b}}dx \notag \\
         &\qquad  \qquad \qquad\qquad -\int_{\Rnkp} (u^2)  \left(-(0,\cdots,0, \frac1{x_{n-k+1}},\cdots,\frac1{x_n})-\frac{mx}{|x|^2} \right)\cdot\frac{-2bx}{|x|^{2b+2}}dx\notag\\ 
          &=\int_{\Rnkp} \frac{(u^2)}{|x|^{2b}} \left(-\sum_{i=n-k+1}^n \frac{1}{x_i^2} + \frac{mn-2m}{|x|^2} \right) \, \frac{1}{|x|^{2b}}dx -2b\int_{\Rnkp} \frac{(u^2)}{|x|^{2b}}  \left( \frac{(k+m)}{|x|^2}\right)dx \notag\\ 
          &=\int_{\Rnkp} \frac{(u^2)}{|x|^{2b}} \left(-\sum_{i=n-k+1}^n \frac{1}{x_i^2} \right)dx +\int_{\Rnkp} \frac{(u^2)}{|x|^{2b}}  \left( \frac{mn-2m-2b(k+m)}{|x|^2}\right)dx\notag\\ 
\end{align}
Adding the three contributions, we obtain 

\begin{align}
    &\int_{\Rnkp}|\nabla w|^2\frac{\pxis |x|^{2m}}{|x|^{2b}}dx \notag\\
   &=\int_{\Rnkp}\left|\nabla \left(\frac{u(x)}{\pxi |x|^m}\right)\right|^2\frac{\pxis |x|^{2m}}{|x|^{2b}}dx\notag\\
    &= \int_{\Rnkp} \frac{\left|\nabla u\right|^2}{|x|^{2b}}dx+\int_{\Rnkp} \frac{(u^2)}{|x|^{2b}}  \left( \frac{mn-2m-2b(k+m) +m^2 +2km}{|x|^2}\right)dx.\notag\\
     &=\int_{\Rnkp} \frac{\left|\nabla u\right|^2}{|x|^{2b}}dx \label{eq:first_term_identity}
      \end{align}
where the last equality holds because $m$ has been chosen so that 
    $$(m+k)^2+(m+k)(n-2-2b)-k(n+k-2)=0.$$ 
    
    This implies    $\qquad \qquad mn-2m-2b(k+m) +m^2 +2km=0.$\\
Combining the three terms yields the asserted identity:
 \begin{align}
   &\lambda^2\int_{\Rnkp}\frac{|\nabla u|^2}{|x|^{2b}}dx + \frac{1}{\lambda^2}\int_{\Rnkp}\frac{u^2}{|x|^{2a}}dx - 2C(n,a,b,k)\int_{\Rnkp}\frac{u^2}{|x|^{a+b+1}}dx \notag \\ 
       &\quad= \lambda^2 \int_{\Rnkp}\left|\nabla\left(\frac{u(x)}{\pxi |x|^m} e^{\frac{|x|^{-a+b+1}}{\lambda^2(-a+b+1)}}\right)\right|^2 e^{-2\frac{|x|^{-a+b+1}}{\lambda^2(-a+b+1)}}\frac{\pxis |x|^{2m}}{|x|^{2b}}dx
\end{align}

\end{proof}

Taking $k=0$ in Theorem \ref{Th:CKNO+_identity} gives the corresponding identity on the whole space $\mathbb R^{n}$, in which the monomial factor disappears and $m$ reduces to the exponent attached to the regions $\mathcal A$ and $\mathcal B$ of Theorem \ref{Th:CKNI}.

\begin{corollary}[Euclidean identity, \(\tau>0\)]
Let $u\in S_{a,b}(\mathbb R^n)$ and assume that
$\tau:=-a+b+1>0$. Then

\begin{align}
&\lambda^2\int_{\mathbb R^n}\frac{|\nabla u|^2}{|x|^{2b}}\,dx
+\frac1{\lambda^2}\int_{\mathbb R^n}\frac{u^2}{|x|^{2a}}\,dx
-2C(n,a,b)\int_{\mathbb R^n}\frac{u^2}{|x|^{a+b+1}}\,dx\\
&\qquad \qquad \qquad \qquad =
\lambda^2\int_{\mathbb R^n}\left|\nabla\left(\frac{u(x)}{|x|^m}e^{\frac{|x|^\tau}{\lambda^2\tau}}
\right)
\right|^2
e^{-2\frac{|x|^\tau}{\lambda^2\tau}}
\frac{|x|^{2m}}{|x|^{2b}}\,dx .
\end{align}

Here $C(n,a,b)$ is as in Theorem \ref{Th:CKNI}, and

\[m=
\begin{cases}
0, & b\le \frac{n-2}{2},\\[4pt]
2(b+1)-n, & b\ge \frac{n-2}{2}.
\end{cases}
\]
\end{corollary}

\begin{proof}[Proof of Theorem \ref{Th:Ckn+_identity}]
 Optimizing over the parameter \(\lambda\) in Theorem~\ref{Th:CKNO+_identity} by choosing
\[
\lambda
=
\left(
\frac{
\int_{\Rnkp}\frac{u^2}{|x|^{2a}}\,dx
}{
\int_{\Rnkp}\frac{|\nabla u|^2}{|x|^{2b}}\,dx
}
\right)^{1/4},
\]
yields the corresponding scale-invariant Caffarelli--Kohn--Nirenberg identity. 
\end{proof}

We now turn to the case
$-a+b+1<0.$
Although the argument is similar to the previous case, the change in sign of the exponent leads to different admissibility conditions and a different extremal profile. We therefore repeat the computation and derive the corresponding exact identity.
\begin{theorem} \label{Th:CKNO-_identity}
    Let $u(x) \in S_{a,b}(\mathbb R^n_{k,+})$ and assume that $-a+b+1<0$. Let $C(n,a,b,k)$ and \(m\) be as in Theorem~\ref{Th:CKNIon_Orthant}. Then for any $\lambda>0$, the following identity holds:
    \begin{equation}
   \begin{aligned}
   \lambda^2\int_{\Rnkp}&\frac{|\nabla u|^2}{|x|^{2b}}dx + \frac{1}{\lambda^2}\int_{\Rnkp}\frac{u^2}{|x|^{2a}}dx - 2C(n,a,b,k)\int_{\Rnkp}\frac{u^2}{|x|^{a+b+1}}dx \\ 
       &= \lambda^2 \int_{\Rnkp}\left|\nabla\left(\frac{u(x)}{\pxi |x|^m} e^{-\frac{|x|^{-a+b+1}}{\lambda^2(-a+b+1)}}\right)\right|^2 e^{2\frac{|x|^{-a+b+1}}{\lambda^2(-a+b+1)}}\frac{\pxis |x|^{2m}}{|x|^{2b}}dx.
   \end{aligned}
\end{equation}

\end{theorem}
\begin{proof}
Expand the gradient:
\[
\nabla\bigl(we^{-\gamma}\bigr)=e^{-\gamma}\bigl(\nabla w - w\,\nabla\gamma\bigr).
\]

\begin{equation}
    \begin{aligned}
        &\lambda^2\int_{\Rnkp}\left|\nabla(we^{-\gamma})\right|^2 e^{2\gamma}\frac{\pxis r^{2m}}{r^{2b}}dx \\
        &=\lambda^2\int_{\Rnkp}\left(|\nabla w|^2+|w|^2|\nabla\gamma|^2 - 2w\nabla w\cdot\nabla\gamma\right) \frac{\pxis r^{2m}}{r^{2b}}dx\\
        &=\lambda^2\int_{\Rnkp}|\nabla w|^2\frac{\pxis r^{2m}}{r^{2b}}dx  + \frac{1}{\lambda^2} \int_{\Rnkp} \frac{u^2}{r^{2a}}\,dx - \int_{\Rnkp} \nabla (w^2) \cdot x \; \frac{r^{2m} \pxis}{r^{a+b+1}} dx 
    \end{aligned}
\end{equation}

The middle term is already in the required form, and, by \eqref{eq:first_term_identity}, the first term is
\begin{align}
    \int_{\Rnkp}|\nabla w|^2\frac{\pxis r^{2m}}{r^{2b}}dx &=\int_{\Rnkp}\left|\nabla \left(\frac{u(x)}{\pxi r^m}\right)\right|^2\frac{\pxis r^{2m}}{r^{2b}}dx \notag\\
    &= \int_{\Rnkp} \frac{\left|\nabla u\right|^2}{r^{2b}}dx
\end{align}

We now turn to the third term. Integrating by parts, and using
$\dive(x\,\rho(x))=n\rho(x)+x\cdot\nabla\rho(x)$, we obtain
\begin{align}
&-\int \nabla\left|\!\left(\frac{u}{\prod x_i^2 r^m}\right)\right|^2 \cdot x
\frac{\prod x_i^2 r^{2m}}{r^{a+b+1}}\,dx  \notag\\
\notag\\
&= \int \frac{u^2}{\prod x_i^2 r^{2m}}
\dive\!\left(
x\,\prod x_i^2 \frac{r^{2m}}{r^{a+b+1}}
\right)\,dx
\notag\\
\notag\\
&= \int \frac{u^2}{\prod x_i^2 r^{2m}}
\left[
n\left(\prod x_i^2 \frac{r^{2m}}{r^{a+b+1}}\right)
+2\pxis\left(0,\ldots ,\frac1{x_{n-k+1}},\ldots ,\frac1{x_n}\right)\cdot x\,\frac{r^{2m}}{r^{a+b+1}} \right.\notag\\
 & \qquad \qquad \qquad  \qquad \left.+\prod x_i^2 (2m-(a+b+1))\frac{r^{2m}}{r^{a+b+1}}
\right]dx
\notag\\
\notag\\
&=
\int \frac{u^2}{\prod x_i^2 r^{2m}}
\left(\prod x_i^2 \frac{r^{2m}}{r^{a+b+1}}\right)
\left(n+2k+2m-(a+b+1)\right)\,dx
\notag\\
\notag\\
&=
\left(n+2k+2m-(a+b+1)\right)
\int \frac{u^2}{r^{a+b+1}}\,dx
\notag\\
\notag\\
&=
-\left|-a+b+1 -\sqrt{(n-2-2b)^2+4k(n+k-2)}\right|
\int \frac{u^2}{r^{a+b+1}}\,dx.\notag\\
&=-2C(n,a,b,k)\int \frac{u^2}{r^{a+b+1}}\,dx.
\end{align}

Combining the three terms yields the asserted identity. Note that the computation is the
same as in the case $\tau>0$, apart from the sign of the exponent in the exponential factor,
which is why the two identities can be stated together by writing $|\tau|$ in place of
$\tau$ in the exponentials.

\end{proof}

\begin{corollary}[Euclidean identity, \(\tau<0\)]
Let \(u\in S_{a,b}(\mathbb R^n)\) and assume
\(\tau:=-a+b+1<0\). Then
\[
\lambda^2\int_{\mathbb R^n}\frac{|\nabla u|^2}{|x|^{2b}}\,dx
+
\frac1{\lambda^2}\int_{\mathbb R^n}\frac{u^2}{|x|^{2a}}\,dx
-
2C(n,a,b)\int_{\mathbb R^n}\frac{u^2}{|x|^{a+b+1}}\,dx
\]
\[
=
\lambda^2
\int_{\mathbb R^n}
\left|
\nabla\left(
\frac{u(x)}{|x|^m}
e^{\frac{|x|^\tau}{\lambda^2\tau}}
\right)
\right|^2
e^{-2\frac{|x|^\tau}{\lambda^2\tau}}
\frac{|x|^{2m}}{|x|^{2b}}\,dx .
\]
Here \(C(n,a,b)\) is as defined in Theorem~\ref{Th:CKNI}, and
\[
m=
\begin{cases}
0, & b\ge \frac{n-2}{2},\\[4pt]
2(b+1)-n, & b\le \frac{n-2}{2}.
\end{cases}
\]
\end{corollary}

\begin{proof}[Proof of Theorem \ref{Th:Ckn-_identity}]
   Optimizing the identity in Theorem \ref{Th:CKNO-_identity} with respect to \(\lambda\), that is by choosing \[
\lambda
=
\left(
\frac{
\int_{\Rnkp}\frac{u^2}{|x|^{2a}}\,dx
}{
\int_{\Rnkp}\frac{|\nabla u|^2}{|x|^{2b}}\,dx
}
\right)^{1/4},
\] we obtain the scale-invariant identity in the case \(\tau<0\). 
\end{proof}

This completes the derivation of the exact identities associated with the Caffarelli--Kohn--Nirenberg inequalities on orthants. 


\section{Weighted Poincar\'e inequalities of generalized Gaussian type: proof of Theorem \ref{Th:poincare_monomial_weight}}\label{S:poincare}

In this section we establish the weighted Poincar\'e inequalities of generalized Gaussian
type. We proceed in two steps. We first treat the case of a radial measure on the whole
space $\mathbb R^{N}$, where the spherical harmonic decomposition reduces the problem to
the one-dimensional estimate of Theorem \ref{th:general_radial_poincare}, applied
separately to each mode. We then pass to monomial weights by the lifting argument, and
thereby prove Theorem \ref{Th:poincare_monomial_weight}. It is worth noting that the
resulting constant is always the smaller of two competing quantities, one coming from the
radial mode and one from the first non-radial mode; which of the two prevails depends on
the parameters, and this dichotomy will reappear in the stability estimates of Section
\ref{S:poincare_stability}.

{Throughout this section, we work with the measures
\[
d\mu(x):=e^{-\delta|x|^\tau}|x|^\beta\,dx,
\qquad
d\nu(x):=e^{-\delta|x|^\tau}|x|^{\beta+\tau-2}\,dx,
\]
and the associated weighted energy norm
\[
\|u\|_W^2
:=
\int_{\mathbb R^n}|\nabla u|^2\,d\mu
+
\int_{\mathbb R^n}|u|^2\,d\nu.
\]
Let
\[
\mathcal A_\tau
:=
\left\{
u\in C_c^\infty(\mathbb R^n):
\|u\|_W<\infty
\right\},
\qquad
W:=\overline{\mathcal A_\tau}^{\,\|\cdot\|_W}.
\]
By the extension result proved in Lemma~\ref{lem:extension-poincare},
it suffices to establish the weighted Poincaré inequality for
$u\in C_c^\infty(\mathbb R^n\setminus\{0\})$; the inequality then
extends to every $u\in\mathcal A_\tau$, and consequently to every
$u\in W$.}

Note that the origin in Theorems
\ref{Th:poincare_monomial_weight} and \ref{Th:stability_poincare_monomialweight}: stating
them for $u\in C_{c}^{\infty}(\mathbb R^{n}_{A,+}\setminus\{0\})$is a genuine restriction only when $A=0$, since for $A\neq0$ the origin
already lies on the boundary of $\mathbb R^{n}_{A,+}$, so that
$C_{c}^{\infty}(\mathbb R^{n}_{A,+}\setminus\{0\})=C_{c}^{\infty}(\mathbb R^{n}_{A,+})$.

\begin{theorem}[Weighted Poincar\'e inequality]\label{Th:general_poincare}
Let {$n\ge2$}, let $\delta>0$, and let
$u \in  W$.
Assume that
\[
\frac{\beta+n+\tau-2}{\tau}>0 .
\]
Then
\begin{equation}
    \int_{\mathbb{R}^n} |\nabla u(x)|^2\,
e^{-\delta|x|^\tau}|x|^\beta\,dx
\ge
C
\inf_{d\in\mathbb R}
\int_{\mathbb{R}^n}|u(x)-d|^2\,
e^{-\delta|x|^\tau}|x|^{\beta+\tau-2}\,dx,
\end{equation}
where the sharp constant is given by $C=\min\{\delta\tau^2,\; \delta\tau s_1\}$, with
\[
s_1=\frac{-(\beta+n-2)+\sgn(\tau)\sqrt{(\beta+n-2)^2+4(n-1)}}{2},
\]
that is, $s_1$ is the root of $s_1^{2}+s_1(\beta+n-2)-(n-1)=0$ for which $\delta\tau s_1>0$. More
clearly, the constant $\delta\tau^{2}$ is the one produced by the radial mode, while
$\delta\tau s_1$ is the one produced by the first non-radial mode.
\end{theorem}

\begin{proof}
Let
$u\in C_c^\infty(\mathbb R^n\setminus\{0\})$; note that subtracting a constant changes
neither side. We expand $u$ into spherical harmonics,
\[
u(r,\sigma)=\sum_{\ell=0}^\infty f_\ell(r)\,\varphi_\ell(\sigma),
\]
where, as in Section \ref{S:preliminaries}, the index $\ell$ runs over a fixed orthonormal
basis of $L^{2}(\mathbb S^{n-1})$ consisting of spherical harmonics, and $\lambda_\ell$
denotes the eigenvalue attached to the degree of $\varphi_\ell$. Since $n\ge2$, non-radial
modes are present, and the smallest non-zero eigenvalue is $\lambda_1=n-1$, with an
eigenspace of dimension $n$ spanned by the coordinate functions.

Then from equations  (\ref{eq:dumu}) and (\ref{eq:umu})
\begin{align}
\int_{\mathbb R^n} |\nabla u|^2 e^{-\delta|x|^\tau}|x|^\beta\,dx
&=\sum_{\ell=0}^\infty
\int_0^\infty
\left(
|f_\ell'(r)|^2+\frac{\lambda_\ell}{r^2}f_\ell(r)^2
\right)
e^{-\delta r^\tau}r^{\beta+n-1}\,dr,
\label{eq:energy-decomp}
\end{align}
and similarly
\begin{align}
&\int_{\mathbb R^n}|u-d|^2 e^{-\delta|x|^\tau}|x|^{\beta+\tau-2}\,dx
=\int_0^\infty |f_0(r)-d|^2 e^{-\delta r^\tau}r^{\beta+\tau+n-3}\,dr
\notag\\
&\quad \qquad \qquad \qquad \qquad \qquad \qquad \qquad \qquad
+\sum_{\ell=1}^\infty
\int_0^\infty |f_\ell(r)|^2 e^{-\delta r^\tau}r^{\beta+\tau+n-3}\,dr.
\label{eq:rhs-decomp}
\end{align}

Thus the problem decouples into the radial mode $\ell=0$ and the non-radial modes $\ell\ge1$.

\medskip
\noindent

For $\ell=0$, we have $\lambda_0=0$, so the contribution to \eqref{eq:energy-decomp} is
\[
\int_0^\infty |f_0'(r)|^2 e^{-\delta r^\tau}r^{\beta+n-1}\,dr.
\]

Now apply the one-dimensional radial Poincar\'e inequality from Theorem
\ref{th:general_radial_poincare} with
\[
\widetilde\beta=\beta+n-1.
\]
This gives
\begin{equation}
\int_0^\infty |f_0'(r)|^2 e^{-\delta r^\tau}r^{\beta+n-1}\,dr
\ge
\delta\tau^2
\inf_{d\in\mathbb R}
\int_0^\infty |f_0(r)-d|^2 e^{-\delta r^\tau}r^{\beta+\tau+n-3}\,dr.
\end{equation}
Hence the optimal constant for the radial mode is $\delta\tau^2.$

\medskip
\noindent

    Now for the non-radial part let
\begin{align*}
        &f_\ell(r)=r^{s_1} g_\ell(r), \qquad \ell\ge1\\
        \implies &f_\ell'(r)=s_1r^{s_1-1}g_\ell(r)+r^{s_1} g_\ell'(r),\\
        \implies &|f_\ell'(r)|^2=s_1^2r^{2s_1-2}g_\ell(r)^2+r^{2s_1}|g_\ell'(r)|^2+2s_1r^{2s_1-1}g_\ell(r)g_\ell'(r).
\end{align*}
Using the spherical harmonic decomposition and
\(\lambda_\ell=\ell(\ell+n-2)\), we have
\begin{equation}
\begin{aligned}
\sum_{\ell=1}^\infty \int_0^\infty
&\left(|f_\ell'(r)|^2+\frac{\lambda_\ell}{r^2}f_\ell(r)^2
\right)
e^{-\delta r^\tau}r^{\beta+n-1}\,dr \ge
\sum_{\ell=1}^\infty \int_0^\infty
\left(
|f_\ell'(r)|^2+\frac{\lambda_1}{r^2}f_\ell(r)^2
\right)
e^{-\delta r^\tau}r^{\beta+n-1}\,dr,\\
&=
\sum_{\ell=1}^\infty \int_0^\infty
\Big(
r^{2s_1}|g_\ell'(r)|^2
+\big(s_1^2+\lambda_1\big)r^{2s_1-2}g_\ell(r)^2
+s_1r^{2s_1-1}(g_\ell^2(r))'
\Big)
e^{-\delta r^\tau}r^{\beta+n-1}\,dr \\
&=
\sum_{\ell=1}^\infty \int_0^\infty
r^{2s_1}|g_\ell'(r)|^2 e^{-\delta r^\tau}r^{\beta+n-1}\,dr
+\big(s_1^2+\lambda_1\big)\sum_{\ell=1}^\infty \int_0^\infty
g_\ell(r)^2 e^{-\delta r^\tau}r^{\beta+n+2s_1-3}\,dr \\
&\quad \qquad \quad \qquad \quad \qquad \quad \qquad \quad \qquad \quad \qquad
-s_1\sum_{\ell=1}^\infty \int_0^\infty
(g_\ell^2(r))(e^{-\delta r^\tau}r^{\beta+n+2s_1-2})'\,dr.
\end{aligned}
\end{equation}
Now integrate by parts in the last term:
\begin{equation}
\begin{aligned}
&-s_1\sum_{\ell=1}^\infty \int_0^\infty
g_\ell(r)^2\,
\frac{d}{dr}\Big(e^{-\delta r^\tau}r^{\beta+n+2s_1-2}\Big)\,dr =
-s_1(\beta+n+2s_1-2)\sum_{\ell=1}^\infty \int_0^\infty
g_\ell(r)^2 e^{-\delta r^\tau}r^{\beta+n+2s_1-3}\,dr \\
&\qquad \qquad \qquad \qquad \qquad \qquad \qquad  \qquad \qquad \qquad \qquad\quad
+\delta\tau s_1\sum_{\ell=1}^\infty \int_0^\infty
g_\ell(r)^2 e^{-\delta r^\tau}r^{\beta+n+2s_1+\tau-3}\,dr.
\end{aligned}
\end{equation}
Substituting this back, we obtain
\begin{equation}
\begin{aligned}
\sum_{\ell=1}^\infty \int_0^\infty
\left(
|f_\ell'(r)|^2+\frac{\lambda_\ell}{r^2}f_\ell(r)^2
\right)&
e^{-\delta r^\tau}r^{\beta+n-1}\,dr \geq
\sum_{\ell=1}^\infty \int_0^\infty
r^{2s_1}|g_\ell'(r)|^2 e^{-\delta r^\tau}r^{\beta+n-1}\,dr \\
&\quad
+\Big(s_1^2+\lambda_1-s_1(\beta+n+2s_1-2)\Big)
\sum_{\ell=1}^\infty \int_0^\infty
g_\ell(r)^2 e^{-\delta r^\tau}r^{\beta+n+2s_1-3}\,dr \\
&\quad
+\delta\tau s_1\sum_{\ell=1}^\infty \int_0^\infty
g_\ell(r)^2 e^{-\delta r^\tau}r^{\beta+n+2s_1+\tau-3}\,dr.
\end{aligned}
\end{equation}
We now choose $s_1$ so that the middle term drops out, that is, so that
\[
s_1^{2}+\lambda_1-s_1(\beta+n+2s-2)=0,
\qquad\text{equivalently}\qquad
s_1^{2}+s_1(\beta+n-2)-(n-1)=0 ,
\]
and, among the two roots, we take the one for which $\delta\tau s_1>0$. Since the product of
the roots equals $-(n-1)<0$, exactly one root is positive and one is negative, so this
selection is unambiguous, and it gives
\[
s_1=\frac{-(\beta+n-2)+\sgn(\tau)\sqrt{(\beta+n-2)^2+4(n-1)}}{2}.
\]
Hence
\begin{align}
&\sum_{\ell=1}^\infty \int_0^\infty\left(|f_\ell'(r)|^2+\frac{\lambda_\ell}{r^2}f_\ell(r)^2\right)e^{-\delta r^\tau}r^{\beta+n-1}\,dr \notag\\ 
&\qquad \qquad\geq \sum_{\ell=1}^\infty \int_0^\infty r^{2s_1}|g_\ell'(r)|^2 e^{-\delta r^\tau}r^{\beta+n-1}\,dr +\delta\tau s_1\sum_{\ell=1}^\infty \int_0^\infty
g_\ell(r)^2 e^{-\delta r^\tau}r^{\beta+n+2s_1+\tau-3}\,dr. \notag \\
&\qquad \qquad\geq \delta\tau s_1\sum_{\ell=1}^\infty \int_0^\infty
g_\ell(r)^2 e^{-\delta r^\tau}r^{\beta+n+2s_1+\tau-3}\,dr.\notag \\
&\qquad \qquad =\delta\tau s_1 \sum_{\ell=1}^\infty \int_0^\infty |f_\ell(r)|^2 e^{-\delta r^\tau}r^{\beta+n+\tau-3}\,dr.
\end{align}

Equality occurs if and only if every one of the intermediate inequalities is an equality.
The first of them was obtained by replacing $\lambda_\ell$ by $\lambda_1$, which forces the
vanishing of all the modes of degree at least two; the second was obtained by discarding
the term containing $|g_\ell'|^2$, which forces the remaining $g_\ell$ to be constant. In
other words,
\[
g_\ell(r)=
\begin{cases}
0, &\text{if the degree of }\varphi_\ell\text{ is at least }2,\\
A_\ell\in\R, &\text{if the degree of }\varphi_\ell\text{ is }1,
\end{cases}
\qquad\text{equivalently}\qquad
f_\ell(r)=
\begin{cases}
0,\\
A_\ell r^{s_1}.
\end{cases}
\]
Since the eigenspace of $\lambda_1$ is $n$-dimensional, the corresponding extremal
functions are exactly the functions $u(x)=d+\bigl(\sum_{i=1}^{n}A_ix_i\bigr)|x|^{s_1-1}$, and
for these the constant $\delta\tau s_1$ is attained.

Combining the radial estimate and all non-radial estimates in
\eqref{eq:energy-decomp}--\eqref{eq:rhs-decomp}, we obtain
\[
\int_{\mathbb R^n} |\nabla u|^2 e^{-\delta|x|^\tau}|x|^\beta\,dx
\ge
\min\{\delta\tau^2,\; \delta\tau s_1\}
\inf_{d\in\mathbb R}
\int_{\mathbb R^n}|u-d|^2e^{-\delta|x|^\tau}|x|^{\beta+\tau-2}\,dx.
\]
This proves the result.
\end{proof}

The proof also identifies the equality cases, which we now record.

\begin{corollary}[Equality case]\label{C:equality_general_poincare}
Under the assumptions of the weighted Poincar\'e inequality, equality holds if and only if one of the following occurs.

\medskip
\noindent
\begin{enumerate}
\item[(i)] \textbf{Radial case:} If $\delta\tau^2<\delta\tau s_1,$
then equality holds if and only if
\[
u(x)=a_0+a_1|x|^\tau.
\]
where $a_0,a_1\in\mathbb{R}$.
\item[(ii)] \textbf{First non-radial case:} If $\delta\tau s_1<\delta\tau^2,$
then equality holds if and only if 
\[
u(x)=A\cdot x\,|x|^{\,s_1-1} +a
\]
for some \(A\in\mathbb R^n\), and $a\in \R$

\item[(iii)] \textbf{Balanced case:} If $\delta\tau^2=\delta\tau s_1,$
then equality holds if and only if
\[
u(x)=A\cdot x\,|x|^{\,s_1-1}+a_0+a_1|x|^{\tau},
\]
where $A\in\mathbb{R}^n$ and $a_0,a_1\in\mathbb{R}$.
\end{enumerate}
\end{corollary}

\begin{proof}
Let
\[
u(r\sigma)=\sum_{\ell=0}^{\infty} f_\ell(r)\varphi_\ell(\sigma)
\]
be the spherical harmonic expansion of \(u\), where \(\varphi_0\) is constant and the eigenspace corresponding to \(\lambda_1=n-1\) is spanned by the restrictions to $\mathbb S^{n-1}$ of the coordinate functions $x\mapsto x_i$, $1\le i\le n$. From the proof of the preceding theorem, the Poincaré defect decomposes into the sum of a radial contribution and the non-radial contributions, and each term is nonnegative. Hence equality can occur only when all non-sharp components vanish and the remaining components are extremals for the corresponding one-dimensional inequalities.

If \(\delta\tau^2<\delta{\tau s_1}\), then the sharp constant comes only from the radial part. Therefore all non-radial components must vanish, and Corollary \ref{cor:general_radial_poincare_equality} gives
\[
f_0(r)=a_0+a_1 r^\tau .
\]
Hence equality holds if and only if
\[
u(x)=a_0+a_1 |x|^\tau ,
\qquad a_0,a_1\in\mathbb R .
\]

If \(\delta\tau s_1<\delta\tau^2\), then the sharp constant comes only from the first non-radial mode. Thus the radial nonconstant part and all modes with \(\ell\ge2\) must vanish. The equality case in the first non-radial inequality gives
\[
f_1(r)=A r^{s_1}
\]
in the \(\lambda_1=n-1\) eigenspace. Since the spherical harmonics of degree one are exactly the linear forms $\sigma\mapsto A\cdot\sigma$, this gives
\[
u(r\sigma)=a+A\cdot \sigma\, r^{s_1}
=
a+A\cdot x\, |x|^{s_1-1},
\]
for some \(A\in\mathbb R^n\) and \(a\in\mathbb R\). Hence equality holds exactly for functions of this form.

Finally, if \(\delta\tau^2=\delta{\tau s_1}\), then both the radial extremal and the first non-radial extremal give equality. Since the defect is a sum of nonnegative terms, equality is possible only for the sum of these two extremal components, together with constants. Therefore
\[
u(x)=A\cdot x\,|x|^{s_1-1}+a_0+a_1|x|^\tau,
\]
where \(A\in\mathbb R^n\) and \(a_0,a_1\in\mathbb R\). Conversely, each of the above functions gives equality by the equality cases in the preceding theorem. This proves the result.
\end{proof}

We close this part with the observation that the classical Gaussian Poincar\'e inequality
is the particular case $\delta=\frac12$, $\tau=2$, $\beta=0$ of Theorem
\ref{Th:general_poincare}.

\begin{corollary}[Classical Gaussian Poincar\'e inequality]
For every $u \in W^{1,2}(\mathbb{R}^n,\gamma_n)$, we have
\begin{equation}
\int_{\mathbb{R}^n} |\nabla u(x)|^2 \, d\gamma_n(x)
\;\ge\;
\int_{\mathbb{R}^n} \left|u(x)-\int_{\mathbb{R}^n} u \, d\gamma_n\right|^2 d\gamma_n(x),
\end{equation}
where
\[
d\gamma_n(x)=\frac{1}{(2\pi)^{n/2}}e^{-|x|^2/2}\,dx
\]
is the standard Gaussian measure on $\mathbb{R}^n$.
Moreover, the constant $1$ is sharp, and equality holds for affine functions
\[
u(x)=a+b\cdot x,
\qquad a\in\mathbb{R},\; b\in\mathbb{R}^n.
\]
    
\end{corollary}

\begin{proof}
We apply Theorem \ref{Th:general_poincare} with
\[
\delta=\frac12,\qquad \tau=2,\qquad \beta=0.
\]
Then
\[
\beta+n+\tau-2 = n>0,
\]
so the assumptions of Theorem \ref{Th:general_poincare} are satisfied. Moreover,
\[
s_1=\frac{-(n-2)+\sqrt{(n-2)^2+4(n-1)}}{2}
=\frac{-(n-2)+n}{2}=1.
\]
Hence
\[
C=\min\{\delta\tau^2,\delta\tau s_1\}
=\min\{2,1\}=1.
\]
Therefore,
\[
\int_{\mathbb{R}^n} |\nabla u|^2\,d\gamma_n
\geq
\int_{\mathbb{R}^n}
\left|u-\int_{\mathbb{R}^n}u\,d\gamma_n\right|^2 d\gamma_n .
\]

Finally, since the sharp constant comes from the first non-radial mode, the equality case follows from Corollary \ref{C:equality_general_poincare}. Thus equality holds precisely for the affine functions
\[
u(x)=a+b\cdot x,
\qquad a\in\mathbb{R},\quad b\in\mathbb{R}^n.
\]
\end{proof}

We are now ready to prove Theorem \ref{Th:poincare_monomial_weight}.

\begin{proof}[Proof of Theorem \ref{Th:poincare_monomial_weight}]

Let $u\in C_c^\infty(\mathbb R^n_{A,+})$. Set $N:=n+|A|.$ We lift $u$ to a function on
\[
\mathbb R^{a_1+1}\times\cdots\times \mathbb R^{a_n+1}\cong \mathbb R^N
\]
by defining
\[
\widetilde u(y_1,\dots,y_n):=u(x_1,\dots,x_n),
\]
where
\[
x_i=
\begin{cases}
|y_i|, & a_i>0,\\
y_i, & a_i=0,
\end{cases}
\qquad y_i\in \mathbb R^{a_i+1}.
\]
Since $u\in C_c^\infty(\mathbb R^n_{A,+})$ and vanishes near the boundary, it follows that
\[
\widetilde u\in C_c^\infty(\mathbb R^N).
\]
Moreover, by construction,
\[
|x|^2=\sum_{i=1}^n x_i^2
=
\sum_{i=1}^n |y_i|^2
=
|y|^2.
\]
Now apply the weighted Poincar\'e inequality in $\mathbb R^N$:
\begin{align*}
\int_{\mathbb R^N} |\nabla \widetilde u(y)|^2
e^{-\delta|y|^\tau}|y|^\beta\,dy
\;\ge\;
C(N,\beta,\tau,\delta)
\inf_{d\in\mathbb R}
\int_{\mathbb R^N} |\widetilde u(y)-d|^2
e^{-\delta|y|^\tau}|y|^{\beta+\tau-2}\,dy,
\end{align*}
where $C(N,\beta,\tau,\delta)$ is the sharp constant from Theorem \ref{Th:general_poincare}.

\medskip
We now identify both sides with the corresponding integrals on $\mathbb R^n_{A,+}$.

For each $i$ with $a_i>0$, writing
\[
y_i=(y_{i,1},\dots,y_{i,a_i+1}),
\qquad
r_i:=|y_i|=x_i,
\]
we have
\[
\frac{\partial \widetilde u}{\partial y_{i,j}}
=
\frac{\partial u(x)}{\partial x_i}\,\frac{y_{i,j}}{r_i},
\qquad j=1,\dots,a_i+1.
\]
Therefore
\[
|\nabla_{y_i}\widetilde u|^2
=
\sum_{j=1}^{a_i+1}
\left(\frac{\partial u(x)}{\partial x_i}\frac{y_{i,j}}{r_i}\right)^2
=
\left(\frac{\partial u(x)}{\partial x_i}\right)^2
\sum_{j=1}^{a_i+1}\frac{y_{i,j}^2}{r_i^2}
=
\left(\frac{\partial u(x)}{\partial x_i}\right)^2.
\]
If $a_i=0$, then $x_i=y_i$ and hence
\[
\frac{\partial \widetilde u}{\partial y_i}
=
\frac{\partial u}{\partial x_i}.
\]
Combining all blocks, we obtain
\[
|\nabla \widetilde u(y)|^2=|\nabla u(x)|^2.
\]
Next, using spherical coordinates in each block $\mathbb R^{a_i+1}$, we get
\[
dy_i=|\mathbb{S}^{a_i}|\,x_i^{a_i}\,dx_i.
\]

Since $|y|=|x|$, it follows that
\begin{align*}
\int_{\mathbb R^N} |\nabla \widetilde u(y)|^2 e^{-\delta|y|^\tau}|y|^\beta\,dy
&=
\prod_{\{i:\,a_i>0\}} |\mathbb S^{a_i}|
\int_{\mathbb R^n_{A,+}} |\nabla u(x)|^2
x^A e^{-\delta|x|^\tau}|x|^\beta\,dx,
\end{align*}
and similarly
\begin{align*}
\int_{\mathbb R^N} |\widetilde u(y)-d|^2 e^{-\delta|y|^\tau}|y|^{\beta+\tau-2}\,dy
&=
\prod_{\{i:\,a_i>0\}} |\mathbb S^{a_i}|
\int_{\mathbb R^n_{A,+}} |u(x)-d|^2
x^A e^{-\delta|x|^\tau}|x|^{\beta+\tau-2}\,dx.
\end{align*}

This leads to 
\begin{align*}
\int_{\mathbb R^n_{A,+}} |\nabla u(x)|^2
x^A e^{-\delta|x|^\tau}|x|^\beta\,dx
\;\ge\;
C(N,\beta,\tau,\delta)\inf_{d\in\mathbb R}
\int_{\mathbb R^n_{A,+}} |u(x)-d|^2
x^A e^{-\delta|x|^\tau}|x|^{\beta+\tau-2}\,dx.
\end{align*}
This is the asserted inequality, with $C=\min\{\delta\tau^{2},\delta\tau s_1\}$ and $s_1$ as in
the statement, since the parameters of Theorem \ref{Th:general_poincare} in dimension $N$
are exactly $\beta$, $\tau$, $\delta$.

It remains to discuss the sharpness. The lifting is not surjective onto
$C_c^\infty(\mathbb R^{N})$: its image consists precisely of the functions that are radial
in each block $\mathbb R^{a_i+1}$, and a competitor in $\mathbb R^{N}$ can be pushed back
down to $\mathbb R^{n}_{A,+}$ only if it has this block-invariance. Now the two extremal
families of Corollary \ref{C:equality_general_poincare} behave differently in this respect.
The radial family $d+c|y|^{\tau}$ depends on $y$ only through $|y|=|x|$, hence is always
block-invariant, and it descends to $u(x)=d+c|x|^{\tau}$; consequently, whenever
$\delta\tau^{2}\le\delta\tau s_1$ the constant $C$ is attained and is therefore sharp. The
first non-radial family $d+B\cdot y\,|y|^{s_1-1}$, on the other hand, is block-invariant only
if the linear form $B\cdot y$ is; and $y\mapsto y_{i,j}$ is invariant under the rotations of
the block $\mathbb R^{a_i+1}$ exactly when that block is one-dimensional, that is, when
$a_i=0$. Thus, if $a_i=0$ for at least one index $i$, the function
$u(x)=d+\bigl(\sum_{i:\,a_i=0}d_ix_i\bigr)|x|^{s_1-1}$ is an admissible competitor realizing
equality, and $C$ is again sharp. This proves the sharpness assertions of the statement.

Finally, if $a_i>0$ for every $i$, no linear form survives the lifting, and the first
admissible non-radial mode is the two-dimensional one described in Remark
\ref{rmk:sharpness_monomial}: the block-invariant harmonics of degree two in $\mathbb R^{N}$
are spanned by
\[
y\longmapsto\sum_{i=1}^{n}c_i\Bigl(|y_i|^{2}-\frac{a_i+1}{N}|y|^{2}\Bigr),
\qquad
\sum_{i=1}^{n}c_i(a_i+1)=0 ,
\]
which is a nonzero space as soon as $n\ge2$, and which corresponds to the eigenvalue
$\lambda_{2}=2N$. Repeating the argument of Theorem \ref{Th:general_poincare} with
$\lambda_{2}$ in place of $\lambda_{1}$ replaces $s_1$ by the root $s_{2}$ of
{$s^{2}+s(\beta+N-2)-2N=0$} with $\delta\tau s_{2}>0$, and shows that in this degenerate case
the optimal constant is $\min\{\delta\tau^{2},\delta\tau s_{2}\}$.
Under the lifting $x_i=|y_i|$, the degree-two extremals
descend to
\[
u(x)
=
a_0
+
|x|^{s_2-2}
\sum_{i=1}^n
c_i
\left(
x_i^2-\frac{a_i+1}{N}|x|^2
\right).
\]
Equivalently, the non-radial part may be written as
\[
|x|^{s_2-2}\sum_{i=1}^n b_i x_i^2,
\qquad
\sum_{i=1}^n b_i(a_i+1)=0.
\]
\end{proof}

\begin{corollary}[Equality cases]\label{C:equality_monomial}
Under the assumptions of {Theorem~\ref{Th:poincare_monomial_weight}}, let $s_1$ be the root of
\[
s_1^2+(\beta+N-2)s_1-(N-1)=0
\]
satisfying $\tau s_1>0$.

If $a_i>0$ for every $i$ and $n\ge 2$, let $s_2$ be the root of
\[
s_2^2+(\beta+N-2)s_2-2N=0
\]
satisfying $\tau s_2>0$.

Then equality in \eqref{eq:poincare_monomial} holds if and only if one of the following
cases occurs.

\begin{enumerate}
\item[(i)] \textbf{Radial case.}
Suppose either
\[
a_i=0 \quad \text{for at least one } i,
\qquad
|\tau|<|s_1|,
\]
or
\[
a_i>0 \quad \text{for every } i,
\qquad
|\tau|<|s_2|.
\]
Then equality holds if and only if
\[
u(x)=b_0+b_1|x|^\tau,
\qquad
b_0,b_1\in\mathbb R.
\]

\item[(ii)] \textbf{First non-radial case.}
Suppose
\[
a_i=0 \quad \text{for at least one } i,
\qquad
|s_1|<|\tau|.
\]
Then equality holds if and only if
\[
u(x)
=
b_0+
\left(
\sum_{\{i:\,a_i=0\}} d_i x_i
\right)
|x|^{s_1-1},
\]
for some $b_0,d_i\in\mathbb R$.

\item[(iii)] \textbf{First balanced case.}
Suppose
\[
a_i=0 \quad \text{for at least one } i,
\qquad
|\tau|=|s_1|.
\]
Then equality holds if and only if
\[
u(x)
=
b_0+b_1|x|^\tau
+
\left(
\sum_{\{i:\,a_i=0\}} d_i x_i
\right)
|x|^{s_1-1},
\]
for some $b_0,b_1,d_i\in\mathbb R$.

\item[(iv)] \textbf{Second non-radial case.}
Suppose
\[
a_i>0 \quad \text{for every } i,
\qquad
|s_2|<|\tau|.
\]
Then equality holds if and only if
\[
u(x)
=
b_0
+
|x|^{s_2-2}\sum_{i=1}^n c_i x_i^2, \qquad
\sum_{i=1}^n c_i(a_i+1)=0.
\]
for some $b_0,c_1,\ldots,c_n\in\mathbb R$, with a nonzero
degree-two term.

\item[(v)] \textbf{Second balanced case.}
Suppose
\[
a_i>0 \quad \text{for every } i,
\qquad
|\tau|=|s_2|.
\]
Then equality holds if and only if
\[
u(x)
=
{b_0+b_1|x|^\tau
+
|x|^{s_2-2}
\sum_{i=1}^n
c_i x_i^2, \qquad
\sum_{i=1}^n c_i(a_i+1)=0,}
\]
for some $b_0, b_1,c_1,\ldots,c_n\in\mathbb R$, with a nonzero
degree-two term.

\end{enumerate}
\end{corollary}

\begin{proof}
Recall that the lifting maps the monomial-weighted problem to the
subspace of functions on $\mathbb R^N$ which are invariant under
independent rotations in the blocks $\mathbb R^{a_i+1}$.

Suppose first that $a_i=0$ for at least one index $i$. The corresponding
block is one-dimensional, and hence the degree-one spherical eigenspace
contains nonzero block-invariant functions. The first admissible
non-radial eigenvalue is therefore
\[
\lambda_1=N-1.
\]
By the equality characterization for the weighted Poincar\'e inequality
in $\mathbb R^N$, the radial extremals are
\[
\widetilde u(y)=b_0+b_1|y|^\tau,
\]
whereas the first non-radial extremals which are block-invariant are
\[
\widetilde u(y)
=
b_0+
\left(
\sum_{\{i:\,a_i=0\}} d_i y_i
\right)
|y|^{s_1-1}.
\]
Since $y_i=x_i$ whenever $a_i=0$ and $|y|=|x|$, these descend precisely
to the functions in {\rm (i)--(iii)}. Comparing
$\delta\tau^2$ and $\delta\tau s_1$ is equivalent to comparing
$|\tau|$ and $|s_1|$, since $\tau s_1>0$.

Suppose now that $a_i>0$ for every $i$. Then every block
$\mathbb R^{a_i+1}$ has dimension at least two, and no nonzero spherical
harmonic of degree one is block-invariant. Hence the eigenvalue
$\lambda_1=N-1$ is absent from the lifted class.

The first nonzero block-invariant spherical harmonics occur at degree
two and are spanned by
\[
y\longmapsto
\sum_{i=1}^n
c_i
\left(
|y_i|^2-\frac{a_i+1}{N}|y|^2
\right).
\]
Indeed,
\[
\Delta_y
\left(
|y_i|^2-\frac{a_i+1}{N}|y|^2
\right)
=
2(a_i+1)
-
\frac{a_i+1}{N}\,2N
=
0.
\]
Thus these are homogeneous harmonic polynomials of degree two. This
space is nonzero for $n\ge2$, and its spherical eigenvalue is
\[
\lambda_2=2N.
\]

Repeating the argument for the weighted Poincar\'e inequality with
$\lambda_2$ in place of $\lambda_1$ gives the non-radial constant
$\delta\tau s_2$. Since the degree-two block-invariant eigenspace is
nontrivial, this value is attained by an admissible competitor, and
hence the sharp constant in this case is
\[
\min\{\delta\tau^2,\delta\tau s_2\}.
\]
The corresponding degree-two extremals are
\[
\widetilde u(y)
=
b_0
+
|y|^{s_2-2}
\sum_{i=1}^n
c_i
\left(
|y_i|^2-\frac{a_i+1}{N}|y|^2
\right).
\]
Under the lifting $x_i=|y_i|$ and $|x|=|y|$, these descend to
\[
u(x)
=
b_0
+
|x|^{s_2-2}
\sum_{i=1}^n
c_i
\left(
x_i^2-\frac{a_i+1}{N}|x|^2
\right).
\]
Comparison of $\delta\tau^2$ and $\delta\tau s_2$ is equivalent to
comparison of $|\tau|$ and $|s_2|$. This yields {\rm (i), (iv), and (v)}
and completes the proof.
\end{proof}


\section{Stability of the weighted Poincar\'e inequalities: proof of Theorem \ref{Th:stability_poincare_monomialweight}}\label{S:poincare_stability}

The goal of this section is to establish the stability of the weighted Poincar\'e
inequalities obtained above. The argument is carried out in three steps. We first prove a
stability estimate for radial functions, where the Laguerre expansion makes the spectral
gap between the first and the second mode explicit. We then prove the corresponding
estimate for the non-radial component supplied by the spherical harmonic decomposition.
Finally, we combine the two, and lift the resulting inequality to monomial weights, thereby
proving Theorem \ref{Th:stability_poincare_monomialweight}. This division reflects the
spectral structure of the underlying operator, and it allows us to isolate the
contributions of the radial and of the angular modes.

{We retain the notation for the weighted energy spaces introduced in the
preceding sections. More precisely, the one-dimensional radial energy
space is
\[
Y_{\delta,\tau,\beta}
=
\left\{
f\in W_{\mathrm{loc}}^{1,2}(0,\infty):
\int_0^\infty |f'(r)|^2e^{-\delta r^\tau}r^\beta\,dr
+
\int_0^\infty |f(r)|^2e^{-\delta r^\tau}
r^{\beta+\tau-2}\,dr
<\infty
\right\},
\]
equipped with the norm
\[
\|f\|_{Y_{\delta,\tau,\beta}}^2
=
\int_0^\infty |f'(r)|^2e^{-\delta r^\tau}r^\beta\,dr
+
\int_0^\infty |f(r)|^2e^{-\delta r^\tau}
r^{\beta+\tau-2}\,dr.
\]
For the corresponding problem on $\mathbb R^n$, we recall that
\[
d\mu(x)=e^{-\delta|x|^\tau}|x|^\beta\,dx,
\qquad
d\nu(x)=e^{-\delta|x|^\tau}|x|^{\beta+\tau-2}\,dx,
\]
and
\[
\|u\|_W^2
=
\int_{\mathbb R^n}|\nabla u|^2\,d\mu
+
\int_{\mathbb R^n}|u|^2\,d\nu.
\]
Here
\[
W=\overline{\mathcal A_\tau}^{\,\|\cdot\|_W},
\qquad
\mathcal A_\tau
=
\left\{
u\in C_c^\infty(\mathbb R^n):
\|u\|_W<\infty
\right\}.
\]

The radial stability estimate will first be established in
$Y_{\delta,\tau,\beta}$ and then combined with the spherical harmonic
decomposition to obtain the corresponding estimate in $W$.}

\begin{theorem}[Stability for the general radial Poincar\'e inequality]\label{th:general_radial_poincare_stability}
Let $\delta>0$, and assume
\[
\frac{\beta+\tau-1}{\tau}>0.
\]
Then for every $f\in {Y_{\delta,\tau,\beta}}$, 
we have
\begin{align}
&\int_0^\infty |f'(r)|^2 e^{-\delta r^\tau}r^\beta\,dr
-\delta\tau^2\int_0^\infty |f(r)-\bar f|^2 e^{-\delta r^\tau}r^{\beta+\tau-2}\,dr \notag \\
&\qquad \qquad \qquad\ge \delta\tau^2\inf_{a,b\in\mathbb R}
\int_0^\infty |f(r)-(a+br^\tau)|^2 e^{-\delta r^\tau}r^{\beta+\tau-2}\,dr,
\end{align}
where
\[
\bar f:=
\frac{\int_0^\infty f(r)e^{-\delta r^\tau}r^{\beta+\tau-2}\,dr}
{\int_0^\infty e^{-\delta r^\tau}r^{\beta+\tau-2}\,dr}.
\]

\end{theorem}

\begin{proof}
The proof is a refinement of that of Theorem \ref{th:general_radial_poincare}: there we
discarded the whole of the spectral gap, whereas here we retain half of it. As before,
replacing $f$ by $f-\bar f$ changes neither side, so we may assume that 
\[
\int_0^\infty f(r)e^{-\delta r^\tau}r^{\beta+\tau-2}\,dr=0.
\]

Using the same change of variables $t=\delta r^\tau$ together with the dictionary
\eqref{eq:parseval_dictionary} and the quadratic-form identity \eqref{eq:laguerre_parseval}
of Lemma \ref{L:laguerre_parseval} (the series below may then be integrated term by term by
monotone convergence, all its terms being nonnegative), we obtain
\begin{align}
        &\int_0^\infty |f'(r)|^2e^{-\delta r^\tau}r^\beta\,dr=\delta\tau^2\frac{1}{\delta|\tau|} \left(\frac 1\delta\right)^\alpha \int_0^\infty \sum_{n=1}^\infty n a_n^2 L_n^{(\alpha)}(t)^2\, e^{-t} \, t^\alpha\,dt\notag\\
         &=\delta\tau^2 \frac{1}{\delta|\tau|}\left(\frac 1\delta\right)^\alpha \int_0^\infty \sum_{n=1}^\infty a_n^2 L_n^{(\alpha)}(t)^2\, e^{-t} \, t^\alpha\,dt +\delta\tau^2 \frac{1}{\delta|\tau|}\left(\frac 1\delta\right)^\alpha \int_0^\infty \sum_{n=2}^\infty (n-1)a_n^2 L_n^{(\alpha)}(t)^2\, e^{-t} \, t^\alpha\,dt\notag\\
         & \geq \delta\tau^2 \frac{1}{\delta|\tau|}\left(\frac 1\delta\right)^\alpha \int_0^\infty \sum_{n=1}^\infty a_n^2 L_n^{(\alpha)}(t)^2\, e^{-t} \, t^\alpha\,dt +\delta\tau^2 \frac{1}{\delta|\tau|}\left(\frac 1\delta\right)^\alpha \int_0^\infty \sum_{n=2}^\infty a_n^2 L_n^{(\alpha)}(t)^2\, e^{-t} \, t^\alpha\,dt\notag\\
         &=\delta\tau^2\int_0^\infty |f(r)|^2 e^{-\delta r^\tau}r^{\beta+\tau-2}\,dr +\delta\tau^2 \frac{1}{\delta|\tau|}\left(\frac 1\delta\right)^\alpha \int_0^\infty \left(\sum_{n=1}^\infty a_n^2 L_n^{(\alpha)}(t)^2 - a_1^2 L_1^{(\alpha)}(t)^2\right) \, e^{-t} \, t^\alpha\,dt\notag\\
         &=\delta\tau^2\int_0^\infty |f(r)|^2 e^{-\delta r^\tau}r^{\beta+\tau-2}\,dr +\delta\tau^2 \frac{1}{\delta|\tau|}\left(\frac 1\delta\right)^\alpha \int_0^\infty |\sum_{n=1}^\infty a_n L_n^{(\alpha)}(t) - a_1 L_1^{(\alpha)}(t)|^2 \, e^{-t} \, t^\alpha\,dt\notag\\
         &\geq \delta\tau^2\int_0^\infty |f(r)|^2 e^{-\delta r^\tau}r^{\beta+\tau-2}\,dr +\delta\tau^2\inf_{a,b\in\mathbb R} \int_0^\infty |f(r) - (a r^{\tau}+b)|^2 \, e^{-\delta r^\tau}r^{\beta+\tau-2}\,dr.
    \end{align}
\end{proof}

We turn to the non-radial component. Here the role played by the gap between the first two
Laguerre modes is played by the gap between the first two non-zero eigenvalues
$\lambda_1=n-1$ and $\lambda_2=2n$ of the Laplace--Beltrami operator, and it is this gap
that produces the parameter $\alpha_1$ below.

\begin{theorem}[Stability in the non-radial case]\label{thm:nonradial_stability}
Let $n\ge2$, let
$\Phi_{u}(x)=\sum_{\ell=1}^{\infty} f_\ell(r)\varphi_\ell(\sigma)$
be the non-radial part of $u$, and let $s_1$ be as in Theorem \ref{Th:general_poincare}. Then
\begin{align}
&\sum_{\ell=1}^{\infty}\int_0^\infty
\left(|f_\ell'(r)|^2 + \frac{\lambda_\ell}{r^2}f_\ell(r)^2\right)e^{-\delta r^\tau}r^{\beta+n-1}\,dr -\delta\tau s_1\sum_{\ell=1}^{\infty}\int_0^\infty |f_\ell(r)|^2 e^{-\delta r^\tau}r^{\beta+n+\tau-3}\,dr \notag \\
& \qquad \qquad \qquad \ge
 \min\{\delta\tau^2,\ \delta\tau\alpha_1 \}
\inf_{D\in \R^n}\norm{\Phi_u-D\cdot x\, r^{s_1-1}}^2_{L^2(e^{-\delta r^\tau}r^{\beta+n+\tau-3}\,dr)}
\end{align}
Here
\begin{equation*}
    \alpha_1 =
\frac{\sgn(\tau)\left(\sqrt{(\beta+n-2)^2+8n}-\sqrt{(\beta+n-2)^2+4(n-1)}\right)}{2},
\qquad\text{that is,}
\end{equation*}

\end{theorem}

\begin{proof}
We refine the proof of Theorem \ref{Th:general_poincare}. As there, we set
$f_\ell(r)=r^{s_1}g_\ell(r)$ for $\ell\ge1$, with
\[
s_1=\frac{-(\beta+n-2)+\sgn(\tau)\sqrt{(\beta+n-2)^2+4(n-1)}}{2} ,
\]
the difference being that we now keep, rather than discard, the two terms that were thrown
away before: the Dirichlet energy of $g_1$, and the excess $\lambda_\ell-\lambda_1$ carried
by the modes of degree at least two.
For the non-radial part, using the spherical harmonic decomposition and
\(\lambda_\ell=\ell(\ell+n-2)\), we have
\begin{align}
&\sum_{\ell=1}^\infty \int_0^\infty \left(|f_\ell'(r)|^2+\frac{\lambda_\ell}{r^2}f_\ell(r)^2\right)
e^{-\delta r^\tau}r^{\beta+n-1}\,dr \notag\\
&=\sum_{\ell=1}^\infty \int_0^\infty\left(|f_\ell'(r)|^2+\frac{\lambda_1}{r^2}f_\ell(r)^2\right)
e^{-\delta r^\tau}r^{\beta+n-1}\,dr \; + \sum_{\ell=2}^{\infty} (\lambda_\ell-\lambda_1)\int_0^\infty \frac{f_\ell(r)^2}{r^2}e^{-\delta r^\tau}r^{\beta+n-1}\,dr,\notag\\
&=\sum_{\ell=1}^\infty \int_0^\infty
r^{2s_1}|g_\ell'(r)|^2 e^{-\delta r^\tau}r^{\beta+n-1}\,dr
+\delta\tau s_1\sum_{\ell=1}^\infty \int_0^\infty
g_\ell(r)^2 e^{-\delta r^\tau}r^{\beta+n+2s_1+\tau-3}\,dr\notag\\
& \qquad \qquad \qquad \qquad \qquad \qquad  \qquad \qquad \qquad + \sum_{\ell=2}^{\infty} (\lambda_\ell-\lambda_1)\int_0^\infty \frac{g_\ell(r)^2}{r^2}e^{-\delta r^\tau}r^{\beta+n+2s_1-1}\,dr, \notag \\
&=\delta\tau s_1\sum_{\ell=1}^\infty \int_0^\infty
g_\ell(r)^2 e^{-\delta r^\tau}r^{\beta+n+2s_1+\tau-3}\,dr + \int_0^\infty |g_1'(r)|^2 e^{-\delta r^\tau}r^{\beta+n+2s_1-1}\,dr \notag \\ 
&\qquad \qquad \qquad \qquad \qquad \qquad  \qquad \qquad \qquad 
+\sum_{\ell=2}^\infty \int_0^\infty \left(|g_\ell'(r)|^2+\frac{\Lambda_\ell}{r^2}g_\ell(r)^2\right)e^{-\delta r^\tau}r^{\beta+n+2s_1-1}\,dr \;
\end{align}
where we have written $\Lambda_\ell:=\lambda_\ell-\lambda_1$. We now estimate the two
remaining terms separately. To the first,
\[
\int_0^\infty |g_1'(r)|^2 e^{-\delta r^\tau}r^{\beta+n+2s_1-1}\,dr ,
\]
we apply the radial Poincar\'e inequality of Theorem \ref{th:general_radial_poincare}, with
$\beta$ replaced by $\beta+n+2s_1-1$.
\begin{equation}
      \int_0^\infty |g_1'(r)|^2 e^{-\delta r^\tau}r^{\beta+n+2s_1-1}\,dr \geq \delta\tau^2 \int_0^\infty \abs{g_1(r)-\overline {g_1}(r)}^2 e^{-\delta r^\tau}r^{\beta+n+2s_1+\tau-3}\,dr
\end{equation}

To the second,
\[
\sum_{\ell=2}^\infty \int_0^\infty \left(|g_\ell'(r)|^2+\frac{\Lambda_\ell}{r^2}g_\ell(r)^2\right)e^{-\delta r^\tau}r^{\beta+n+2s_1-1}\,dr ,
\]
we apply the argument used for the non-radial modes in the proof of Theorem
\ref{Th:general_poincare}, this time setting $g_\ell(r)=r^{\alpha_1}h_\ell(r)$. This gives
\begin{align}
        &\sum_{\ell=2}^\infty \int_0^\infty \left(|g_\ell'(r)|^2+\frac{\Lambda_\ell}{r^2}g_\ell(r)^2\right)e^{-\delta r^\tau}r^{\beta+n+2s_1-1}\,dr \notag\\
        &= \sum_{\ell=3}^\infty \int_0^\infty \left(\frac{\Lambda_\ell-\Lambda_2}{r^2}r^{2\alpha_1}h_\ell(r)^2\right)e^{-\delta r^\tau}r^{\beta+n+2s_1-1}\,dr+\sum_{\ell=2}^\infty \int_0^\infty \left(|h'_\ell|^2 \right))e^{-\delta r^\tau}r^{\beta+n+2s_1+2\alpha_1-1}\,dr \notag\\
        &\qquad \qquad \qquad  \qquad+(\alpha_1^2+\Lambda_2-\alpha_1(\beta+n+2s_1+2\alpha_1-2))\int_0^\infty |h_\ell|^2 e^{-\delta r^\tau}r^{\beta_1+n+2s_1+2\alpha_1-3}\,dr \notag\\
        &\qquad \qquad \qquad \qquad \qquad \qquad \qquad \qquad  + \delta\tau\alpha_1 \sum_{\ell=2}^\infty \int_0^\infty |h_\ell(r)|^2e^{-\delta r^\tau}r^{\beta+n+2s_1+2\alpha_1+\tau-3}\,dr \notag\\
        &\geq \delta\tau\alpha_1 \sum_{\ell=2}^\infty \int_0^\infty |g_\ell(r)|^2e^{-\delta r^\tau}r^{\beta+n+2s_1+\tau-3}\,dr
\end{align}

where $\alpha_1$ is the root of the equation $(\alpha_1^2+\Lambda_2-\alpha_1(\beta+n+2s_1+2\alpha_1-2))=0$ and $\tau\alpha_1>0$. that is 
\begin{align}
&\alpha_1=
\begin{cases}
\dfrac{-(\beta+n+2s_1-2)+
\sqrt{(\beta+n+2s_1-2)^2+4\Lambda_2}}{2},
& \tau>0,\\[1.2em]
\dfrac{-(\beta+n+2s_1-2)-
\sqrt{(\beta+n+2s_1-2)^2+4\Lambda_2}}{2},
& \tau<0.
\end{cases} \notag
\end{align}
and, substituting the value of $s_1$ and using $\Lambda_2=\lambda_2-\lambda_1=2n-(n-1)=n+1$, this becomes 
\begin{equation}
    \alpha_1 =
\begin{cases}
\displaystyle
\frac{-\sqrt{(\beta+n-2)^2+4(n-1)}
+\sqrt{(\beta+n-2)^2+8n}}{2},
& \tau>0,\\[1.5em]
\displaystyle
\frac{
\sqrt{(\beta+n-2)^2+4(n-1)}
-
\sqrt{(\beta+n-2)^2+8n}
}{2},
& \tau<0.
\end{cases}
\end{equation}

Now combining the two cases  we have
\begin{align}
\sum_{\ell=1}^\infty \int_0^\infty &\left(|f_\ell'(r)|^2+\frac{\lambda_\ell}{r^2}f_\ell(r)^2\right) e^{-\delta r^\tau}r^{\beta+n-1}\,dr 
-\delta\tau s_1\sum_{\ell=1}^\infty \int_0^\infty \abs{f_\ell(r)}^2 e^{-\delta r^\tau}r^{\beta+n+\tau-3}\,dr \notag \\ 
&\geq \delta\tau^2 \int_0^\infty \abs{g_1(r)-\overline {g_1}(r)}^2 e^{-\delta r^\tau}r^{\beta+n+2s_1+\tau-3}\,dr+ \delta\tau\alpha_1 \sum_{\ell=2}^\infty \int_0^\infty |g_\ell(r)|^2e^{-\delta r^\tau}r^{\beta+n+2s_1+\tau-3}\,dr \notag \\
&=
\delta\tau^2\inf_{d\in\mathbb R}
\int_0^\infty |f_1(r)-d r^{s_1}|^2
e^{-\delta r^\tau}r^{\beta+n+\tau-3}\,dr
+ \delta\tau\alpha_1\sum_{\ell=2}^\infty
\int_0^\infty |f_\ell(r)|^2
e^{-\delta r^\tau}r^{\beta+n+\tau-3}\,dr \notag \\
&\geq
\min\{\delta\tau^2,\delta\tau\alpha_1\}
\Biggl(
\inf_{d\in\mathbb R}
\int_0^\infty |f_1(r)-d r^{s_1}|^2
e^{-\delta r^\tau}r^{\beta+n+\tau-3}\,dr
+\sum_{\ell=2}^\infty
\int_0^\infty |f_\ell(r)|^2
e^{-\delta r^\tau}r^{\beta+n+\tau-3}\,dr
\Biggr) \notag \\
&=
\min\{\delta\tau^2,\delta\tau\alpha_1\}\inf_{d\in\mathbb R}
\int_0^\infty \Bigl|\sum_{\ell=1}^\infty f_\ell (r)\varphi_\ell(\sigma)-d\, \varphi_1(\sigma) r^{s_1}\Bigr|^2
e^{-\delta r^\tau}r^{\beta+n+\tau-3}\,dr \notag \\
&\geq
\min\{\delta\tau^2,\delta\tau\alpha_1\}
\inf_{D\in\mathbb R^n}
\int_{\mathbb R^n}
\left|\Phi_u (x)-D\cdot x\,|x|^{s_1-1}\right|^2
e^{-\delta|x|^\tau}|x|^{\beta+\tau-2}\,dx .
\end{align}

Here \(\Phi_u\) denotes the non-radial part of \(u\); in the last two lines the index
$\ell=1$ is understood to run over an orthonormal basis of the $n$-dimensional eigenspace of
$\lambda_1$, so that the competitors $d\,\varphi_1(\sigma)r^{s_1}$ sweep out exactly the
family $D\cdot x\,|x|^{s_1-1}$, $D\in\mathbb R^{n}$. This proves the theorem.

\end{proof}

We now combine the radial and non-radial stability estimates in order to establish stability of weighted Poincaré inequalities on $\mathbb R^n$. Recall that, for the weighted Poincaré inequality, the equality cases are described by three different manifolds of optimizers, according to the relation between the relevant constants. Accordingly, the stability statement is also divided into three cases, with the distance taken from the corresponding optimizer manifold in each case.

 \begin{theorem}[Stability of weighted Poincaré inequalities on $\mathbb R^n$]
\label{Th:poincare_stability}
Let $d\mu=e^{-\delta|x|^\tau}|x|^\beta\,dx,$
and
$d\nu=e^{-\delta|x|^\tau}|x|^{\beta+\tau-2}\,dx.$ For
\[
u\in   W,
\]
define
\[
\bar u=\frac{\int_{\mathbb R^n}u\,d\nu}
{\int_{\mathbb R^n}d\nu}.
\]
Then the following stability estimates hold.

\medskip

\noindent
\textup{(i)} If $\tau^2<\tau s_1$, then
\[
\begin{aligned}
&\int_{\mathbb R^n}|\nabla u|^2\,d\mu
-\delta\tau^2\int_{\mathbb R^n}|u-\bar u|^2\,d\nu\\
&\qquad\ge
\min\{\delta\tau^2,\delta\tau(s_1-\tau)\}
\inf_{b_0,b_1\in\mathbb R}
\int_{\mathbb R^n}
|u-b_0-b_1|x|^\tau|^2\,d\nu.
\end{aligned}
\]

\medskip

\noindent
\textup{(ii)} If $\tau s_1<\tau^2$, then
\[
\begin{aligned}
&\int_{\mathbb R^n}|\nabla u|^2\,d\mu
-\delta\tau s_1\int_{\mathbb R^n}|u-\bar u|^2\,d\nu\\
&\qquad\ge
\min\{\delta\tau(\tau-s_1),\delta\tau^2,\delta\tau\alpha_1\}
\inf_{\substack{b\in\mathbb R\\ B\in\mathbb R^n}}
\int_{\mathbb R^n}
|u-b-B\cdot x\,|x|^{s_1-1}|^2\,d\nu.
\end{aligned}
\]

\medskip

\noindent
\textup{(iii)} If $\tau=s_1$, then
\[
\begin{aligned}
&\int_{\mathbb R^n}|\nabla u|^2\,d\mu
-\delta\tau^2\int_{\mathbb R^n}|u-\bar u|^2\,d\nu\\
&\qquad\ge
\min\{\delta\tau^2,\delta\tau\alpha_1\}
\inf_{\substack{b_0,b_1\in\mathbb R\\ B\in\mathbb R^n}}
\int_{\mathbb R^n}
|u-b_0-b_1|x|^\tau-B\cdot x\,|x|^{s_1-1}|^2\,d\nu.
\end{aligned}
\]

Here $s_1$ is as in Theorem~\ref{Th:general_poincare}, and
\[
\alpha_1=
\sgn(\tau)\frac{
-\sqrt{(\beta+n-2)^2+4(n-1)}
+\sqrt{(\beta+n-2)^2+8n}
}{2}.
\]
\end{theorem}

\begin{proof}
Write
\[
u(r,\sigma)
=
\sum_{\ell=0}^\infty f_\ell(r)\varphi_\ell(\sigma)
=
u_0(r)+\Phi_u(r,\sigma),
\]
where
\[
u_0=f_0,
\qquad
\Phi_u=\sum_{\ell=1}^\infty f_\ell(r)\varphi_\ell(\sigma).
\]
Since $d\mu$ and $d\nu$ are radial, the radial and nonradial parts are
orthogonal. Moreover, $\overline{\Phi_u}=0$, and hence $\bar u=\bar u_0$.
Therefore,
\[
\int_{\mathbb R^n}|\nabla u|^2\,d\mu
=
\int_{\mathbb R^n}|\nabla u_0|^2\,d\mu
+
\int_{\mathbb R^n}|\nabla\Phi_u|^2\,d\mu
\]
and
\[
\int_{\mathbb R^n}|u-\bar u|^2\,d\nu
=
\int_{\mathbb R^n}|u_0-\bar u_0|^2\,d\nu
+
\int_{\mathbb R^n}|\Phi_u|^2\,d\nu.
\]

The radial and nonradial stability estimates give
\[
\int_{\mathbb R^n}|\nabla u_0|^2\,d\mu
-\delta\tau^2
\int_{\mathbb R^n}|u_0-\bar u_0|^2\,d\nu
\ge \delta\tau^2X_0
\]
and
\[
\int_{\mathbb R^n}|\nabla\Phi_u|^2\,d\mu
-\delta\tau s_1
\int_{\mathbb R^n}|\Phi_u|^2\,d\nu
\ge
\min\{\delta\tau^2,\delta\tau\alpha_1\}Y_1.
\]

We now consider the three cases separately.

\medskip

\noindent
\textup{(i)} Suppose that $\tau^2<\tau s_1$. Then the sharp Poincaré
constant is $\delta\tau^2$. Hence
\[
\begin{aligned}
&\int_{\mathbb R^n}|\nabla u|^2\,d\mu
-\delta\tau^2
\int_{\mathbb R^n}|u-\bar u|^2\,d\nu\\
&=
\left(
\int_{\mathbb R^n}|\nabla u_0|^2\,d\mu
-\delta\tau^2
\int_{\mathbb R^n}|u_0-\bar u_0|^2\,d\nu
\right)+
\left(
\int_{\mathbb R^n}|\nabla\Phi_u|^2\,d\mu
-\delta\tau^2
\int_{\mathbb R^n}|\Phi_u|^2\,d\nu
\right)\\
&=
\left(
\int_{\mathbb R^n}|\nabla u_0|^2\,d\mu
-\delta\tau^2
\int_{\mathbb R^n}|u_0-\bar u_0|^2\,d\nu
\right)+
\left(
\int_{\mathbb R^n}|\nabla\Phi_u|^2\,d\mu
-\delta\tau s_1
\int_{\mathbb R^n}|\Phi_u|^2\,d\nu
\right)\\
&\qquad \qquad \qquad \qquad \qquad \qquad \qquad \qquad \qquad \qquad \qquad\qquad+
\delta\tau(s_1-\tau)
\int_{\mathbb R^n}|\Phi_u|^2\,d\nu.\\
&\ge
\delta\tau^2\inf_{b_0,b_1\in\mathbb R}\int_{\mathbb R^n}|u_0-b_0-b_1|x|^\tau|^2\,d\nu+\delta\tau(s_1-\tau)\int_{\mathbb R^n}|\Phi_u|^2\,d\nu,\\
&\ge
\min\{\delta\tau^2,\delta\tau(s_1-\tau)\}\inf_{b_0,b_1\in\mathbb R}
\int_{\mathbb R^n}
|u-b_0-b_1|x|^\tau|^2\,d\nu,.
\end{aligned}
\]

\medskip

\noindent
\textup{(ii)} Suppose that $\tau s_1<\tau^2$. Then the sharp Poincaré
constant is $\delta\tau s_1$. Hence
\[
\begin{aligned}
&\int_{\mathbb R^n}|\nabla u|^2\,d\mu
-\delta\tau s_1\int_{\mathbb R^n}|u-\bar u|^2\,d\nu\\
&=
\left(
\int_{\mathbb R^n}|\nabla u_0|^2\,d\mu
-\delta\tau s_1\int_{\mathbb R^n}|u_0-\bar u_0|^2\,d\nu
\right)
+
\left(
\int_{\mathbb R^n}|\nabla\Phi_u|^2\,d\mu
-\delta\tau s_1\int_{\mathbb R^n}|\Phi_u|^2\,d\nu
\right)\\
&=
\left(
\int_{\mathbb R^n}|\nabla u_0|^2\,d\mu
-\delta\tau^2\int_{\mathbb R^n}|u_0-\bar u_0|^2\,d\nu
\right)
+\delta\tau(\tau-s_1)
\int_{\mathbb R^n}|u_0-\bar u_0|^2\,d\nu\\
&\qquad \qquad \qquad \qquad \qquad \qquad \qquad \qquad \qquad+
\left(
\int_{\mathbb R^n}|\nabla\Phi_u|^2\,d\mu
-\delta\tau s_1\int_{\mathbb R^n}|\Phi_u|^2\,d\nu
\right)\\
&\ge
\delta\tau(\tau-s_1)
\int_{\mathbb R^n}|u_0-\bar u_0|^2\,d\nu
+
\min\{\delta\tau^2,\delta\tau\alpha_1\}
\inf_{B\in\mathbb R^n}
\int_{\mathbb R^n}
|\Phi_u-B\cdot x\,|x|^{s_1-1}|^2\,d\nu\\
&\ge
\min\{\delta\tau(\tau-s_1),\delta\tau^2,\delta\tau\alpha_1\}
\inf_{\substack{b\in\mathbb R\\B\in\mathbb R^n}}
\int_{\mathbb R^n}
|u-b-B\cdot x\,|x|^{s_1-1}|^2\,d\nu.
\end{aligned}
\]

\medskip

\noindent
\textup{(iii)} Suppose that $\tau=s_1$. Then
$\delta\tau^2=\delta\tau s_1$, and hence
\[
\begin{aligned}
&\int_{\mathbb R^n}|\nabla u|^2\,d\mu
-\delta\tau^2\int_{\mathbb R^n}|u-\bar u|^2\,d\nu\\
&=
\left(
\int_{\mathbb R^n}|\nabla u_0|^2\,d\mu
-\delta\tau^2\int_{\mathbb R^n}|u_0-\bar u_0|^2\,d\nu
\right)
+
\left(
\int_{\mathbb R^n}|\nabla\Phi_u|^2\,d\mu
-\delta\tau s_1\int_{\mathbb R^n}|\Phi_u|^2\,d\nu
\right)\\
&\ge
\delta\tau^2
\inf_{b_0,b_1\in\mathbb R}
\int_{\mathbb R^n}
|u_0-b_0-b_1|x|^\tau|^2\,d\nu
+
\min\{\delta\tau^2,\delta\tau\alpha_1\}
\inf_{B\in\mathbb R^n}
\int_{\mathbb R^n}
|\Phi_u-B\cdot x\,|x|^{s_1-1}|^2\,d\nu\\
&\ge
\min\{\delta\tau^2,\delta\tau\alpha_1\}
\inf_{\substack{b_0,b_1\in\mathbb R\\B\in\mathbb R^n}}
\int_{\mathbb R^n}
\left|
u-b_0-b_1|x|^\tau
-B\cdot x\,|x|^{s_1-1}
\right|^2\,d\nu.
\end{aligned}
\]
This proves \textup{(iii)} and completes the proof.
\end{proof}

It remains to pass from the whole space to monomial weights. As in the proof of Theorem
\ref{Th:poincare_monomial_weight}, this is done by lifting; the only point requiring care is
that the competitors appearing on the right-hand side must themselves be liftable, which is
why the linear forms occurring below are built only from the variables $x_i$ with $a_i=0$.

\begin{proof}[Proof of Theorem \ref{Th:stability_poincare_monomialweight}]
Let
\[
A=(a_1,\ldots,a_n),\qquad N=n+|A|.
\]
We lift \(u\) from \(\mathbb R^n_{A,+}\) to the block-radial function
\(\widetilde u\) on \(\mathbb R^N\), apply
Theorem~\ref{Th:poincare_stability} in the lifted space, and then descend
the resulting estimates.

More precisely, write
\[
y=(y_1,\ldots,y_n),\qquad y_i\in\mathbb R^{a_i+1},
\]
so that \(x_i=|y_i|\) whenever \(a_i>0\), while \(x_i=y_i\) whenever
\(a_i=0\). By the lifting argument used in the proof of
Theorem~\ref{Th:poincare_monomial_weight}, we have
\[
\int_{\mathbb R^N}|\nabla\widetilde u|^2
e^{-\delta|y|^\tau}|y|^\beta\,dy
=
c_A\int_{\mathbb R^n_{A,+}}|\nabla u|^2\,d\mu_A
\]
and, for every \(b\in\mathbb R\),
\[
\int_{\mathbb R^N}|\widetilde u-b|^2
e^{-\delta|y|^\tau}|y|^{\beta+\tau-2}\,dy
=
c_A\int_{\mathbb R^n_{A,+}}|u-b|^2\,d\nu_A,
\]
where
\[
c_A=\prod_{i=1}^n|\mathbb S^{a_i}|.
\]
Thus the deficit in the lifted space is \(c_A\) times the corresponding
deficit on \(\mathbb R^n_{A,+}\). It remains to determine which
optimizing functions in the lifted space are compatible with the
block-radial structure of \(\widetilde u\).

We first consider case~\textup{(i)}. If \(a_i=0\) for at least one \(i\)
and \(|\tau|<|s_1|\), then the first admissible nonradial mode has degree
one, but the radial mode gives the sharp constant. If \(a_i>0\) for every
\(i\) and \(|\tau|<|s_2|\), the degree-one mode is not admissible and the
first admissible nonradial mode has degree two; nevertheless, the radial
mode again gives the sharp constant. Hence the radial stability estimate
in case~\textup{(i)} follows from the lifted theorem and the identities
above.

We next suppose that \(a_i=0\) for at least one \(i\). Let
\(\xi=(\xi_1,\ldots,\xi_n)\in\mathbb R^N\), where
\(\xi_i\in\mathbb R^{a_i+1}\). If \(a_i=0\), we identify \(\xi_i\) with a
scalar \(d_i\). Then
\[
\xi\cdot y
=
\sum_{\{i:a_i>0\}}\xi_i\cdot y_i
+
\sum_{\{i:a_i=0\}}d_i y_i.
\]
For every \(i\) such that \(a_i>0\), the function \(\widetilde u\) is
radial in the block \(y_i\). Therefore, for each coordinate \(y_{i,j}\),
\[
\int_{\mathbb R^N}
(\widetilde u-b)y_{i,j}|y|^{s_1-1}\,d\nu_N=0,
\]
where
\[
d\nu_N=e^{-\delta|y|^\tau}|y|^{\beta+\tau-2}\,dy.
\]
Moreover, the mixed quadratic terms vanish by symmetry:
\[
\int_{\mathbb R^N}
y_{i,j}y_{\ell,m}|y|^{2s_1-2}\,d\nu_N=0
\qquad\text{if }(i,j)\neq(\ell,m).
\]
Consequently, the coefficients corresponding to blocks with \(a_i>0\)
occur only through nonnegative quadratic terms, and their minimizing
values are zero. Hence only the one-dimensional blocks \(a_i=0\)
survive. Since \(y_i=x_i\) for such indices,
\[
\begin{aligned}
&\inf_{\substack{b\in\mathbb R\\ \xi\in\mathbb R^N}}
\int_{\mathbb R^N}
\left|
\widetilde u-b-(\xi\cdot y)|y|^{s_1-1}
\right|^2\,d\nu_N =
c_A
\inf_{\substack{b\in\mathbb R\\ d_i\in\mathbb R}}
\int_{\mathbb R^n_{A,+}}
\left|
u-b-
\left(\sum_{\{i:a_i=0\}}d_i x_i\right)|x|^{s_1-1}
\right|^2\,d\nu_A.
\end{aligned}
\]
This proves the identification of the distance term in
case~\textup{(ii)}. The same argument, with
\(\widetilde u-b_0-b_1|y|^\tau\) in place of \(\widetilde u-b\), gives
\[
\begin{aligned}
&\inf_{\substack{b_0,b_1\in\mathbb R\\ \xi\in\mathbb R^N}}
\int_{\mathbb R^N}
\left|
\widetilde u-b_0-b_1|y|^\tau-(\xi\cdot y)|y|^{s_1-1}
\right|^2\,d\nu_N\\
&\qquad \qquad\qquad\qquad\qquad=
c_A
\inf_{\substack{b_0,b_1\in\mathbb R\\ d_i\in\mathbb R}}
\int_{\mathbb R^n_{A,+}}
\left|
u-b_0-b_1|x|^\tau-
\left(\sum_{\{i:a_i=0\}}d_i x_i\right)|x|^{s_1-1}
\right|^2\,d\nu_A,
\end{aligned}
\]
which gives case~\textup{(iii)}. Thus cases~\textup{(i)}--\textup{(iii)}
follow from the corresponding lifted estimates.

It remains to prove cases~\textup{(iv)} and~\textup{(v)}. Assume now
that \(a_i>0\) for every \(i\). Since \(\widetilde u\) is radial in every
block \(\mathbb R^{a_i+1}\), all odd-degree spherical harmonic modes
vanish {: indeed, $\widetilde u$ is then invariant under $y\mapsto-y$, because
$-I_{a_i+1}\in O(a_i+1)$ acts trivially on block-radial functions, and the antipodal map
multiplies a spherical harmonic of degree $\ell$ by $(-1)^{\ell}$}. Hence its nonradial part has an expansion of the form
\[
\widetilde\Phi_u(r,\sigma)
=
\sum_{\ell=1}^{\infty}
f_{2\ell}(r)\varphi_{2\ell}(\sigma).
\]
In particular, the first admissible nonradial mode has degree two rather
than degree one.

The block-invariant harmonic polynomials of degree two are precisely of
the form
\[
\sum_{i=1}^n c_i|y_i|^2,
\qquad
\sum_{i=1}^n c_i(a_i+1)=0.
\]
Under the lifting, \(|y_i|=x_i\) and \(|y|=|x|\); therefore,
\[
\begin{aligned}
&\int_{\mathbb R^N}
\left|
\widetilde u-b-|y|^{s_2-2}\sum_{i=1}^n c_i|y_i|^2
\right|^2\,d\nu_N =
c_A
\int_{\mathbb R^n_{A,+}}
\left|
u-b-|x|^{s_2-2}\sum_{i=1}^n c_i x_i^2
\right|^2\,d\nu_A.
\end{aligned}
\]
The analogous identity holds for the balanced family
\[
b_0+b_1|y|^\tau
+
|y|^{s_2-2}\sum_{i=1}^n c_i|y_i|^2.
\]

We now repeat the argument used in the proof of
Theorem~\ref{thm:nonradial_stability}. There, the first two nonradial
eigenvalues were
\[
\lambda_1=N-1,\qquad \lambda_2=2N,
\]
which led to the parameters \(s_1\) and \(\alpha_1\). In the present
block-invariant subspace, the first two admissible nonradial eigenvalues
are instead
\[
\lambda_2=2N,\qquad \lambda_4=4(N+2).
\]
Thus the same argument applies after making the replacements
\[
\lambda_1\mapsto\lambda_2,\qquad
\lambda_2\mapsto\lambda_4,\qquad
s_1\mapsto s_2,\qquad
\alpha_1\mapsto\alpha_2.
\]
Equivalently, the quantities \(4(N-1)\) and \(8N\) in the formulas for
\(s_1\) and \(\alpha_1\) are replaced by \(8N\) and \(16(N+2)\),
respectively. Hence
\[
s_2
=
\frac{-(\beta+N-2)+\operatorname{sgn}(\tau)
\sqrt{(\beta+N-2)^2+8N}}{2},
\]
and
\[
\alpha_2
=
\frac{\operatorname{sgn}(\tau)}{2}
\left(
\sqrt{(\beta+N-2)^2+16(N+2)}
-
\sqrt{(\beta+N-2)^2+8N}
\right).
\]
Applying the proof of Theorem~\ref{thm:nonradial_stability} with these
replacements gives the lifted estimates corresponding to the degree-two
and balanced radial/degree-two cases. Descending them by the preceding
identities gives cases~\textup{(i)} cases~\textup{(iv)} and~\textup{(v)}, respectively.
\end{proof}


\section{Stability of the Caffarelli--Kohn--Nirenberg inequalities on orthants and on $\R^n$}\label{S:CKN_stability}

In this section we combine the exact identities of Section \ref{S:identities} with the
weighted Poincar\'e inequalities of Sections \ref{S:poincare} and
\ref{S:poincare_stability}, and thereby prove the stability estimates announced in the
introduction, both in their scale invariant and in their scale non-invariant forms. The
mechanism is the one described earlier: the identity converts the deficit into a weighted
Dirichlet energy, and the Poincar\'e inequality converts that energy into a distance to the
family of optimizers.

We first record the two identities in a form covering both signs of $\tau$ at once. Writing
$|-a+b+1|$ in place of $-a+b+1$ in the exponentials, Theorems \ref{Th:CKNO+_identity} and
\ref{Th:CKNO-_identity} combine into
    \begin{equation}\label{eq:combine_Identity}
   \begin{aligned}
   &\lambda^2\int_{\Rnkp}\frac{|\nabla u|^2}{|x|^{2b}}dx + \frac{1}{\lambda^2}\int_{\Rnkp}\frac{u^2}{|x|^{2a}}dx - 2C(n,a,b,k)\int_{\Rnkp}\frac{u^2}{|x|^{a+b+1}}dx \\ 
       &= \lambda^2 \int_{\Rnkp}\left|\nabla\left(\frac{u(x)}{\pxi |x|^m} e^{\frac{|x|^{-a+b+1}}{\lambda^2|-a+b+1|}}\right)\right|^2 e^{-2\frac{|x|^{-a+b+1}}{\lambda^2|-a+b+1|}}\frac{\pxis |x|^{2m}}{|x|^{2b}}dx.
   \end{aligned}
\end{equation}
and, in the same way, Theorems \ref{Th:Ckn+_identity} and \ref{Th:Ckn-_identity} combine, for the choice
    \[
\lambda = \left( \frac{\int_{\Rnkp}\frac{u^2}{|x|^{2a}}dx}
{   \int_{\Rnkp}\frac{|\nabla u|^2}{|x|^{2b}}dx} \right)^{\tfrac{1}{4}},
\]
into the following identity:
\begin{equation}\label{eq:combine_non_invarient_identity}
   \begin{aligned}
   &\left( \int_{\Rnkp}\frac{|\nabla u|^2}{|x|^{2b}}dx \right)^ \frac 12  \left(\int_{\Rnkp}\frac{u^2}{|x|^{2a}}dx\right)^\frac 12 - C(n,a,b,k)\int_{\Rnkp}\frac{u^2}{|x|^{a+b+1}}dx \\ 
       &= \frac{\lambda^2}{2} \int_{\Rnkp}\left|\nabla\left(\frac{u(x)}{\pxi |x|^m} e^{\frac{|x|^{-a+b+1}}{\lambda^2|-a+b+1|}}\right)\right|^2 e^{-2\frac{|x|^{-a+b+1}}{\lambda^2|-a+b+1|}}\frac{\pxis |x|^{2m}}{|x|^{2b}}dx.
   \end{aligned}
\end{equation}

Recall from Theorem \ref{Th:CKNIon_Orthant} that, for \(\tau=-a+b+1\neq 0\), the extremal functions for the sharp Caffarelli--Kohn--Nirenberg inequality are given by the following family.
\[
 E
:=
\left\{
A\left(\prod_{i=n-k+1}^{n}x_i\right)|x|^m
\exp\left(-\frac{B}{|\tau|}|x|^\tau\right)
:\ A\in\mathbb R,\ B>0
\right\},
\]
where, as in Theorem \ref{Th:CKNIon_Orthant},
\[
m+k=\frac{-(n-2b-2)+\sgn(\tau)\sqrt{(n-2b-2)^{2}+4k(n+k-2)}}{2}.
\]

To apply the weighted Poincaré inequalities to the identities established
above, we record the corresponding parameter dictionary. Under the lifting
procedure, we take $A=(0,\ldots,0,2,\ldots,2),$ with \(n-k\) zeros and \(k\) twos, we have \(|A|=2k\) and hence
\[
N=n+|A|=n+2k,\qquad
\tau=-a+b+1,\qquad
\delta=\frac{2}{\lambda^2|\tau|},\qquad
\beta=2m-2b.
\]
Since
\[
m+k=\frac{-(n-2b-2)+\operatorname{sgn}(\tau)\Theta}{2},
\]
we obtain
\[
\beta+N-2
=2m-2b+n+2k-2
=\operatorname{sgn}(\tau)\Theta.
\]
Thus the hypothesis
\[
\frac{\beta+N+\tau-2}{\tau}>0
\]
of Theorem
\ref{Th:poincare_monomial_weight}  and Theorem \ref{Th:stability_poincare_monomialweight} holds for either sign of $\tau$.\\ 
Substituting $N=n+2k \text{ and } \beta+N-2=\operatorname{sgn}(\tau)\Theta$
into the stability parameters of Theorem \ref{Th:stability_poincare_monomialweight}, we obtain
\[
s_q=\frac{\operatorname{sgn}(\tau)}{2}
\left(\sqrt{\Theta^2+4\lambda_q^{(N)}}-\Theta\right),
\qquad q=1,2,
\]
and
\[
\alpha_1=\frac{\operatorname{sgn}(\tau)}{2}
\left(\sqrt{\Theta^2+4\lambda_2^{(N)}}
-\sqrt{\Theta^2+4\lambda_1^{(N)}}\right),
\qquad
\alpha_2=\frac{\operatorname{sgn}(\tau)}{2}
\left(\sqrt{\Theta^2+4\lambda_4^{(N)}}
-\sqrt{\Theta^2+4\lambda_2^{(N)}}\right).
\]
These are precisely the stability parameters defined in \eqref{eq:s1s2} {and}
\eqref{eq:alpha1alpha2}. Consequently,
\begin{equation}\label{s1s2a1a2}
  \frac{\lambda^{2}}{2}\,\delta\tau^{2}=|\tau| ,
\quad
\frac{\lambda^{2}}{2}\,\delta\tau s_1=|s_1| ,
\quad 
\frac{\lambda^{2}}{2}\,\delta\tau\alpha_1=|\alpha_1| ,
\quad
\frac{\lambda^{2}}{2}\,\delta\tau s_2= |s_2| ,
\quad 
\frac{\lambda^{2}}{2}\,\delta\tau\alpha_2=|\alpha_2| ,  
\end{equation}

For convenience, {this notation} will be used throughout this section.

\begin{theorem}[Stability of the Caffarelli--Kohn--Nirenberg inequality]\label{Th:stability_ckn_scale_noninvariant}
Let $u\in S_{a,b}(\mathbb{R}^n_{k,+})$, assume $\tau:=-a+b+1\neq 0$, and let
\[
C(n,a,b,k)=\frac{|\tau|+\sqrt{(n-2b-2)^{2}+4k(n+k-2)}}{2}
\]
be the sharp constant of Theorem \ref{Th:CKNIon_Orthant}. Then
\[
\begin{aligned}
&\int_{\mathbb{R}^n_{k,+}}\frac{|\nabla u|^2}{|x|^{2b}}\,dx
+\int_{\mathbb{R}^n_{k,+}}\frac{u^2}{|x|^{2a}}\,dx
-2C(n,a,b,k)\int_{\mathbb{R}^n_{k,+}}\frac{u^2}{|x|^{a+b+1}}\,dx  \\
&\qquad \qquad \qquad \qquad \qquad \qquad\geq 
2 C_2(n,a,b,k)
\inf_{v \in E }
\int_{\mathbb{R}^n_{k,+}}
\left|u(x)-v(x)\right|^2\frac1{|x|^{a+b+1}}
dx .
\end{aligned}
\]
where  $C_2(n,a,b,k)={\begin{cases}
    \min\{|\tau|, |s_1| \}, & k<n, \\
    \min\{|\tau|, |s_2| \}, & k=n.
\end{cases}}$

\end{theorem}
    \begin{proof}
By the identity \eqref{eq:combine_Identity}, we have
\begin{align}
   &\lambda^2\int_{\Rnkp}\frac{|\nabla u|^2}{|x|^{2b}}dx + \frac{1}{\lambda^2}\int_{\Rnkp}\frac{u^2}{|x|^{2a}}dx - 2C(n,a,b,k)\int_{\Rnkp}\frac{u^2}{|x|^{a+b+1}}dx \notag \\ 
       &\qquad = \lambda^2 \int_{\Rnkp}\left|\nabla\left(\frac{u(x)}{\pxi |x|^m} e^{\frac{|x|^{-a+b+1}}{\lambda^2|-a+b+1|}}\right)\right|^2 e^{-2\frac{|x|^{-a+b+1}}{\lambda^2|-a+b+1|}}\frac{\pxis |x|^{2m}}{|x|^{2b}}dx.\notag \\
\end{align}

Therefore applying the weighted Poincaré inequality from Theorem \ref{Th:poincare_monomial_weight} to the right side of the identity we get,
\begin{align}
    &\lambda^2\int_{\Rnkp}\frac{|\nabla u|^2}{|x|^{2b}}dx + \frac{1}{\lambda^2}\int_{\Rnkp}\frac{u^2}{|x|^{2a}}dx - 2C(n,a,b,k)\int_{\Rnkp}\frac{u^2}{|x|^{a+b+1}}dx \notag \\ 
    &\qquad\geq \lambda^2 C_2(n,A,\beta,\tau,\delta)
\inf_{A\in \mathbb R}\int_{\mathbb{R}^n_{k,+}}\left|u(x)-A|x|^me^{-\frac{|x|^\tau}{\lambda^2|\tau|}}\pxi\right|^2\frac1{|x|^{a+b+1}}dx  \notag\\
\end{align}

where
\begin{align*}
    C_2(n,A,\beta,\tau,\delta)={\begin{cases}
  \min\{\delta\tau^{2},\ \delta\tau s_2\},  &  k=n,\\
  \min\{\delta\tau^{2},\ \delta\tau s_1\}, & k<n,
\end{cases}} 
\end{align*}

Using equation \eqref{s1s2a1a2}, \eqref{eq:s1s2}
\begin{align*}
   \lambda^2 C_2(n,A,\beta,\tau,\delta)=2C_2(n,a,b,k)={\begin{cases}
    2\min\{|\tau|, |s_1| \}, & k<n, \\
    2\min\{|\tau|, |s_2| \}, & k=n,
\end{cases}}
\end{align*}

And taking $\lambda=1$ gives the asserted inequality.
\end{proof}

We now turn to the scale invariant form of the estimate, that is, to Theorem
\ref{Th:stability_ckn} of the introduction. It is obtained from the identity
\eqref{eq:combine_non_invarient_identity} by exactly the same argument, the only difference
being that the optimal choice of $\lambda$ has already been made in that identity.
\begin{proof}[Proof of Theorem \ref{Th:stability_ckn}]
Arguing exactly as in the proof of the preceding theorem and applying Theorem~\ref{Th:poincare_monomial_weight} to the identity \eqref{eq:combine_non_invarient_identity}, we obtain the scale-invariant stability in the case 
\begin{equation}
   \begin{aligned}
   &\left( \int_{\Rnkp}\frac{|\nabla u|^2}{|x|^{2b}}dx \right)^ \frac 12  \left(\int_{\Rnkp}\frac{u^2}{|x|^{2a}}dx\right)^\frac 12 - C(n,a,b,k)\int_{\Rnkp}\frac{u^2}{|x|^{a+b+1}}dx \\ 
       &= \frac{\lambda^2}{2} \int_{\Rnkp}\left|\nabla\left(\frac{u(x)}{\pxi |x|^m} e^{\frac{|x|^{-a+b+1}}{\lambda^2|-a+b+1|}}\right)\right|^2 e^{-2\frac{|x|^{-a+b+1}}{\lambda^2|-a+b+1|}}\frac{\pxis |x|^{2m}}{|x|^{2b}}dx\\
       &\qquad\geq  \frac{\lambda^2}{2} C_2(n,A,\beta,\tau,\delta)
\inf_{A\in \mathbb R}\int_{\mathbb{R}^n_{k,+}}\left|u(x)-A|x|^me^{-\frac{|x|^\tau}{\lambda^2|\tau|}}\pxi\right|^2\frac1{|x|^{a+b+1}}dx  \\
       &\qquad\geq \frac{\lambda^2}{2} C_2(n,A,\beta,\tau,\delta)
\inf_{v\in E}\int_{\mathbb{R}^n_{k,+}}\left|u(x)-v(x)\right|^2\frac1{|x|^{a+b+1}}dx .
   \end{aligned}
\end{equation}

where
\begin{align*}
    C_2(n,A,\beta,\tau,\delta)={\begin{cases}
  \min\{\delta\tau^{2},\ \delta\tau s_2\},  &  k=n,\\
  \min\{\delta\tau^{2},\ \delta\tau s_1\}, & k<n,
\end{cases}} 
\end{align*}

Using equation \eqref{s1s2a1a2}, \eqref{eq:s1s2}
\begin{align*}
   \frac{\lambda^2}{2} C_2(n,A,\beta,\tau,\delta)={C_2(n,a,b,k)=\begin{cases}
    \min\{|\tau|, |s_1| \}, & k<n, \\
    \min\{|\tau|, |s_2| \}, & k=n,
\end{cases}}
\end{align*}

\end{proof}

Taking $k=0$ in Theorem \ref{Th:stability_ckn}, we obtain the corresponding stability
estimate on the whole Euclidean space, which appears to be new even in that setting.

\begin{corollary}[Stability of the $L^2$-CKN inequality on $\mathbb R^n$]
Let $n\ge 1$, $a,b\in\mathbb R$, set $\tau:=1+b-a$, and assume $\tau\neq0$. Define
\[
\rho(u):=
\left(\int_{\mathbb R^n}\frac{|\nabla u|^2}{|x|^{2b}}\,dx\right)^{1/2}
\left(\int_{\mathbb R^n}\frac{|u|^2}{|x|^{2a}}\,dx\right)^{1/2}
-
C(n,a,b)\int_{\mathbb R^n}\frac{|u|^2}{|x|^{a+b+1}}\,dx.
\]
Then
\begin{equation}
 \rho(u)\ge \min \left\{ |\tau|,\ \frac{\sqrt{(n-2b-2)^2+4(n-1)}-|n-2b-2|}{2} \right\}
\inf_{c\in\mathbb R,\,B>0} \int_{\mathbb R^n}
\frac{\left|u-c\,{|x|^{m}}e^{-B\frac{|x|^\tau}{|\tau|}}\right|^2}{|x|^{a+b+1}}\,dx ,
\end{equation}
where
\[
m=\frac{-(n-2b-2)+\sgn(\tau)|n-2b-2|}{2}
=
\begin{cases}
0, & \text{in the regions }\mathcal A,\\
2(b+1)-n, & \text{in the regions }\mathcal B,
\end{cases}
\]
so that the extremal family is the one attached to the corresponding region in Theorem
\ref{Th:CKNI}.

\end{corollary}

Having established the first stability estimate for the Caffarelli–Kohn–Nirenberg inequality, we now investigate the stability of this estimate itself. By applying the stability result for the associated weighted Poincaré inequality, we obtain a further quantitative control of the deficit, which leads to the second stability result for the Caffarelli–Kohn–Nirenberg inequality.

\begin{proof}[Proof of Theorem \ref{Th:second_stability_ckn}]
By the identity \eqref{eq:combine_non_invarient_identity}, we have

\begin{equation}
   \begin{aligned}
   \rho(u)&=
\left( \int_{\Rnkp}\frac{|\nabla u|^2}{|x|^{2b}}dx \right)^ \frac 12  \left(\int_{\Rnkp}\frac{u^2}{|x|^{2a}}dx\right)^\frac 12 - C(n,a,b,k)\int_{\Rnkp}\frac{u^2}{|x|^{a+b+1}}dx \\ 
       &= \frac{\lambda^2}{2} \int_{\Rnkp}\left|\nabla\left(\frac{u(x)}{\pxi |x|^m} e^{\frac{|x|^{-a+b+1}}{\lambda^2|-a+b+1|}}\right)\right|^2 e^{-2\frac{|x|^{-a+b+1}}{\lambda^2|-a+b+1|}}\frac{\pxis |x|^{2m}}{|x|^{2b}}dx.
   \end{aligned}
\end{equation}
Since the right-hand side is precisely of the Poincaré type, with measure of the form
$d\mu(x)=e^{-\delta |x|^\tau}x^A|x|^\beta,dx,$

we may apply the stability result of Theorem \ref{Th:stability_poincare_monomialweight} directly to this term, thereby obtaining the corresponding second stability estimate for the Caffarelli–Kohn–Nirenberg inequality on $\Rnkp$. Set
\[
w(x)
:=
\frac{u(x)}{\pxi |x|^m}
e^{\frac{|x|^\tau}{\lambda^2|\tau|}}, \; \text{ and } \; d\mu_A
=
\pxis
e^{-\frac{2|x|^\tau}{\lambda^2|\tau|}}
|x|^{2m-2b}\,dx .
\]
and
\[
d\nu_A
=
\pxis
e^{-\frac{2|x|^\tau}{\lambda^2|\tau|}}
|x|^{2m-2b+\tau-2}\,dx .
\]
Then
\[
\rho(u)
=
\frac{\lambda^2}{2}
\int_{\Rnkp}|\nabla w|^2\,d\mu_A,
\]
Since $\tau=-a+b+1$, then for every admissible function $\psi$,
\[
\int_{\Rnkp}|w-\psi|^2\,d\nu_A
=
\int_{\Rnkp}
\left|
u-\psi\,\pxi |x|^m
e^{-\frac{|x|^\tau}{\lambda^2|\tau|}}
\right|^2
\frac{dx}{|x|^{a+b+1}}.
\]

We now apply Theorem~\ref{Th:stability_poincare_monomialweight} to 
$$\frac{\lambda^2}{2}
\int_{\Rnkp}|\nabla w|^2\,d\mu_A - \frac{\lambda^2}{2}C_2(n,A,\beta,\tau,\delta) \int_{\Rnkp}|w-d|^2\,d\nu_A$$

Using equation \eqref{s1s2a1a2}, \eqref{eq:s1s2}, \eqref{eq:alpha1alpha2}
and noticing 

\begin{align*}
   \frac{\lambda^2}{2} C_2(n,A,\beta,\tau,\delta)=C_2(n,a,b,k)={\begin{cases}
    \min\{|\tau|, |s_1| \}, & k<n, \\
    \min\{|\tau|, |s_2| \}, & k=n,
\end{cases}}
\end{align*}

From Theorem \ref{Th:stability_ckn} we have the following.

First suppose that {$k<n$ and  $|\tau|<|s_1|$,  or  $k=n$ and $|\tau|<|s_2|$}; then Theorem~\ref{Th:stability_poincare_monomialweight} gives
\[
\begin{aligned}
&\frac{\lambda^2}{2}\int_{\Rnkp}|\nabla w|^2\,d\mu_A
-
|\tau|
\inf_{d\in\mathbb R}
\int_{\Rnkp}|w-d|^2\,d\nu_A
\\
&\qquad\ge
{\min\{|\tau|,|s_q|-|\tau|\}}
\inf_{A_0,A_1\in\mathbb R}
\int_{\Rnkp}
|w-A_0-A_1|x|^\tau|^2\,d\nu_A .
\end{aligned}
\]

Next suppose $|\tau|>|s_1|$, and $k<n$ then Theorem~\ref{Th:stability_poincare_monomialweight} gives

\[
\begin{aligned}
&\frac{\lambda^2}{2}\int_{\Rnkp}|\nabla w|^2\,d\mu_A
-|s_1|\inf_{d\in\mathbb R} \int_{\Rnkp}|w-d|^2\,d\nu_A
\\
&\qquad\ge
\min\{|\tau|-|s_1|,|\tau|,|\alpha_1|\}
\inf_{A_0\in\mathbb R,\ b_i\in\mathbb R}
\int_{\Rnkp}
\left|
w-A_0-
\left(\sum_{i=1}^{n-k}b_i x_i\right)
|x|^{s_1-1}
\right|^2
\,d\nu_A .
\end{aligned}
\]

Now suppose $|\tau|=|s_1|$ and $k<n$, then the radial and first non-radial modes occur at the same level. Hence Theorem~\ref{Th:stability_poincare_monomialweight} gives
\[
\begin{aligned}
&\frac{\lambda^2}{2}\int_{\Rnkp}|\nabla w|^2\,d\mu_A
-
|\tau|
\inf_{d\in\mathbb R}
\int_{\Rnkp}|w-d|^2\,d\nu_A
\\
&\qquad\ge
\min\{|\tau|,|\alpha_1|\}
\inf_{A_0,A_1\in\mathbb R,\ b_i\in\mathbb R}
\int_{\Rnkp}
\left|
w-A_0-A_1|x|^\tau
-
\left(\sum_{i=1}^{n-k}b_i x_i\right)
|x|^{s_1-1}
\right|^2
\,d\nu_A .
\end{aligned}
\]

Now suppose that $|\tau|>|s_2|$ and {$k=n$}; then Theorem~\ref{Th:stability_poincare_monomialweight} gives
\[
\begin{aligned}
&\frac{\lambda^2}{2}\int_{\Rnkp}|\nabla w|^2\,d\mu_A
-
|s_2|
\inf_{d\in\mathbb R}
\int_{\Rnkp}|w-d|^2\,d\nu_A
\\
&\qquad\ge
\min\{|\tau|, |\tau|-|s_2|,\ |\alpha_2|\}
\inf_{\substack{A_{0}\in\mathbb R,\\ {\sum_{i=1}^{n} c_{i}=0}} }
\int_{\Rnkp}
\left|
w-A_0-\left(\sum_{i=1}^{n}c_i x_i^{2}\right)|x|^{s_2-2}\right|^2\,d\nu_A .
\end{aligned}
\]

Finally, if $|\tau|=|s_2|$ and {$k=n$}, then Theorem~\ref{Th:stability_poincare_monomialweight} gives
\[
\begin{aligned}
&\frac{\lambda^2}{2}\int_{\Rnkp}|\nabla w|^2\,d\mu_A
-
|s_2|
\inf_{d\in\mathbb R}
\int_{\Rnkp}|w-d|^2\,d\nu_A
\\
&\qquad\ge
\min\{|\tau|, \ |\alpha_2|\}
\inf_{\substack{A_{0},A_1\in\mathbb R,\\ {\sum_{i=1}^{n} c_{i}=0}}}
\int_{\Rnkp}
\left|
w-A_0-A_1|x|^\tau
-
{\left(\sum_{i=1}^{n}c_i x_i^{2}\right)
|x|^{s_2-2}}
\right|^2
\,d\nu_A .
\end{aligned}
\]
Returning to $u$, and then taking the infimum over the parameter $B > 0$ gives the  estimates. This completes the proof.
\end{proof}

In the same way, taking $k=0$ in Theorem \ref{Th:second_stability_ckn} yields the second
stability estimate on $\mathbb R^{n}$. Note that the hypothesis $k<n$ is then automatic.

\begin{corollary}[Second stability of the Caffarelli--Kohn--Nirenberg inequality on $\mathbb R^n$]
Let $n\ge2$, let $u\in S_{a,b}(\mathbb{R}^n)$, and assume
\[
\tau:=-a+b+1\neq 0.
\]
Define the Caffarelli--Kohn--Nirenberg {deficit} by
\[
\rho(u)
=
\left(\int_{\mathbb{R}^n}\frac{|\nabla u|^2}{|x|^{2b}}\,dx\right)^{\frac12}
\left(\int_{\mathbb{R}^n}\frac{u^2}{|x|^{2a}}\,dx\right)^{\frac12}
-C(n,a,b,0)
\int_{\mathbb{R}^n}\frac{u^2}{|x|^{a+b+1}}\,dx .
\]
Then the following estimates hold.

\begin{enumerate}
\item If $|\tau|<|s_1|$, then
\[
\begin{aligned}
&\rho(u)
-
|\tau|
\inf_{\phi\in E}
\int_{\mathbb{R}^n}
|u-\phi|^2\frac{dx}{|x|^{a+b+1}}
\\
&\qquad\ge
\min\{|\tau|,|s_1|-|\tau|\}
\inf_{A_0,A_1\in\mathbb R,\ B>0}
\int_{\mathbb{R}^n}
\left|
u-
(A_0+A_1|x|^\tau)
e^{-B\frac{|x|^\tau}{|\tau|}}
|x|^m
\right|^2
\frac{dx}{|x|^{a+b+1}} .
\end{aligned}
\]

\item If $|\tau|>|s_1|$, then
\[
\begin{aligned}
&\rho(u)
-
|s_1|
\inf_{\phi\in E}
\int_{\mathbb{R}^n}
|u-\phi|^2\frac{dx}{|x|^{a+b+1}}
\\
&\qquad\ge
\min\{|\tau|-|s_1|,|\tau|,|\alpha_1|\}
\\
&\qquad\qquad\times
\inf_{A_0\in\mathbb R,\ b_i\in\mathbb R,\ B>0}
\int_{\mathbb{R}^n}
\left|
u-
\left(
A_0+
\sum_{i=1}^{n}b_i x_i\,|x|^{s_1-1}
\right)
e^{-B\frac{|x|^\tau}{|\tau|}}
|x|^m
\right|^2
\frac{dx}{|x|^{a+b+1}} .
\end{aligned}
\]

\item If $|\tau|=|s_1|$, then
\[
\begin{aligned}
&\rho(u)
-
|\tau|
\inf_{\phi\in E}
\int_{\mathbb{R}^n}
|u-\phi|^2\frac{dx}{|x|^{a+b+1}}
\\
&\qquad\ge
\min\{|\tau|,|\alpha_1|\}
\\
&\qquad\qquad\times
\inf_{A_0,A_1\in\mathbb R,\ b_i\in\mathbb R,\ B>0}
\int_{\mathbb{R}^n}
\left|
u-
\left(
A_0+A_1|x|^\tau+
\sum_{i=1}^{n}b_i x_i\,|x|^{s_1-1}
\right)
e^{-B\frac{|x|^\tau}{|\tau|}}
|x|^m
\right|^2
\frac{dx}{|x|^{a+b+1}} .
\end{aligned}
\]
\end{enumerate}
\end{corollary}

\begin{proof}
This is the specialization of Theorem~\ref{Th:second_stability_ckn} to the case $k=0$. In this case
\[
\mathbb R^n_{k,+}=\mathbb R^n,
\qquad
C(n,a,b,k)=C(n,a,b,0),
\]
and
\[
s_1
=\sgn(\tau)
\frac{
{-|n-2b-2|}
+
\sqrt{(n-2b-2)^2+4\lambda_1^{(n)}}
}{2},
\]
while
\[
\alpha_1
=\sgn(\tau)
\frac{
\sqrt{(n-2b-2)^2+4\lambda_2^{(n)}}
-
\sqrt{(n-2b-2)^2+4\lambda_1^{(n)}}
}{2}.
\]
Moreover, since $n-k=n$, the first non-radial eigenspace is generated by
\[
x_i\,|x|^{s_1-1},
\qquad i=1,\dots,n.
\]
Substituting $k=0$ into the three cases of the theorem yields the stated inequalities.
\end{proof}


\section*{Appendix: density of smooth functions supported away from the origin}\label{S:appendix}

In this appendix we collect the density results used throughout the paper. They are what
allows us to prove all of our inequalities for functions in $C_c^\infty$ supported away
from the origin, and then to extend them to the natural energy spaces. The point requiring
proof is that removing a neighbourhood of the origin costs nothing in the weighted norms,
which is a consequence of the fact that the capacity of a point vanishes in the relevant
range of the parameters.

\begin{proposition}[Density of test functions away from the origin]
Let \(\delta>0\) and \(\tau\neq0\). Define
\[
\|u\|_X^2
:=
\int_{\mathbb R^n}
|\nabla u(x)|^2
e^{-\delta|x|^\tau}|x|^\beta\,dx
+
\int_{\mathbb R^n}
|u(x)|^2
e^{-\delta|x|^\tau}|x|^{\beta+\tau-2}\,dx.
\]
Assume either
\[
\tau<0,
\]
or
\[
\tau>0,\qquad
\beta+n+\tau-2>0,
\qquad
\beta+n-2\geq0.
\]
Then every function in \(C_c^\infty(\mathbb R^n)\) has finite
\(\|\cdot\|_X\)-norm.

Let
\[
X
:=
\overline{C_c^\infty(\mathbb R^n)}^{\,\|\cdot\|_X}.
\]
Then
\[
\overline{C_c^\infty(\mathbb R^n\setminus\{0\})}^{\,\|\cdot\|_X}
=
X.
\]
\end{proposition}

\begin{proof}
It is enough to prove that every
\(u\in C_c^\infty(\mathbb R^n)\) can be approximated in the
\(\|\cdot\|_X\)-norm by functions in
\(C_c^\infty(\mathbb R^n\setminus\{0\})\).

We distinguish three cases.

\medskip

\noindent
\textbf{Case 1: \(\tau<0\).}

Choose a radial cutoff \(\eta\in C^\infty([0,\infty))\) satisfying
\[
0\leq\eta\leq1,
\qquad
\eta(r)=0\quad\text{for }0\leq r\leq1,
\qquad
\eta(r)=1\quad\text{for }r\geq2,
\]
and define
\[
\eta_\varepsilon(x)
:=
\eta\!\left(\frac{|x|}{\varepsilon}\right).
\]
Then
\[
|\nabla\eta_\varepsilon|
\leq \frac{C}{\varepsilon}.
\]
Set
\[
u_\varepsilon:=\eta_\varepsilon u.
\]
Thus
\[
u_\varepsilon\in C_c^\infty(\mathbb R^n\setminus\{0\}).
\]

Since
\[
u-u_\varepsilon=(1-\eta_\varepsilon)u,
\]
the lower-order term satisfies
\[
\begin{aligned}
&
\int_{\mathbb R^n}
|u-u_\varepsilon|^2
e^{-\delta|x|^\tau}|x|^{\beta+\tau-2}\,dx
\\
&\leq
\int_{B_{2\varepsilon}}
|u|^2e^{-\delta|x|^\tau}|x|^{\beta+\tau-2}\,dx
\longrightarrow0
\end{aligned}
\]
by the absolute continuity of the integral.

For the gradient term,
\[
\nabla(u-u_\varepsilon)
=
(1-\eta_\varepsilon)\nabla u
-u\nabla\eta_\varepsilon.
\]
Therefore,
\[
\begin{aligned}
&
\int_{\mathbb R^n}
|\nabla(u-u_\varepsilon)|^2
e^{-\delta|x|^\tau}|x|^\beta\,dx
\\
&\leq
2I_\varepsilon+2J_\varepsilon,
\end{aligned}
\]
where
\[
I_\varepsilon
:=
\int_{B_{2\varepsilon}}
|\nabla u|^2e^{-\delta|x|^\tau}|x|^\beta\,dx
\longrightarrow0
\]
by absolute continuity, and
\[
J_\varepsilon
:=
\int_{B_{2\varepsilon}\setminus B_\varepsilon}
|u|^2|\nabla\eta_\varepsilon|^2
e^{-\delta|x|^\tau}|x|^\beta\,dx.
\]

Since \(u\) is bounded,
\[
J_\varepsilon
\leq
\frac{C}{\varepsilon^2}
\int_\varepsilon^{2\varepsilon}
e^{-\delta r^\tau}r^{\beta+n-1}\,dr.
\]
Because \(\tau<0\),
\[
e^{-\delta r^\tau}
=
e^{-\delta r^{-|\tau|}}.
\]
For every \(M>0\), there exists \(C_M>0\) such that
\[
e^{-\delta r^{-|\tau|}}
\leq C_Mr^M,
\qquad 0<r<1.
\]
Consequently,
\[
J_\varepsilon
\leq
C_M\varepsilon^{-2}
\int_\varepsilon^{2\varepsilon}
r^{\beta+n+M-1}\,dr
\leq
C_M\varepsilon^{\beta+n+M-2}.
\]
Choosing \(M\) sufficiently large yields
\[
J_\varepsilon\longrightarrow0.
\]
Hence
\[
\|u_\varepsilon-u\|_X\longrightarrow0.
\]

\medskip

\noindent
\textbf{Case 2: \(\tau>0\) and \(\beta+n-2>0\).}

Use the same cutoff \(\eta_\varepsilon\) and set
\[
u_\varepsilon=\eta_\varepsilon u.
\]
The lower-order term tends to zero by absolute continuity:
\[
\int_{\mathbb R^n}
|u-u_\varepsilon|^2
e^{-\delta|x|^\tau}|x|^{\beta+\tau-2}\,dx
\longrightarrow0.
\]

For the gradient term, the part containing
\((1-\eta_\varepsilon)\nabla u\) also tends to zero by absolute
continuity. For the cutoff-gradient term, since
\(e^{-\delta|x|^\tau}\leq1\),
\[
\begin{aligned}
J_\varepsilon
&:=
\int_{B_{2\varepsilon}\setminus B_\varepsilon}
|u|^2|\nabla\eta_\varepsilon|^2
e^{-\delta|x|^\tau}|x|^\beta\,dx
\\
&\leq
\frac{C}{\varepsilon^2}
\int_{\varepsilon<|x|<2\varepsilon}|x|^\beta\,dx
\\
&\leq
C\varepsilon^{\beta+n-2}.
\end{aligned}
\]
Since
\[
\beta+n-2>0,
\]
we have
\[
J_\varepsilon\longrightarrow0.
\]
Therefore
\[
\|u_\varepsilon-u\|_X\longrightarrow0.
\]

\medskip

\noindent
\textbf{Case 3: \(\tau>0\) and \(\beta+n-2=0\).}

In this case,
\[
\beta=2-n.
\]
The assumption
\[
\beta+n+\tau-2>0
\]
then reduces to
\[
\tau>0.
\]

For \(0<\varepsilon<1/4\), choose a smooth radial logarithmic cutoff
\(\eta_\varepsilon\) satisfying
\[
0\leq\eta_\varepsilon\leq1,
\]
\[
\eta_\varepsilon=0
\quad\text{on }B_\varepsilon,
\qquad
\eta_\varepsilon=1
\quad\text{on }\mathbb R^n\setminus B_{\sqrt{\varepsilon}},
\]
and
\[
|\nabla\eta_\varepsilon(x)|
\leq
\frac{C}{|x|\log(1/\varepsilon)}
\]
for
\[
\varepsilon<|x|<\sqrt{\varepsilon}.
\]
Set
\[
u_\varepsilon:=\eta_\varepsilon u.
\]

Because \(1-\eta_\varepsilon\) is supported in
\(B_{\sqrt{\varepsilon}}\), the lower-order term tends to zero by
absolute continuity. Similarly,
\[
\int_{B_{\sqrt{\varepsilon}}}
|\nabla u|^2e^{-\delta|x|^\tau}|x|^\beta\,dx
\longrightarrow0.
\]

It remains to estimate
\[
J_\varepsilon
:=
\int_{\varepsilon<|x|<\sqrt{\varepsilon}}
|u|^2|\nabla\eta_\varepsilon|^2
e^{-\delta|x|^\tau}|x|^\beta\,dx.
\]
Using the boundedness of \(u\) and
\(e^{-\delta|x|^\tau}\leq1\), we obtain
\[
\begin{aligned}
J_\varepsilon
&\leq
\frac{C}{\log^2(1/\varepsilon)}
\int_{\varepsilon<|x|<\sqrt{\varepsilon}}
|x|^{\beta-2}\,dx
\\
&=
\frac{C}{\log^2(1/\varepsilon)}
\int_\varepsilon^{\sqrt{\varepsilon}}
r^{\beta+n-3}\,dr.
\end{aligned}
\]
Since
\[
\beta+n-2=0,
\]
we have
\[
\beta+n-3=-1.
\]
Thus
\[
\begin{aligned}
J_\varepsilon
&\leq
\frac{C}{\log^2(1/\varepsilon)}
\int_\varepsilon^{\sqrt{\varepsilon}}\frac{dr}{r}
\\
&\leq
\frac{C}{\log(1/\varepsilon)}
\longrightarrow0.
\end{aligned}
\]
Therefore,
\[
\|u_\varepsilon-u\|_X\longrightarrow0.
\]

In all three cases, every
\(u\in C_c^\infty(\mathbb R^n)\) can be approximated in \(X\) by
functions in
\(C_c^\infty(\mathbb R^n\setminus\{0\})\).
Since \(C_c^\infty(\mathbb R^n)\) is dense in \(X\) by definition,
we conclude that
\[
\overline{C_c^\infty(\mathbb R^n\setminus\{0\})}^{\,\|\cdot\|_X}
=
X.
\]
\end{proof}

\begin{proposition}\label{P:Sab_completion}
Let
\[
\|u\|_{a,b}^{2}
:=
\int_{\mathbb R^n}\frac{|u|^2}{|x|^{2a}}\,dx
+
\int_{\mathbb R^n}\frac{|\nabla u|^2}{|x|^{2b}}\,dx
\]
and
\[
\mathcal A
:=
\{u\in C_c^\infty(\mathbb R^n):\|u\|_{a,b}<\infty\}.
\]
If \(a\ge n/2\) or \(b\le (n-2)/2\), then
\[
S_{a,b}(\R^n)
=\overline{C_c^\infty(\mathbb R^n\setminus\{0\})}^{\,\|\cdot\|_{a,b}}
=\overline{\mathcal A}^{\,\|\cdot\|_{a,b}}.
\]
\end{proposition}

\begin{proof}
Since \(C_c^\infty(\mathbb R^n\setminus\{0\})\subset \mathcal A\), one inclusion is immediate. We prove the other one. Let \(u\in\mathcal A\). Choose \(\eta_\varepsilon\in C^\infty(\mathbb R^n)\) such that \(\eta_\varepsilon=0\) in \(B_\varepsilon\), \(\eta_\varepsilon=1\) outside \(B_{2\varepsilon}\), and \(|\nabla\eta_\varepsilon|\le C/\varepsilon\). Set \(u_\varepsilon=\eta_\varepsilon u\). Then \(u_\varepsilon\in C_c^\infty(\mathbb R^n\setminus\{0\})\).

First,
\[
\int_{\mathbb R^n}\frac{|u_\varepsilon-u|^2}{|x|^{2a}}\,dx
\le
\int_{B_{2\varepsilon}}\frac{|u|^2}{|x|^{2a}}\,dx
\to 0,
\]
because \(|u|^2|x|^{-2a}\in L^1(\mathbb R^n)\). Moreover, this integrability implies the following: if \(u(0)\ne0\), then near the origin
\[
\int_{B_\delta}\frac{|u|^2}{|x|^{2a}}\,dx
\gtrsim
\int_0^\delta r^{n-1-2a}\,dr,
\]
which is finite only if \(a<n/2\). Hence, whenever \(a\ge n/2\), every \(u\in\mathcal A\) satisfies \(u(0)=0\).

It remains to control the gradient part. Since
\[
\nabla(u_\varepsilon-u)
=
(\eta_\varepsilon-1)\nabla u
+
u\nabla\eta_\varepsilon,
\]
we have
\[
\int_{\mathbb R^n}\frac{|\nabla(u_\varepsilon-u)|^2}{|x|^{2b}}\,dx
\le
C\int_{B_{2\varepsilon}}\frac{|\nabla u|^2}{|x|^{2b}}\,dx
+
C\int_{B_{2\varepsilon}\setminus B_\varepsilon}
\frac{|u|^2|\nabla\eta_\varepsilon|^2}{|x|^{2b}}\,dx .
\]
The first term tends to \(0\). We estimate the second one.

If \(u(0)\ne0\), then necessarily \(a<n/2\). Thus, by the assumption of the proposition, we must have \(b\le (n-2)/2\). If \(b<(n-2)/2\), then \(u\) is bounded near \(0\), and
\[
\int_{B_{2\varepsilon}\setminus B_\varepsilon}
\frac{|u|^2|\nabla\eta_\varepsilon|^2}{|x|^{2b}}\,dx
\le
C\varepsilon^{-2}\int_\varepsilon^{2\varepsilon}r^{n-1-2b}\,dr
\le
C\varepsilon^{n-2b-2}
\to0.
\]
If \(b=(n-2)/2\), we replace \(\eta_\varepsilon\) by the standard logarithmic cutoff, satisfying \(\eta_\varepsilon=0\) in \(B_{\varepsilon^2}\), \(\eta_\varepsilon=1\) outside \(B_\varepsilon\), and
\[
|\nabla\eta_\varepsilon|
\le
\frac{C}{|x|\log(1/\varepsilon)} .
\]
Then
\[
\int_{\varepsilon^2<|x|<\varepsilon}
\frac{|u|^2|\nabla\eta_\varepsilon|^2}{|x|^{2b}}\,dx
\le
\frac{C}{(\log(1/\varepsilon))^2}
\int_{\varepsilon^2}^{\varepsilon}\frac{dr}{r}
\to0.
\]
Finally, suppose \(u(0)=0\). If \(u\) is not flat at \(0\), let \(m\ge1\) be
the first nonzero order in the Taylor expansion of \(u\) at \(0\), so that
\(|u(x)|\le C|x|^m\) near \(0\). Let \(P_m\) denote the first nonzero
homogeneous term of the Taylor expansion of \(u\) at the origin. Since
\(P_m\) is nonconstant, \(\nabla P_m\) is a nonzero homogeneous polynomial
of degree \(m-1\). Choose \(\omega_0\in\mathbb S^{n-1}\) such that
\(\nabla P_m(\omega_0)\neq0\). Since
\[
\nabla u(x)=\nabla P_m(x)+O(|x|^m)
\]
near \(0\), homogeneity and continuity provide an open cone \(\Gamma\)
containing the direction \(\omega_0\) and constants \(c,r_0>0\) such that
\[
|\nabla u(x)|\ge c|x|^{m-1},
\qquad x\in\Gamma,\quad 0<|x|<r_0.
\]
Hence
\[
\infty>
\int_{\Gamma\cap B_{r_0}}
\frac{|\nabla u|^2}{|x|^{2b}}\,dx
\ge
c^2\bigl|\Gamma\cap\mathbb S^{n-1}\bigr|
\int_0^{r_0}r^{n+2m-2b-3}\,dr.
\]
The last integral is finite only if
\[
n+2m-2b-2>0.
\]
Therefore,
\[
\int_{B_{2\varepsilon}\setminus B_\varepsilon}
\frac{|u|^2|\nabla\eta_\varepsilon|^2}{|x|^{2b}}\,dx
\le
C\varepsilon^{-2}
\int_\varepsilon^{2\varepsilon}r^{n-1+2m-2b}\,dr
\le
C\varepsilon^{n+2m-2b-2}
\to0.
\]
If \(u\) is flat at \(0\), then the same estimate holds with any fixed
\(m\). Choosing \(m\) sufficiently large so that
\(n+2m-2b-2>0\), the conclusion follows as well.

Thus \(\|u_\varepsilon-u\|_{a,b}\to0\). Hence every
\(u\in\mathcal A\) belongs to the \(\|\cdot\|_{a,b}\)-closure of
\(C_c^\infty(\mathbb R^n\setminus\{0\})\), and the two completions are equal.
\end{proof}

\begin{lemma}[Extension of the Poincaré inequality]\label{lem:extension-poincare}
Let \(\delta>0\) and define $d\mu(x)=e^{-\delta|x|^\tau}|x|^\beta\,dx,
$ $d\nu(x)=e^{-\delta|x|^\tau}|x|^{\beta+\tau-2}\,dx,$ and $\|u\|_W^2:=\int_{\mathbb R^n}|\nabla u|^2\,d\mu+\int_{\mathbb R^n}|u|^2\,d\nu.$
Let $\mathcal A_\tau := \left\{ u\in C_c^\infty(\mathbb R^n):\|u\|_W<\infty \right\},$ and define $ W:= \overline{\mathcal A_\tau}^{\,\|\cdot\|_W}.$
Assume that
$$\frac{(\beta+n+\tau-2)}{\tau}>0.$$
Suppose that $\int_{\mathbb R^n}|w-w_\nu|^2\,d\nu \leq C\int_{\mathbb R^n}|\nabla w|^2\,d\mu$ holds for every $w\in C_c^\infty(\mathbb R^n\setminus\{0\}),$
where
\[
w_\nu
:=
\frac{1}{\nu(\mathbb R^n)}
\int_{\mathbb R^n}w\,d\nu.
\]
Then the same inequality holds for every \(u\in\mathcal A_\tau\),
and consequently for every \(u\in W\).
\end{lemma}

\begin{proof}
The assumption $\tau(\beta+n+\tau-2)>0$ ensures that $0<\nu(\mathbb R^n)<\infty.$ Indeed, if \(\tau>0\), the exponential controls the behavior at
infinity and the condition
\(\beta+n+\tau-2>0\) guarantees integrability near the origin.
If \(\tau<0\), the exponential controls the behavior near the origin,
whereas the condition
\(\beta+n+\tau-2<0\) guarantees integrability at infinity.
We consider the cases \(\tau>0\) and \(\tau<0\) separately.

\medskip

\noindent
\textbf{Case 1: \(\tau>0\).} Let \(u\in\mathcal A_\tau\), and set $v:=u-u(0).$
Then $v(0)=0,$ $\nabla v=\nabla u.$
Moreover, since \(u\in L^2(d\nu)\) and
\(\nu(\mathbb R^n)<\infty\), we have
\[
v\in L^2(d\nu).
\]

Choose \(\eta_\varepsilon\in C^\infty(\mathbb R^n)\) such that
\[
0\leq\eta_\varepsilon\leq1,
\qquad
\eta_\varepsilon=0
\quad\text{on }B_\varepsilon,
\qquad
\eta_\varepsilon=1
\quad\text{on }\mathbb R^n\setminus B_{2\varepsilon},
\]
and
\[
|\nabla\eta_\varepsilon|
\leq\frac{C}{\varepsilon}.
\]
Set
\[
v_\varepsilon:=\eta_\varepsilon v.
\]

We first show that
\[
\|v_\varepsilon-v\|_W\longrightarrow0
\qquad\text{as }\varepsilon\to0.
\]
For the lower-order term,
\[
\int_{\mathbb R^n}|v_\varepsilon-v|^2\,d\nu
\leq
\int_{B_{2\varepsilon}}|v|^2\,d\nu
\longrightarrow0
\]
by the absolute continuity of the integral.

Furthermore,
\[
\nabla(v_\varepsilon-v)
=
(\eta_\varepsilon-1)\nabla v
+
v\nabla\eta_\varepsilon.
\]
Hence
\[
\begin{aligned}
\int_{\mathbb R^n}
|\nabla(v_\varepsilon-v)|^2\,d\mu
&\leq
2\int_{B_{2\varepsilon}}|\nabla v|^2\,d\mu+
2\int_{B_{2\varepsilon}\setminus B_\varepsilon}
|v|^2|\nabla\eta_\varepsilon|^2\,d\mu.
\end{aligned}
\]
The first term tends to zero by the absolute continuity of the
integral.

Suppose first that \(v\) is not flat at the origin, and let \(m\geq1\)
be the order of its first nonzero Taylor term. Then
\[
|v(x)|\leq C|x|^m
\]
near the origin. Since
\[
\int_{\mathbb R^n}|\nabla v|^2\,d\mu<\infty
\]
and
\[
e^{-\delta|x|^\tau}\sim1
\qquad\text{as }x\to0,
\]
Since the leading Taylor polynomial may have vanishing gradient in some directions, we restrict to a suitable cone. Let \(P_m\) denote the first nonzero homogeneous term in the Taylor expansion of \(v\) at the origin. Since \(P_m\) is nonconstant, there exist an open cone \(\Gamma\subset\mathbb R^n\) and constants \(c,r_0>0\) such that \[ |\nabla v(x)|\ge c|x|^{m-1} \qquad\text{for }x\in\Gamma,\quad 0<|x|<r_0.
\]
Hence,
\[
\infty>
\int_{\Gamma\cap B_{r_0}}|\nabla v|^2\,d\mu
\geq
c\int_0^{r_0}r^{\beta+n+2m-3}\,dr,
\]
which implies
\[
\beta+n+2m-2>0.
\]

Consequently,
\[
\begin{aligned}
&
\int_{B_{2\varepsilon}\setminus B_\varepsilon}
|v|^2|\nabla\eta_\varepsilon|^2\,d\mu
\\
&\leq
\frac{C}{\varepsilon^2}
\int_{\varepsilon<|x|<2\varepsilon}
|x|^{2m+\beta}\,dx
\\
&\leq
C\varepsilon^{\beta+n+2m-2}
\longrightarrow0.
\end{aligned}
\]
If \(v\) is flat at the origin, then for every \(m\geq1\),
\[
|v(x)|\leq C_m|x|^m
\]
near \(0\). Choosing \(m\) sufficiently large gives the same
conclusion. Therefore,
\[
\|v_\varepsilon-v\|_W\longrightarrow0.
\]

Since \(v=u-u(0)\) is not necessarily compactly supported, we also
introduce an outer cutoff. Choose
\(\chi_R\in C_c^\infty(\mathbb R^n)\) satisfying
\[
0\leq\chi_R\leq1,
\qquad
\chi_R=1\quad\text{on }B_R,
\qquad
\chi_R=0\quad\text{outside }B_{2R},
\]
and
\[
|\nabla\chi_R|\leq\frac{C}{R}.
\]
Set
\[
v_{\varepsilon,R}
:=
\chi_Rv_\varepsilon.
\]
Then
\[
v_{\varepsilon,R}
\in C_c^\infty(\mathbb R^n\setminus\{0\}).
\]

For every fixed \(\varepsilon>0\),
\[
\|v_{\varepsilon,R}-v_\varepsilon\|_W
\longrightarrow0
\qquad\text{as }R\to\infty.
\]
The lower-order term tends to zero because
\(v_\varepsilon\in L^2(d\nu)\). The gradient term containing
\((1-\chi_R)\nabla v_\varepsilon\) tends to zero by absolute
continuity. Finally,
\[
\begin{aligned}
\int_{\mathbb R^n}
|v_\varepsilon|^2|\nabla\chi_R|^2\,d\mu
&\leq
\frac{C}{R^2}
\int_{R<|x|<2R}
e^{-\delta|x|^\tau}|x|^\beta\,dx
\\
&\longrightarrow0,
\end{aligned}
\]
because \(\tau>0\).

Applying the assumed inequality to \(v_{\varepsilon,R}\), and then
letting \(R\to\infty\) and \(\varepsilon\to0\), gives
\[
\int_{\mathbb R^n}|v-v_\nu|^2\,d\nu
\leq
C\int_{\mathbb R^n}|\nabla v|^2\,d\mu.
\]
Since
\[
v_\nu=u_\nu-u(0),
\]
we have
\[
v-v_\nu=u-u_\nu,
\qquad
\nabla v=\nabla u.
\]
Thus
\[
\int_{\mathbb R^n}|u-u_\nu|^2\,d\nu
\leq
C\int_{\mathbb R^n}|\nabla u|^2\,d\mu.
\]

\medskip

\noindent
\textbf{Case 2: \(\tau<0\).}

Let \(u\in\mathcal A_\tau\). In this case no subtraction of \(u(0)\)
is needed, because the factor
\[
e^{-\delta|x|^\tau}
=
e^{-\delta/|x|^{|\tau|}}
\]
decays faster than every power near the origin.

Choose \(\eta_\varepsilon\) as above and set
\[
u_\varepsilon:=\eta_\varepsilon u.
\]
Since \(u\) is compactly supported,
\[
u_\varepsilon
\in C_c^\infty(\mathbb R^n\setminus\{0\}).
\]

For the lower-order term,
\[
\int_{\mathbb R^n}|u_\varepsilon-u|^2\,d\nu
\leq
\int_{B_{2\varepsilon}}|u|^2\,d\nu
\longrightarrow0.
\]

For the gradient term,
\[
\nabla(u_\varepsilon-u)
=
(\eta_\varepsilon-1)\nabla u
+
u\nabla\eta_\varepsilon.
\]
Therefore,
\[
\begin{aligned}
\int_{\mathbb R^n}
|\nabla(u_\varepsilon-u)|^2\,d\mu
&\leq
2\int_{B_{2\varepsilon}}|\nabla u|^2\,d\mu
\\
&\quad+
2\int_{B_{2\varepsilon}\setminus B_\varepsilon}
|u|^2|\nabla\eta_\varepsilon|^2\,d\mu.
\end{aligned}
\]
The first term tends to zero by absolute continuity. Since \(u\) is
bounded,
\[
\begin{aligned}
&
\int_{B_{2\varepsilon}\setminus B_\varepsilon}
|u|^2|\nabla\eta_\varepsilon|^2\,d\mu
\\
&\leq
\frac{C}{\varepsilon^2}
\int_\varepsilon^{2\varepsilon}
e^{-\delta r^\tau}r^{\beta+n-1}\,dr.
\end{aligned}
\]
Because \(\tau<0\),
\[
e^{-\delta r^\tau}
=
e^{-\delta r^{-|\tau|}}.
\]
For every \(M>0\), there exists \(C_M>0\) such that
\[
e^{-\delta r^{-|\tau|}}
\leq C_M r^M,
\qquad 0<r<1.
\]
Hence
\[
\begin{aligned}
&
\frac{1}{\varepsilon^2}
\int_\varepsilon^{2\varepsilon}
e^{-\delta r^\tau}r^{\beta+n-1}\,dr
\\
&\leq
C_M\varepsilon^{-2}
\int_\varepsilon^{2\varepsilon}
r^{\beta+n+M-1}\,dr
\\
&\leq
C_M\varepsilon^{\beta+n+M-2}.
\end{aligned}
\]
Choosing \(M\) sufficiently large gives
\[
\beta+n+M-2>0,
\]
and therefore
\[
\|u_\varepsilon-u\|_W\longrightarrow0.
\]
Applying the assumed inequality to \(u_\varepsilon\) and passing to the
limit gives
\[
\int_{\mathbb R^n}|u-u_\nu|^2\,d\nu
\leq
C\int_{\mathbb R^n}|\nabla u|^2\,d\mu.
\]
Thus the inequality holds for every \(u\in\mathcal A_\tau\) in both
cases.\\

Finally, since  \(\mathcal A_\tau\) is dense in \(W\), let \(u\in W\) and choose
\(u_j\in\mathcal A_\tau\) such that \(u_j\to u\) in \(W\). Then
\[
u_j\to u \quad\text{in }L^2(d\nu),
\qquad
\nabla u_j\to\nabla_Wu \quad\text{in }L^2(d\mu).
\]
Since \(0<\nu(\mathbb R^n)<\infty\), the averaging map
\(f\mapsto f_\nu\) is continuous on \(L^2(d\nu)\); hence
\[
u_j-(u_j)_\nu\to u-u_\nu
\quad\text{in }L^2(d\nu).
\]
Passing to the limit in the inequality for \(u_j\) yields
\[
\int_{\mathbb R^n}|u-u_\nu|^2\,d\nu
\le
C\int_{\mathbb R^n}|\nabla_Wu|^2\,d\mu.
\]
Thus the inequality extends to every \(u\in W\).

\end{proof}

We conclude with the following remark that transplants the above statements to the orthant, which is
the form in which they are used in the body of the paper.

\begin{remark}\label{rmk:density_orthant}
The results of this appendix are stated on $\mathbb R^{n}$, but the versions we use are
those on the orthant, that is, with $\mathbb R^{n}$ replaced by $\Rnkp$ and
$C_c^\infty(\mathbb R^{n}\setminus\{0\})$ by $C_c^\infty(\Rnkp\setminus\{0\})$. No new
argument is required. The walls $\{x_i=0\}$, $i\ge n-k+1$, present no difficulty, since by
definition $C_c^\infty(\Rnkp)$ consists of functions already vanishing to infinite order
there, so that the only set that has to be removed is a neighbourhood of the origin; and the
cutoff estimates above involve only the radial variable $|x|$, which is unaffected by the
restriction to the orthant. Concretely, if $\eta_\varepsilon$ is the radial cutoff used in
the proofs above and $u\in C_c^\infty(\Rnkp)$, then $\eta_\varepsilon u$ belongs to
$C_c^\infty(\Rnkp\setminus\{0\})$ and the same computations, carried out over $\Rnkp$
instead of $\mathbb R^{n}$, give $\|\eta_\varepsilon u-u\|_X\to0$. The same remark applies
verbatim to the space $S_{a,b}(\Rnkp)$ and to the monomial-weight spaces of Section
\ref{S:poincare}.
\end{remark}

\end{document}